\documentclass[12pt, reqno]{amsart}
\usepackage{amsmath, amssymb, amscd, amsthm, comment, amsxtra}
\usepackage{graphicx}
\usepackage{subfigure}
\usepackage{enumerate}
\usepackage[pdftex, linktocpage=true]{hyperref}
\usepackage{euscript}
\usepackage[format=plain,labelfont=bf,up,width=.99\textwidth]{caption}
\usepackage[toc,page]{appendix}
\usepackage{xcolor}
\usepackage{soul}

\theoremstyle{plain}
\newtheorem{thm}{Theorem}[section]
\newtheorem{lem}[thm]{Lemma}
\newtheorem{prop}[thm]{Proposition}
\newtheorem{cor}[thm]{Corollary}

\newtheorem*{thmmain}{Main Theorem}
\newtheorem*{question*}{Question}

\theoremstyle{definition}
\newtheorem{defn}[thm]{Definition}

\newtheorem{rem}[thm]{Remark}

\newtheorem*{problem*}{Open Problem}

\numberwithin{equation}{section}
\newcommand{\thmref}[1]{Theorem~\ref{#1}}
\newcommand{\propref}[1]{Proposition~\ref{#1}}
\newcommand{\lemref}[1]{Lemma~\ref{#1}}
\newcommand{\corref}[1]{Corollary~\ref{#1}}
\newcommand{\figref}[1]{Figure~\ref{#1}}
\newcommand{\secref}[1]{Section~\ref{#1}}
\newcommand{\subsecref}[1]{Subsection~\ref{#1}}
\newcommand{\appref}[1]{Appendix~\ref{#1}}
\newcommand{\remref}[1]{Remark~\ref{#1}}

\newcommand{\defnref}[1]{Definition~\ref{#1}}

\newcommand{\md}[1]{\;(\operatorname{mod}\; #1)}

\renewcommand{\epsilon}{\varepsilon}

\newcommand{\bbB}{\mathbb B}
\newcommand{\bbC}{\mathbb C}
\newcommand{\bbD}{\mathbb D}

\newcommand{\bbI}{\mathbb I}

\newcommand{\bbL}{\mathbb L}

\newcommand{\bbN}{\mathbb N}
\newcommand{\bbO}{\mathbb O}
\newcommand{\bbP}{\mathbb P}

\newcommand{\bbR}{\mathbb R}
\newcommand{\bbS}{\mathbb S}

\newcommand{\bbZ}{\mathbb Z}

\newcommand{\cB}{\mathcal B}
\newcommand{\cC}{\mathcal C}
\newcommand{\cD}{\mathcal D}
\newcommand{\cE}{\mathcal E}

\newcommand{\cG}{\mathcal G}

\newcommand{\cI}{\mathcal I}

\newcommand{\cK}{\mathcal K}
\newcommand{\cL}{\mathcal L}

\newcommand{\cP}{\mathcal P}

\newcommand{\cR}{\mathcal R}

\newcommand{\cT}{\mathcal T}
\newcommand{\cU}{\mathcal U}
\newcommand{\cV}{\mathcal V}

\newcommand{\cX}{\mathcal X}
\newcommand{\cY}{\mathcal Y}
\newcommand{\cZ}{\mathcal Z}

\newcommand{\bfT}{\mathbf T}
\newcommand{\bfU}{{\mathbf U}}

\newcommand{\bfp}{\mathbf p}

\newcommand{\bft}{\mathbf t}

\newcommand{\hE}{\hat E}
\newcommand{\hF}{\hat F}

\newcommand{\hJ}{\hat J}
\newcommand{\hK}{\hat K}

\newcommand{\hM}{\hat M}
\newcommand{\hN}{\hat N}

\newcommand{\hQ}{\hat Q}

\newcommand{\hT}{\hat T}
\newcommand{\hU}{\hat U}

\newcommand{\hW}{\hat W}

\newcommand{\hZ}{\hat Z}

\newcommand{\hd}{\hat d}
\newcommand{\he}{\hat e}

\newcommand{\tiE}{\tilde E}

\newcommand{\tiM}{\tilde M}
\newcommand{\tiN}{\tilde N}

\newcommand{\tiQ}{\tilde Q}

\newcommand{\tiU}{\tilde U}

\newcommand{\tiW}{\tilde W}

\newcommand{\tiZ}{\tilde Z}

\newcommand{\tic}{\tilde c}

\newcommand{\til}{\tilde l}
\newcommand{\tim}{\tilde m}

\newcommand{\tip}{\tilde p}

\newcommand{\tix}{\tilde x}
\newcommand{\tiy}{\tilde y}

\newcommand{\hdelta}{\hat \delta}

\newcommand{\bepsilon}{{\bar \epsilon}}

\newcommand{\bL}{{\bar L}}

\newcommand{\hGamma}{\hat \Gamma}

\newcommand{\hLambda}{\hat \Lambda}

\newcommand{\hPhi}{\hat \Phi}

\newcommand{\chPhi}{\check \Phi}
\newcommand{\chU}{\check U}

\newcommand{\tiPhi}{\tilde \Phi}

\newcommand{\tiOmega}{\tilde \Omega}

\newcommand{\ticU}{\tilde \cU}
\newcommand{\bcB}{\bar \cB}

\newcommand{\bdelta}{{\bar \delta}}

\newcommand{\bK}{{\bar K}}

\newcommand{\udelta}{\underline{\delta}}
\newcommand{\uepsilon}{\underline{\epsilon}}

\newcommand{\bkappa}{\bar{\kappa}}
\newcommand{\ukappa}{\underline{\kappa}}

\newcommand{\baeta}{\bar\eta}
\newcommand{\hcT}{{\hat\cT}}
\newcommand{\chcT}{{\check\cT}}

\newcommand{\chcU}{\check{\cU}}
\newcommand{\ticV}{\tilde{\cV}}

\newcommand{\hcD}{\hat{\cD}}

\newcommand{\hcU}{\hat\cU}

\newcommand{\ut}{\underline{t}}
\newcommand{\hpartial}{\hat\partial}
\newcommand{\bpartial}{\bar\partial}

\newcommand{\bLp}{\overline{L'}}

\newcommand{\bPsi}{\bar\Psi}

\newcommand{\Jac}{\operatorname{Jac}}
\newcommand{\diam}{\operatorname{diam}}

\newcommand{\dist}{\operatorname{dist}}

\newcommand{\Id}{\operatorname{Id}}
\newcommand{\loc}{\operatorname{loc}}

\newcommand{\depth}{\operatorname{depth}}

\newcommand{\logl}{\log_\lambda}
\newcommand{\radius}{\operatorname{radius}}

\newcommand{\cd}{\operatorname{cd}}

\newcommand{\vd}{\operatorname{vd}}

\newcommand{\de}{\operatorname{d}}
\newcommand{\ee}{\operatorname{e}}
\newcommand{\Ly}{\operatorname{Ly}}
\newcommand{\entry}{\operatorname{entry}}
\newcommand{\Dis}{\operatorname{Dis}}

\newcommand{\matsp}[1]{\hspace{5mm} \text{#1} \hspace{5mm}}

\newcommand{\comma}{, \hspace{5mm}}

\newcommand{\frU}{{\mathfrak{U}}}

\newcommand{\frH}{{\mathfrak{H}}}

\newcommand{\frHL}{{\mathfrak{H}\mathfrak{L}}}

\title[Renormalization of Unicritical Diffeomorphisms of the Disk]{Renormalization of Unicritical Diffeomorphisms\\ of the Disk}
\author{Sylvain Crovisier,
Mikhail Lyubich,
Enrique Pujals,
Jonguk Yang}
\thanks{The first author was partially supported by the ERC project 692925 -- NUHGD. The second author was partly supported by the NSF, the Hagler and Clay Fellowships, the Institute for Theoretical Studies at ETH (Zurich), and SLMSI (formerly MSRI Berkeley). The fourth author was partially supported by SLMSI, Institut Mittag-Leffler, and the thematic programs `Topological, smooth and holomorphic dynamics, ergodic theory, fractals’ of the Simons Foundation Award No. 663281 for IMPAN}

\begin{document}

\maketitle

\begin{abstract}
We introduce a class of infinitely renormalizable  {\it unicritical} diffeomorphisms of the disk (with a non-degenerate ``critical point"). In this class of dynamical systems, we show that under renormalization, maps eventually become H\'enon-like, and then converge super-exponentially fast to the space of one-dimensional unimodal maps. We also completely characterize the local geometry of every stable and center manifolds that exist in these systems. The theory is based upon a quantitative reformulation of the Oseledets-Pesin theory yielding a {\it unicritical structure}  of the maps in question comprising regular Pesin boxes co-existing with ``critical tunnels''  and ``valuable crescents''. In
forthcoming notes we will show that infinitely renormalizable perturbative H\'enon-like maps of bounded type belong to our class.
\end{abstract}
\tableofcontents


\section{Introduction}\label{sec:intro}

Critical points play a fundamental role in one-dimensional dynamics as a natural source of non-linearity. They allow even very simple maps to have extremely rich and fascinatingly complicated dynamical behaviors.

The simplest setting for studying the dynamical properties of critical points is the class of unimodal maps. A {\it unimodal map} is a $C^2$-map $f : I \to I$ on an interval $I \subset \bbR$ with a unique critical point $c_0 \in I$. In this paper, we always assume that $c_0$ is of {\it quadratic type}: $f'(c_0) = 0$ and $f''(c_0) \neq 0$. We say that $f$ is {\it normalized} if $c_0 = 0$ and $f(c_0) = 1$. Let $r \geq 2$ be an integer. The space of normalized $C^r$-unimodal maps is denoted $\mathfrak{U}^r$.

The model examples of unimodal maps are given by the {\it quadratic family}:
$$
\mathfrak{Q} := \{f_a(x) := x^2 +a \; | \; a \in \bbR\}.
$$
The study of this family has led to some incredibly deep mathematics that spans over two decades of research by numerous authors. At the heart of this topic lies the celebrated renormalization theory of unimodal maps completed by Sullivan, McMullen, Lyubich and Avila.

There have been several works that extend this theory to a higher dimensional setting (see e.g. \cite{CoEK}, \cite{GvST} and \cite{dCLM}). However, the limitation of these approaches is that they are all based on perturbative methods (relying on the robustness of the hyperbolic one-dimensional renormalization operator). Hence, they only apply to systems that are sufficiently close to the space of one-dimensional unimodal maps.

   The main goal of this paper is to develop a non-perturbative renormalization theory of ``unicritical'' diffeomorphisms in dimension two. The most immediate challenge is to identify what being ``unicritical'' should mean in this context. Drawing motivations from the one-dimensional and the perturbative two-dimensional cases, we define a {\it critical point} as a point whose local stable and center manifolds are well-defined and form a quadratic tangency. Clearly, any invariant set that feature such a point cannot be uniformly partially hyperbolic. However, if the failure of uniform regularity is restricted to a slow exponentially shrinking neighborhood of the backward orbit of the critical point, then we say that the map is {\it unicritical}. One may think of this condition as a relaxation of uniform partial hyperbolicity that allows for the existence of a tangency between stable and center manifolds. The main result of this paper is that renormalizations of two-dimensional unicritical diffeomorphisms converge super-exponentially fast to the space of one-dimensional unimodal maps.

\subsection{H\'enon-like maps}

The two-dimensional analog of the quadratic family is given by the {\it H\'enon family}:
$$
\mathfrak{H} := \{F_{a,b}(x,y) := (x^2+a - by, x) \; | \; a,b\in\bbR\}.
$$
We identify the degenerate H\'enon map $F_{a,0}$ with $f_a$. 
This way, $\frH$ can be viewed as a two-dimensional extension of the quadratic family.

As in the one-dimensional case, we would like to consider H\'enon maps as the model example of a more general class of maps. The simplest and the most obvious generalization is given by H\'enon-like maps. Let $V = I \times I \subset \bbR^2$. A $C^2$-diffeomorphism $F : V \to F(V) \Subset V$ is {\it H\'enon-like} if it is of the form
$$
F(x,y) = (f(x,y), x)
\matsp{for}
(x,y) \in V,
$$
where for all $y \in I$, the map $x\mapsto f(x,y)$
is unimodal.
We say that $F$ is a {\it normalized} H\'enon-like map if $0\in V$, and $f(\cdot, 0)$ is a normalized unimodal map. The space of normalized $C^r$-H\'enon-like maps is denoted $\frHL^r$.

For $\beta \in (0,1]$, we say that $F$ is {\it $\beta$-thin (in $C^r$)} if
$$
\|\Jac F\|=\|\partial_y f\|_{C^{r-1}} < \beta.
$$
The space of $\beta$-thin maps in $\frHL^r$ is denoted $\frHL^r_\beta$. We say that $F\in\frHL^r_\beta$ is a {\it perturbative H\'enon-like map} if $\beta \ll 1$.

For any 1D map $f : I \to I$, define a degenerate 2D map $\iota(f) : V \to V$ by
\begin{equation}\label{eq:embed}
\iota(f)(x,y) := (f(x), x).
\end{equation}
We denote $\frHL^r_0 := \iota(\frU^r)$.

\subsection{Renormalizability}

Let $F : \Omega \to F(\Omega) \Subset \Omega$ be a dissipative $C^r$-diffeomorphism defined on a Jordan domain $\Omega \Subset \bbR^2$ (i.e. $\Omega$ is a topological disk whose boundary is a topological circle). We say that $F$ is {\it renormalizable} if there exists an $R$-periodic Jordan subdomain $\cD \Subset \Omega$ for some integer $R > 1$:
$$
F^i(\cD) \cap \cD = \varnothing
\matsp{for}
1\leq i < R,
\matsp{and}
F^R(\cD) \Subset \cD.
$$

The map $F$ is {\it infinitely renormalizable} if there exist a nested sequence of of Jordan domains and an increasing sequence of natural numbers
$$
\Omega =: \cD^0 \Supset \cD^1 \Supset \ldots
\matsp{and}
1 =: R_0 < R_1 < \ldots
$$
such that for $n \in \bbN$, the domain $\cD^n$ is $R_n$-periodic, and
$$
r_{n-1} := R_n/R_{n-1} \geq 2.
$$
We refer to $\cD^n$ as the {\it $n$th renormalization domain}. The primary set of interest is the {\it renormalization limit set}, defined as
$$
\Lambda_F := \bigcap_{n = 0}^\infty \bigcup_{i=0}^{R_n-1} F^i(\cD^n).
$$

The {\it $n$th pre-renormalization of $F$} is the first return map to the domain $\cD^n$:
$$
p\cR^n(F) := F^{R_n}|_{\cD^n}.
$$
Let $\Psi^n : \cD^n \to \Omega^n$ be an appropriately defined ``rescaling map'' (non-linear in general) that turns $\cD^n$ into a domain $\Omega^n \subset \bbR^2$ of some ``standard size.'' 
The {\it $n$th renormalization of $F$} is then:
$$
\cR^n(F) := \Psi^n \circ p\cR^n(F) \circ (\Psi^n)^{-1}.
$$
Since $F$ is assumed to be dissipative:
$
\|\Jac F\| \leq \lambda < 1,
$
a heuristic computation shows that
$$
\|\Jac \cR^n(F)\| \approx \|\Jac F^{R_n}\| \leq \lambda^{R_n} \rightarrow 0
\matsp{as}
n \to \infty.
$$
Thus, one expects that under renormalization, the dissipative two-dimensional map $F$ becomes more and more one-dimensional. Our main goal is to rigorously establish a sufficiently stronger version of this statement for the class of maps that we will introduce below.

\subsection{Mild dissipativity}

The map $F$ is {\it mildly dissipative} if for any ergodic measure $\mu$ (not supported on a hyperbolic sink), the strong-stable manifold $W^{ss}(p)$ of $\mu$-almost every point $p \in \Omega$ is {\it proper} in the following sense: if $W_\Omega^{ss}(p)$ is the connected component of $W^{ss}(p) \cap \Omega$ containing $p$, then $\overline{W^{ss}_\Omega(p)}$ is a Jordan arc whose endpoints are contained in $\partial \Omega$. See \figref{fig:milddiss}.

\begin{figure}[h]
\centering
\includegraphics[scale=0.45]{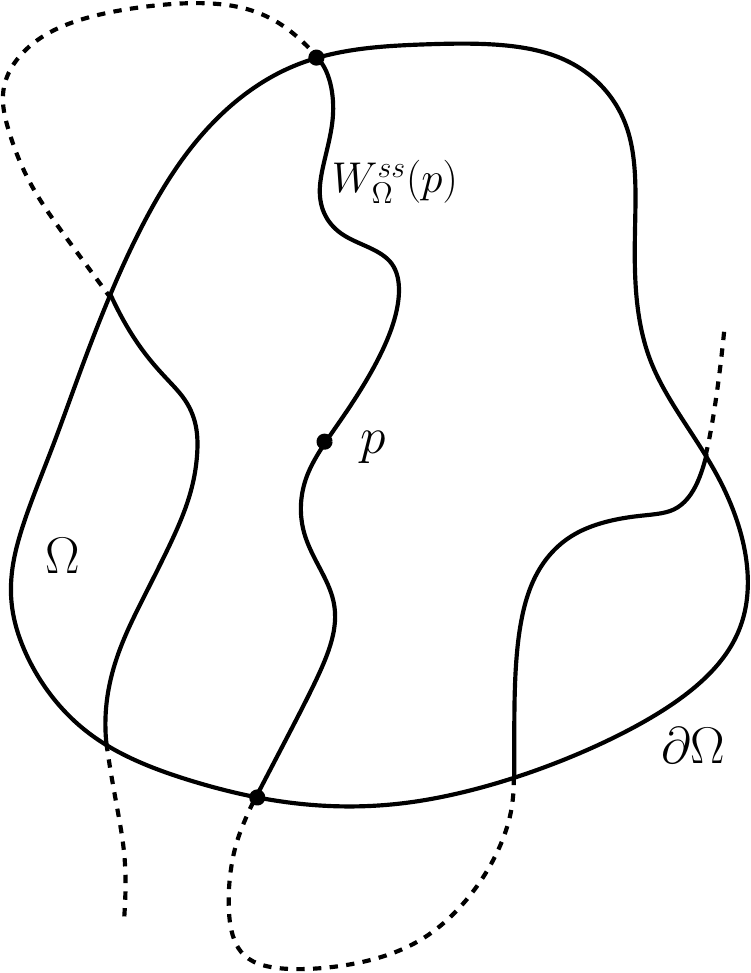}
\caption{Proper strong-stable manifold $W^{ss}(p)$ in $\Omega$.}
\label{fig:milddiss}
\end{figure}

All H\'enon maps $F_{a,b}$ with $|b|<1/4$ 
are known to be mildly dissipative (see \cite{DL} and \cite{CP}). In \cite{CPT}, Crovisier-Pujals-Tresser give an explicit description of the topological structure of mildly dissipative diffeomorphisms of the disk with zero topological entropy. Specifically, they show that any such map must either be infinitely renormalizable (of eventually period-doubling type), or is a generalized Morse-Smale system.


\subsection{Unique ergodicity}\label{subsec:uni erg}

Let $F$ be infinitely renormalizable. Suppose that its renormalization limit set $\Lambda_F$ has a unique invariant measure $\mu$ (this is always true if $F$ is mildly dissipative, see \propref{uni erg}). The Lyapunov exponents of $F$ with respect to $\mu$ are $0$ and $\log \lambda_F < 0$ for some $\lambda_F \in (0,1)$ (see \propref{non hyp}). By the Oseledets theorem, the tangent space at $\mu$-a.e. point $p \in \Lambda_F$ splits as $E^{ss}_p\oplus E^c_p$, so that
\begin{equation}\label{eq:for ly reg}
\lim_{n \to +\infty} \frac{1}{n}\log \|DF^n|_{E_p^{ss}}\| = \log\lambda_F
\matsp{and}
\lim_{n \to +\infty} \frac{1}{n}\log \|DF^{-n}|_{E_p^c}\| = 0.
\end{equation}
Since $F$ is $C^r$ with $r > 1$, it follows from~\eqref{eq:for ly reg} and the unique ergodicity that
$p$ has $C^r$-smooth strong-stable and center manifolds $W^{ss}(p)$ and $W^c(p)$ respectively (see \subsecref{stable-center manifolds} and \propref{lya reg}). The former is tangent to $E^{ss}_p$ and the latter to $E^c_p$. The center manifold $W^c(p)$ is not uniquely defined. However, its $C^r$-jet at $p$ is unique (see \propref{center jet}).


\subsection{Regular unicriticality}

Understanding the dynamics of smooth maps of the interval is centered around the analysis of their critical points. Therefore, when studying smooth maps in dimension two, it is natural to try and identify the ``critical sets'' of these maps that play a similar role.

This strategy is implemented in the landmark paper of Benedicks-Carleson \cite{BCa} to prove the existence of strange attractors for a positive measure set of H\'enon maps. In their setting, the critical set is defined as a certain set of tangencies between stable and unstable manifolds obtained as the limit of an inductive procedure. This topic has since been further developed by numerous authors, including Wang-Young \cite{WYo}, Palis-Yoccoz \cite{PaYo} and Berger \cite{Be}.

A much more general (albeit less explicit) way to identify the critical set is to characterize it as the obstruction locus for dominated splitting (see \defnref{def general crit point}). This idea is explored by Pujals-Rodriguez Hertz in \cite{PR}, and Crovisier-Pujals in \cite{CP2}.

In our paper, we are mainly interested in two-dimensional maps that are infinitely renormalizable. It turns out that such maps can be {\it unicritical}---the critical set can consist of a single point. The notion of the critical set we use incorporates aspects from the two aforementioned bodies of work. The novelty of our approach is the emphasis on the regularity of points that are not ``too close'' to the critical orbit. Since the critical orbit is expected to be dense in the renormalization limit set, this property is rather delicate to formulate. The rigorous definitions are given below.

Let $F$ be an infinitely renormalizable map as considered in \subsecref{subsec:uni erg}.

\begin{defn}
The orbit of a point $p\in \Lambda_F$ is {\it regular critical} if $p$ admits a direction
$$
E^*_{p} := E^{ss}_{p} = E^c_{p}
$$
that satisfies both conditions in~\eqref{eq:for ly reg}.
The criticality is {\it quadratic} if, additionally,
$W^{ss}(p)$ and $W^c(p)$ have a quadratic tangency at $p$.
\end{defn}

\begin{defn}
For $N \in \bbN$; $L \geq 1$ and $\epsilon \in (0,1)$, we say that the point $p$ is {\it $N$-times forward $(L, \epsilon)$-regular} if
it admits a direction $E_p$ such that
\begin{equation}\label{eq:intro reg}
\|DF^n|_{E_p}|\| \leq L\lambda_F^{\epsilon n}
\matsp{for each}
1\leq n \leq N.
\end{equation}
\end{defn}



For $t > 0$ and $p \in \bbR^2$, we denote the ball
$$
\bbD_p(t) := \{q \in \bbR^2 \; | \; \dist(q, p) < t\}.
$$

\begin{defn}\label{def unicrit}
For $0<\epsilon < \delta < 1$, we say that $F$ is {\it $(\delta, \epsilon)$-regularly unicritical on the limit set $\Lambda_F$} if the following conditions hold.
\begin{enumerate}[i)]
\item There is a regular quadratic critical orbit point $c_1 \in \Lambda_F$ (referred to as the {\it critical value}).
\item For all $t >0$, there exists $L(t) \geq 1$ such that for any $N \in \bbN$, if
\begin{equation}\label{eq:away from crit}
p \in \Lambda \setminus \bigcup_{n = 0}^{N-1} \bbD_{F^{-n}(c_{1})}(t\lambda_F^{\epsilon n}),
\end{equation}
then $p$ is $N$-times forward $(L(t), \delta)$-regular.
\end{enumerate}
When $\delta$ and $\epsilon$ are implicit, we simply say that $F$ is {\it regularly unicritical on $\Lambda_F$}.
\end{defn}

\begin{rem}
By \propref{measure in crit disk}, the measure of the set of points satisfying \eqref{eq:away from crit} for all $N \in \bbN$ goes to $1$ as $t$ goes to $0$. This shows that the condition imposed by \defnref{def unicrit} ii) is not vacuous. Moreover, by \propref{unique crit}, a regularly unicritical system has a unique critical orbit as the name implies.
\end{rem}

\begin{rem}\label{1d is unicrit}
The analog of \defnref{def unicrit} holds in the 1D case. See \appref{sec:reg unicrit 1d}.
\end{rem}

\begin{rem}\label{ren unicrit is unicrit}
In the 1D case, an important motivation for the introduction of unimodal maps is that a pre-renormalization of a unimodal map is again unimodal. Observe that in the 2D case, regular unicriticality is likewise preserved under pre-renormalization.
\end{rem}

\begin{rem}\label{rem exists}
In \cite{CLPY2}, we will show that any infinitely renormalizable $C^2$-diffeo\-mor\-phisms of bounded type whose renormalizations are eventually contained in a $C^1$-compact subset of $\frHL^2$ is regularly unicritical on its renormalization limit set.
\end{rem}


\subsection{Statements of the main results}

Consider a H\'enon-like map $G : V \to V$ defined on a square $V = I \times I \subset \bbR^2$ containing $0$. Denote the critical point of $g(\cdot) := G(\cdot, 0)$ by $c \in I$. Let $\bbS_G : \bbR^2 \to \bbR^2$ be the affine scaling map given by 
$$
\bbS_G(x,y) := |g(c)-c|^{-1}(x-c, y).
$$
Observe that $\bbS_G \circ G \circ \bbS_G^{-1}$ is normalized.

\begin{thmmain}
There exist universal constants $\epsilon_0, \alpha \in (0,1)$ such that the following holds.
For an integer $r \geq 2$, let $F : \Omega \to F(\Omega) \Subset \Omega\Subset \bbR^2 $ be a mildly dissipative, infinitely renormalizable $C^r$-diffeomorphism with
renormalization limit set
$$
\Lambda_F := \bigcap_{n=1}^\infty \bigcup_{i=0}^{R_n-1} F^i(\cD^n)
$$
Suppose $F$ is $(\epsilon, \epsilon^{1/\alpha})$-regularly unicritical on $\Lambda_F$ for some $\epsilon \in (0, \epsilon_0)$ with the critical value $c_1$. Then for any $n\in\bbN$ sufficiently large, there exists a $R_n$-periodic quadrilateral $\cV^n $ containing $c_1$, and a $C^r$-diffeomorphism
$$
\Psi^n : (\cV^n, c_1) \to (V^n, 0)
$$
from $\cV^n$ to a square $V^n = I^n\times I^n \subset \bbR^2$ such that the following properties hold.
\begin{enumerate}[i)]
\item We have $\Lambda_F \cap \cD^n\Subset \cV^n \Subset \cV^{n-1}$.
\item We have
$$
\|(\Psi^n)^{\pm 1}\|_{C^r} = O(1)
\matsp{and}
\|\Psi^n - \Psi^{n-1}|_{\cV^n}\|_{C^r} < \lambda^{(1-\epsilon^\alpha)R_n}.
$$
\item Define the $n$th pre-renormalization and the renormalization of $F$ as
$$
F_n := \Psi^n \circ F^{R_n}|_{\cV^n} \circ (\Psi^n)^{-1}
\matsp{and}
\cR^n(F) := \bbS_{F_n} \circ F_n \circ \bbS_{F_n}^{-1}
$$
respectively. Then $\cR^n(F)$ is a $\delta_n$-thin $C^r$-H\'enon-like map with $\delta_n < \lambda^{(1-\epsilon^\alpha)R_n}$.
\end{enumerate}
\end{thmmain}

As a prerequisite for the proof of the Main Theorem, we obtain an extremely detailed understanding of the regularity of $F$ at every point in $\Lambda_F$ (not just on unspecified sets of large or full measure). As a result, we are also able to explicitly describe the local geometry of every strong stable and center manifolds that exist in the system. See \thmref{reg blocks} and the discussion in \subsecref{subsec:local geom}.

In a sequel to this paper (\cite{CLPY2}), we prove {\it a priori} bounds for the class of bounded type infinitely renormalizable $C^r$-diffeomorphisms of the disk that satisfy the conclusion of the Main Theorem. As a consequence, we have the following corollaries for the map $F$.

\begin{cor}\label{tot disconnect}
Suppose $F$ is of bounded type. Then the Hausdorff dimension of $\Lambda_F$ is strictly less than $1$. In particular, $\Lambda_F$ is totally disconnected and minimal.
\end{cor}

\begin{cor}\label{renorm conv}
Suppose $F$ is of bounded type and at least $C^4$. Then the sequence of renormalizations of $F$ given in the Main Theorem converge exponentially fast to a compact subset of $\frHL^r_0$ that can be identified with the renormalization horseshoe of 1D unimodal maps.
\end{cor}


\subsection{The boundary of chaos in the H\'enon family}

For $\beta \in (0, 1]$, denote
$$
\frH_\beta := \{F_{a,b} \in \frH \; | \; b \in [0, \beta)\} \quad \text{and} \quad\frH_0 := \{F_{a,0} \in \frH\}.
$$
Denote by $\partial^1\frH_\beta$ the set of all maps in $\frH_\beta$ that have zero topological entropy and can be approximated by H\'enon maps with positive entropy (note that $\partial^1 \frH_\beta \subsetneq \partial \frH_\beta$). We refer to $\partial^1\frH_\beta$ as the {\it boundary of chaos in} $\frH_\beta$. If $F_{a,b} \in \partial^1\frH_{1/4}$, then $F_{a,b}$ is mildly dissipative and has zero entropy. Hence $F_{a,b}$ is infinitely renormalizable by \cite{CPT}.

A priori, the boundary of chaos could be a highly pathological (disconnected) continuum. However, in \cite{dCLM}, de Carvalho-Lyubich-Martens characterize this set in the perturbative case ($\partial^1 \frH_\delta$ with $\delta \ll 1$) in the following way.

In the quadratic family, the transition from zero to positive entropy happens precisely at the limit $f_{a_*}$ of period-doubling bifurcations. Moreover, under renormalization, $\cR^n(f_{a_*})$ converges to some universal map $f_* \in \frU^r$ as $n\to \infty$. Since $\frH_0$ can be identified as the embedding of the quadratic family under $\iota$ (see \eqref{eq:embed}), we have $\partial^1 \frH_0 = \{F_{a_*, 0} = \iota(f_{a_*})\}$, and $\cR^n(F_{a_*,0})$ converges to some universal degenerate $0$-H\'enon-like map $F_* = \iota(f_*)$ as $n \to \infty$. By the hyperbolicity of renormalization, it follows that there is a real analytic curve
$$
\{F^*_b := F_{a_*(b), b}\; | \; 
b \in [0, \delta)\}
$$
of infinitely renormalizable H\'enon maps of period-doubling type extending from $F^*_0 = F_{a_*, 0}$ such that $\cR^n(F^*_b) \to F_*$ as $n \to \infty$. De Carvalho-Lyubich-Martens show that this curve coincides with $\partial^1\frH_\delta$.



We conclude this subsection with the following motivating open problem.

\begin{problem*}
Characterize the boundary of chaos $\partial^1 \frH_\beta$ for some definite non-perturbative value of $\beta \in (0,1]$.
\end{problem*}

Conjecturally, $\partial^1 \frH^0_1$ consists of infinitely renormalizable H\'enon maps that converge to the period-doubling renormalization fixed point $F_*$ under renormalization. Note that by \remref{rem exists}, all such maps must be regularly unicritical on its renormalization limit set. If the conjeture holds, then one would be able to conclude that $\partial^1 \frH_1$ is a piecewise analytic curve. Towards this open problem, we state the following corollary of the Main Theorem and \corref{renorm conv}.

\begin{cor}
Let $r \geq 4$ be an integer. Consider a $C^r$-diffeomorphism $F : \Omega \to F(\Omega) \Subset \Omega$ defined on a Jordan domain $\Omega \Subset \bbR^2$. Suppose $F$ is mildly dissipative, infinitely renormalizable and regularly unicritical on its renormalization limit set. If $F$ has zero topological entropy, then $F$ is in the boundary of chaos. More precisely, in any neighborhood of $F$ in the $C^r$-topology, there exists a diffeomorphism with positive topological entropy.
\end{cor}

\subsection{Conventions}

Unless otherwise specified, we adopt the following conventions. 

Any diffeomorphism on a domain in $\bbR^2$ is assumed to be orientation-preserving. The projective tangent space at a point $p \in \bbR^2$ is denoted by $\bbP^2_p$.
\medskip

We typically denote constants by $K \geq 1$, $k >0$ (and less frequently $C \geq 1$, $c>0$).
Given a number $\kappa > 0$, we use $\bkappa$ to denote any number that satisfy
$$
\kappa < \bkappa < C\kappa^D
$$
for some uniform constants $C > 1$ and $D > 1$ (if $\kappa >1$) or $D \in (0, 1)$ (if $\kappa < 1$). We allow $\bkappa$ to absorb any uniformly bounded coefficient or power. So for example, if $\bkappa >1$, then we may write
$$
\text{``}\;\; 10\bkappa^5 = \bkappa\;\;\text{''}.
$$
Similarly, we use $\ukappa$ to denote any number that satisfy
$$
c\kappa^d < \ukappa < \kappa
$$
for some uniform constants $c \in (0, 1)$ and $d \in (0, 1)$ (if $\kappa >1$) or $d >1$ (if $\kappa <1$). As before, we allow $\ukappa$ to absorb any uniformly bounded coefficient or power. So for example, if $\ukappa >1$, then we may write
$$
\text{``}\;\;\tfrac{1}{3}\ukappa^{1/4} = \ukappa\;\;\text{''}.
$$
These notations apply to any positive real number: e.g. $\bepsilon > \epsilon$, $\udelta < \delta$, $\bL > L$, etc.

Note that $\bkappa$ can be much larger than $\kappa$ (similarly, $\ukappa$ can be much smaller than $\kappa$). Sometimes, we may avoid the $\bkappa$ or $\ukappa$ notation when indicating numbers that are somewhat or very close to the original value of $\kappa$. For example, if $\kappa \in (0,1)$ is a small number, then we may denote $\kappa':=(1-\bkappa)\kappa$. Then $\ukappa \ll \kappa' < \kappa$.

\medskip

For any set $X_m \subset \Omega$ with a numerical index $m \in \bbZ$, we denote
$$
X_{m+l} := F^l(X_m)
$$
for all $l \in \bbZ$ for which the right-hand side is well-defined. Similarly, for any direction $E_{p_m} \in \bbP^2_{p_m}$ at a point $p_m \in \Omega$, we denote
$$
E_{p_{m+l}} := DF^l(E_{p_m}).
$$

We use $n, m, i, j$ to denote integers (and less frequently $l, k$). Typically (but not always), $n \in\bbN$ and $m \in \bbZ$. We sometimes use $l >0$ for positive geometric quantities (such as length). The letter $i$ is never the imaginary number.

We typically use $N, M$ to indicate fixed integers (often related to variables $n, m$).

We use calligraphic font $\cU, \cT, \cI,$ etc, for objects in the phase space and regular fonts $U, T, I,$ etc, for corresponding objects in the linearized/uniformized coordinates. A notable exception is for the invariant manifolds $W^{ss}, W^c$.

We use $p, q$ to indicate points in the phase space, and $z, w$ for points in linearized/uniformized coordinates.


\section{Presentation of the Main Concepts and Arguments}
We now present dynamical and geometrical features of unicritical diffeomorphisms that are developed in the following sections, and explain key concepts for the proofs.

Let $r \geq 2$ be an integer. Consider a mildly dissipative, infinitely renormalizable $C^r$-diffeomorphism $F : \Omega \to F(\Omega) \Subset \Omega$ of a Jordan domain $\Omega \subset \bbR^2$,
which is $(\epsilon, \uepsilon)$-regularly unicritical on its renormalization limit set
$$
\Lambda := \bigcap_{n=1}^\infty \bigcup_{i=0}^{R_n-1} F^i(\cD^n)
$$
Let $0$ and $\log\lambda < 0$ be the Lyapunov exponents of $F$ with respect to the unique invariant probability measure $\mu$ on $\Lambda$.




\subsection{Regular blocks}
Let $\epsilon \in (0, 1)$ be a constant such that $\bepsilon < 1$. We may assume without loss of generality that $F$ is sufficiently homogeneous (see \secref{sec:homog}).

For $L \geq 1$, we say that $p\in \Lambda$ is {\it (infinitely) forward $(L, \epsilon)$-regular} if it has a strong-stable direction $E^{ss}_p \in \bbP^2_p$ such that
\begin{equation}\label{eq:hom for reg}
\|DF^n|_{E_p^{ss}}\| < L \lambda^{(1-\epsilon) n}
\matsp{for all}
n \in \bbN.
\end{equation}
Similarly, $p$ is {\it (infinitely) backward $(L, \epsilon)$-regular} if it has a center direction $E^c_p \in \bbP^2_p$ such that
\begin{equation}\label{eq:hom back reg}
\|DF^{-n}|_{E_p^c}\| < L \lambda^{-\epsilon n}
\matsp{for all}
n \in \bbN.
\end{equation}
Lastly, we say that $p$ is {\it Pesin $(L, \epsilon)$-regular} if \eqref{eq:hom for reg} and \eqref{eq:hom back reg} hold, and $\measuredangle(E^{ss}_p, E^c_p) > 1/L$. The constants $L$ and $\epsilon$ are referred to as a {\it regularity factor} and a {\it marginal exponent} respectively.

If $p$ is forward/backward $(L, \epsilon)$-regular, then the forward/backward iterates of $F$ can be nearly linearized in a neighborhood $\cU_p$ of $p$, called a {\it regular neighborhood at $p$}, with $\diam(\cU_p) = 1/\bL$ (see \thmref{reg chart}). If $p$ is Pesin $(L, \epsilon)$-regular, then the local linearization can be done for iterates of $F$ going in both directions (see \thmref{reg chart}).

The sets of all points $p \in \Lambda$ such that $p$ is infinitely forward, backward and Pesin $(L, \epsilon)$-regular are referred to as the {\it forward, backward} and {\it Pesin $(L, \epsilon)$-regular blocks}, and are denoted by $\Lambda^+_{L, \epsilon}$, $\Lambda^-_{L, \epsilon}$ and $\Lambda^P_{L, \epsilon}$ respectively.

If $p \in \Lambda^{+/-}_{L, \epsilon}$, then $p$ has a $C^r$ strong-stable/center manifold $W^{ss/c}(p)$. The {\it local strong-stable/center manifold $W^{ss/c}(p)$} at $p$ is defined as the connected component of $W^{ss/c}(p) \cap \cU_p$ containing $p$. The $C^r$-geometry of $W^{ss/c}_{\loc}(p)$ is bounded by $\bL$ (see \propref{ss c geo}). Moreover, $W^{ss/c}_{\loc}(p)$ varies $C^{1+\alpha}$-continuously on $p \in \Lambda^{+/-}_{L, \epsilon}$ (see Propositions \ref{ss c contin}).

By classical Oseledets-Pesin theory, we have $\mu(\Lambda^P_{L, \epsilon}) \to 1$ as $L \to \infty$. However, for the proof of the Main Theorem, we need a much more complete description of the regular blocks of $F$ in $\Lambda$. This is provided by the following theorem, which is proved in \subsecref{sec:refine uni}.

\begin{thm}\label{reg blocks}
Suppose that $F$ is $(\epsilon, \uepsilon)$-regularly unicritical on its renormalization limit set $\Lambda$. Let $\{c_m := F^m(c_0)\}_{m \in \bbZ}$ be the regular critical orbit of $F$. Denote $\epsilon' := (1-\bepsilon)\epsilon$. Then there exists a 
constant $\bft \in (0,1)$ such that for any $t \in (0,\bft]$, we have
\begin{enumerate}[i)]
\item $\Lambda^+_{1/\ut, \epsilon} \supset \displaystyle \Lambda \setminus \bigcup_{n = 0}^\infty \bbD_{c_{-n}}(t\lambda^{\epsilon' n})$,
\item $\Lambda^-_{1/\ut, \epsilon} \supset \displaystyle \Lambda \setminus \bigcup_{n = 1}^\infty \bbD_{c_n}(t\lambda^{\epsilon' n})$, and
\item $\Lambda^P_{1/\ut, \epsilon} \supset \displaystyle \Lambda \setminus \bigcup_{m = -\infty}^\infty \bbD_{c_m}(t\lambda^{\epsilon' |m|})$.
\end{enumerate}
\end{thm}


\subsection{Local geometry of the stable and center manifolds}\label{subsec:local geom}

Using \thmref{reg blocks}, we are able to explicitly describe the local geometry of every strong stable and center manifolds that exist in the system. This description centers around the critical orbit $\{c_m\}_{m\in \bbZ}$, and relies on two key results: \thmref{unif reg crit} and \thmref{pinching}. The first result  is about the local dynamics of $F$ near $\{c_m\}_{m\in\bbZ}$. The second result is about the local geometry of $\Lambda$ near $c_0$.

In the following discussion, we designate the ``vertical'' direction as the direction of strong contraction, and the ``horizontal'' direction as the neutral direction (when such directions are well-defined). We also distinguish $c_0$ and $c_1$ as the {\it critical point} and the {\it critical value} (the choice of $c_0$ within the critical orbit can be done arbitrarily).

Consider a cover of the critical orbit by slow exponentially shrinking disks 
\begin{equation}\label{eq:intro crit disk}
c_m \in \bbD_m := \bbD_{c_m}(\bft\lambda^{\epsilon' |m|})
\matsp{for}
m \in \bbZ,
\end{equation}
where $\bft$ is given in \thmref{reg blocks}. \thmref{unif reg crit} states that under a suitable change of coordinates, the dynamics of $F$ on these disks can be uniformized in the following way. On $\bbD_m$ with $m \neq 0$, iterates of $F$ become $C^1$-close to iterates under the linear map $A(x,y) = (x, \lambda y)$. On $\bbD_0$, an application of $F$ uniformizes to an application of the H\'enon map $H(x,y) = (x^2-\lambda y,x)$. See \figref{fig:critnbh}. This is analogous to the one-dimensional case, where for a unimodal map $f$, iterations away from the critical point is effectively linear, while at the critical point, $f$ can be uniformized to the power map $P(x) = x^2$.

\begin{figure}[h]
\centering
\includegraphics[scale=0.21]{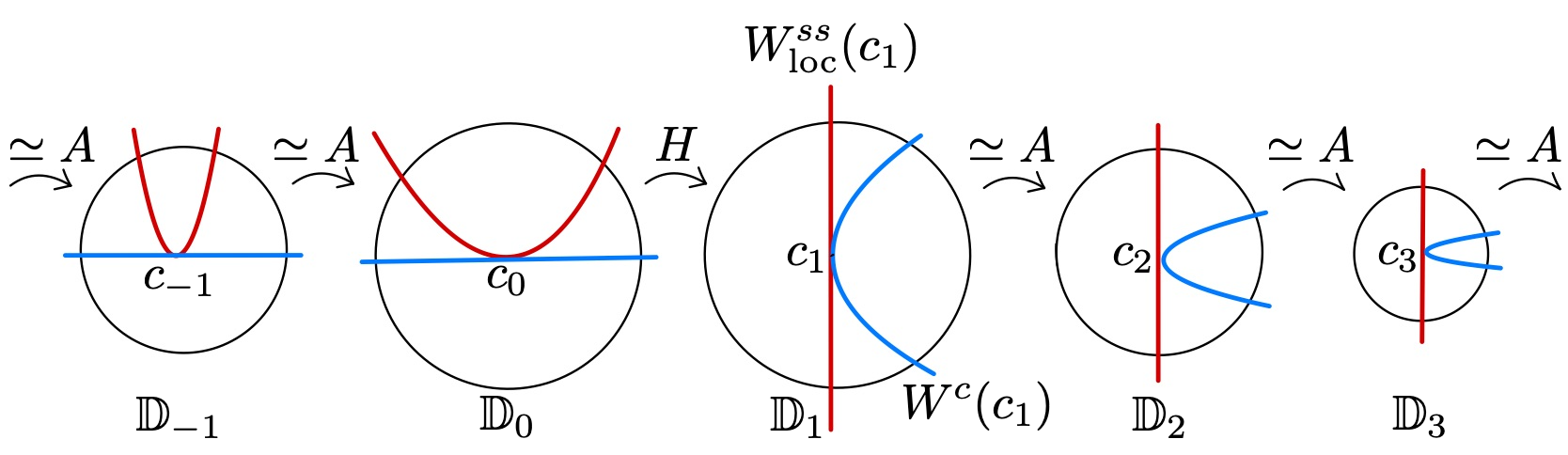}
\caption{Uniformization of the local dynamics of $F$ near the critical orbit $\{c_m\}_{m\in\bbZ}$.}
\label{fig:critnbh}
\end{figure}

In the uniformizing coordinates, the local strong stable manifold $W^{ss}_{\loc}(c_n)$ with $n \geq 1$ is a vertical line segment, while the center manifold $W^c(c_{-m})$ with $m \leq 0$ is a horizontal line segment. At the critical moment, $F$ maps straight horizontal curves in $\bbD_0$ to quadratic vertical curves in $\bbD_1$, and $F^{-1}$ maps straight vertical curves in $\bbD_1$ to quadratic horizontal curves in $\bbD_0$. Observe that the curvature of $W^c(c_n)$ for $n \geq 1$ increases exponentially fast near $c_n$ under forward iterates of $F$. Similarly, the curvature of $W^{ss}(c_{-n})$ for $n \leq 0$ increases exponentially fast near $c_{-n}$ under backward iterates of $F$.

\thmref{pinching} states that $\Lambda$ ``pinches'' at $c_0$. That is, if $\Phi_{c_0} : \bbD_0 \to \Phi_{c_0}(\bbD_0)$ is the change of coordinate on $\bbD_0$ that uniformizes the dynamics of $F$, then
$$
\Phi_{c_0}(\Lambda \cap \bbD_0) \subset T_0 := \{(x,y) \; | \; |y| \leq |x|^{1/\epsilon}\}.
$$
Let $\cT_{c_0} := \Phi_{c_0}^{-1}(T_0)$. It follows that $\Lambda \cap \bbD_m$ is contained in a pinched region $\cT_{c_m}$ for all $m \in \bbZ$, so that for $n \geq 0$, we have
$$
F^n(\cT_{c_{-n}}) \subset \cT_{c_0}
\matsp{and}
F^{-(1+n)}(\cT_{c_{1+n}}) \subset \cT_{c_0}.
$$
We call $\cT_{c_{-n}}$ the {\it critical tunnel at $c_{-n}$}, and $\cT_{c_{1+n}}$ the {\it valuable crescent at $c_{1+n}$}. We also show that for all $m, l \in \bbZ$ with $|m| < |l|$ we have
$$
\cT_{c_m} \cap \cT_{c_l} = \varnothing
\matsp{or}
\cT_{c_m} \supset \cT_{c_l}.
$$
The {\it generation of $\cT_{c_{-n}}$} is the number of critical tunnels that contain $\cT_{c_{-n}}$. Similarly, the {\it generation of $\cT_{c_{1+n}}$} is the number of valuable crescents that contain $\cT_{c_{1+n}}$.

Lastly, we introduce the vertical manifold
$$
W^v(c_0) := \Phi_{c_0}^{-1}(\{(0, y) \; | \; y \in \bbR\})\cap \bbD_0.
$$
For $n \geq 0$, define the {\it vertical manifold $W^v(c_{-n})$} and the {\it horizontal manifold $W^h(c_{1+n})$} as the longest arcs in $\bbD_{-n}$ and $\bbD_{1+n}$ containing $c_{-n}$ and $c_{1+n}$ respectively such that
$$
F^n(W^v(c_{-n})) \subset W^v(c_0)
\matsp{and}
F^{-(1+n)}(W^h(c_{1+n}))\cap \bbD_0 \subset W^v(c_0).
$$
One should consider $W^v(c_{-n})$ and $W^h(c_{1+n})$ as ``stand ins'' for the strong stable and center manifolds for $c_{-n}$ and $c_{1+n}$ respectively. See \figref{fig:tunnel}

\begin{figure}[h]
\centering
\includegraphics[scale=0.27]{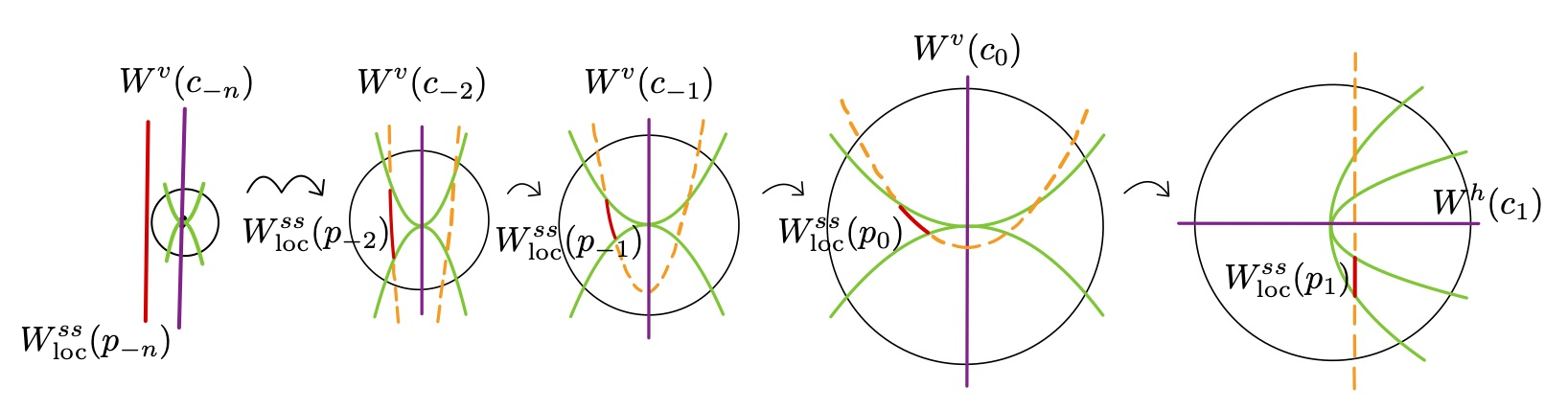}
\caption{Critical tunnels, valuable crescents, and vertical and horizontal manifolds. The local geometry of the strong-stable manifold of $p_{-i} \in \cT_{c_{-i}}$ for $i \geq 0$ is initially distorted near the critical orbit, but becomes vertical again under the backward dynamics of $F$.}
\label{fig:tunnel}
\end{figure}

We are now ready to describe the local geometry of all invariant manifolds that exist in $\Lambda$. We begin with strong stable manifolds. Consider a sufficiently small neighborhood $\cU_{p}$ of a uniformly forward regular point $p$. For $k\geq 0$, let $\bfT_{p}^k$ be the set of all critical tunnels of generation $k$ contained in $\cU_{p}$.

If $q \in \cU_{p}$ is uniformly forward regular, then $W^{ss}_{\loc}(q)$ is a proper straight vertical curve in $\cU_p$. Moreover, if $\cT_{c_{-m}} \in \bfT_p^0$, then $W^v(c_{-m})$ is also a proper straight vertical curve in $\cU_p$. If, on the other hand, $q$ is not uniformly forward regular, then $q \in \cT_{c_{-m}}$ for some $\cT_{c_{-m}} \in \bfT_p^0$. See \figref{fig:tunnelinbox}.

\begin{figure}[h]
\centering
\includegraphics[scale=0.21]{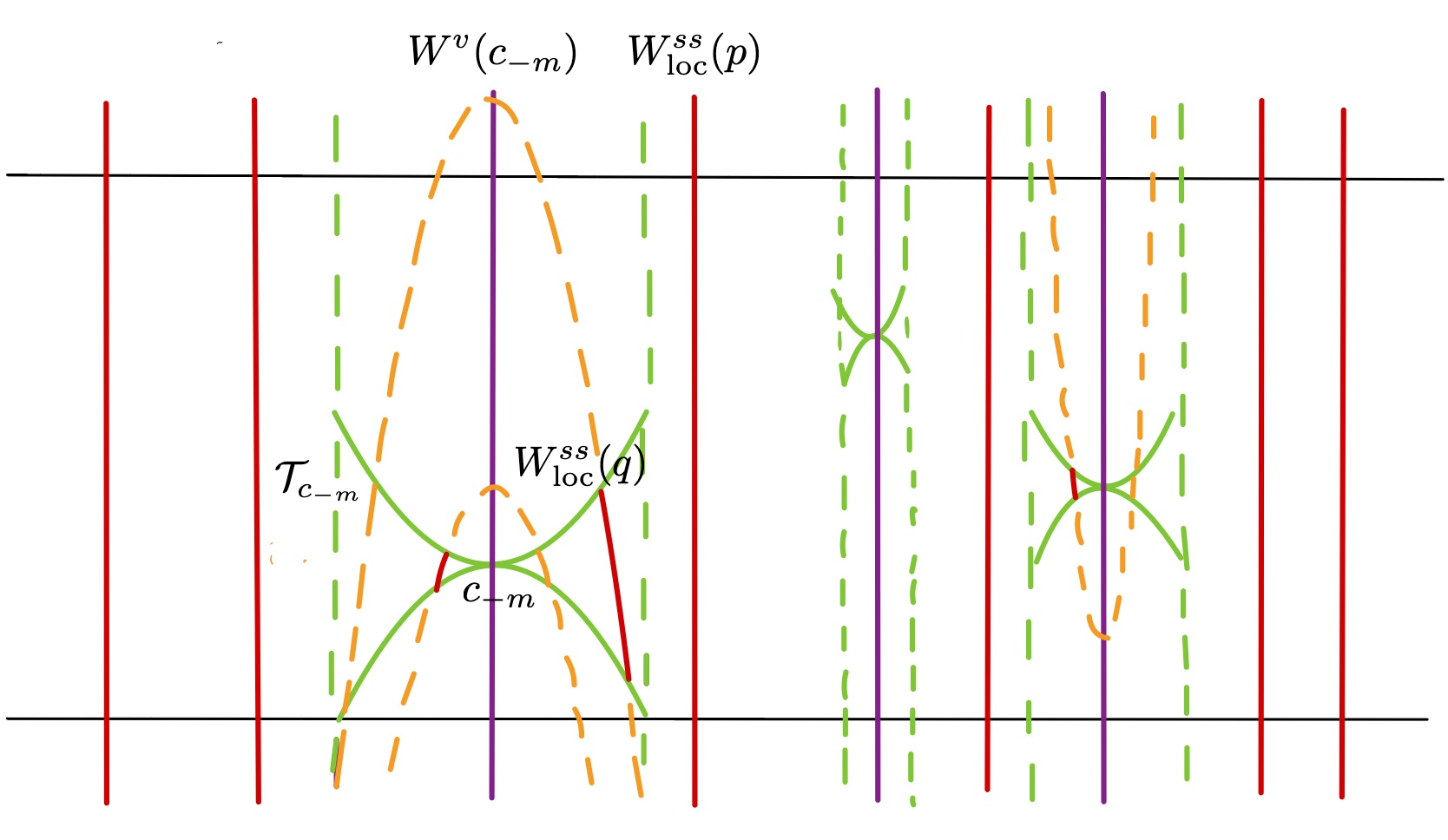}
\caption{Stable lamination in a neighborhood $\cU_p$ of a uniformly forward regular point $p$. If $q \in \cU_p$ is likewise uniformly forward regular, then $W^{ss}_{\loc}(q)$ is proper and vertical in $\cU_p$. If $q$ is not uniformly forward regular, then $q\in \cT_{c_{-m}}$ for some $m\geq 0$. If $\cT_{c_{-m}}$ is of generation $0$, then $W^v(c_{-m})$ is also proper and vertical.}
\label{fig:tunnelinbox}
\end{figure}

To complete our description, we need to understand what local strong stable manifolds look like inside the critical tunnels of generation $0$. Towards this end, let $p = c_1$ in the above description. For $k \geq 0$, applying $F^{-1}$ increases the the generation of critical tunnels in $\bfT_{c_1}^k$ by one:
$$
\bfT_{c_0}^{k+1} = \{\cT_{c_{-m-1}} \; | \; \cT_{c_{-m}} \in \bfT_{c_1}^k\}.
$$
Hence, we see that if $q_0 \in \cU_{c_0}$ is not contained in any $\cT_{c_{-m}} \in \bfT_{c_0}^1$, then $q_1 \in \cU_{c_1}$ is uniformly forward regular, and $W^{ss}_{\loc}(q_1)$ is a proper straight vertical curve. Thus, we conclude that
$$
W^{ss}_{\loc}(q_0) \subset F^{-1}(W^{ss}_{\loc}(q_1)) \cap \cT_{c_0}
$$
is a subarc of a quadratic horizontal curve (proper in $\cT_{c_0}$). See \figref{fig:tunnel}. Pulling back under the dynamics of $F$, we obtain the local stable manifolds of every point contained in a critical tunnel of generation $0$, but not contained in a critical tunnel of generation $1$.

This observation also explains the mechanism by which poor forward regularity can improve under the dynamics. By uniformization, we see that iterating by $F^{-1}$ on $\cT_{c_0}$ is effectively iterating by the linear map $A^{-1}(x,y) = (x, \lambda^{-1}y)$. Under this strong vertical expansion, subarcs of quadratic horizontal curves in $\cT_{c_0}$ become straighter and more vertical. So for $q_0 \in \cT_{c_0}$ not contained in any critical tunnels of higher generation, at the moment $-n$ when $q_{-n} \notin \cT_{c_{-n}}$, the forward regularity at $q_{-n}$ is fully recovered to be once again uniform. See \figref{fig:tunnel}.

\begin{figure}[h]
\centering
\includegraphics[scale=0.22]{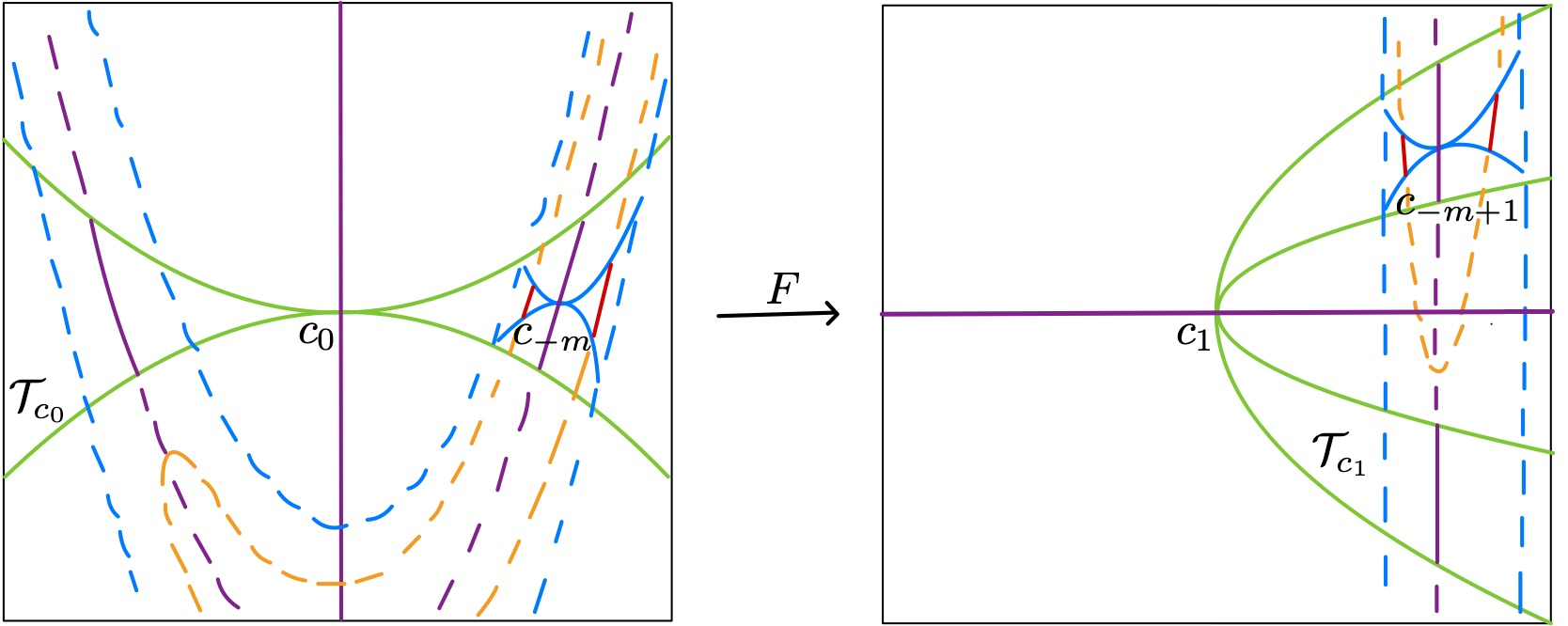}
\caption{Local strong-stable manifold contained in a critical tunnel $\cT_{c_{-m}} \subset \cT_{c_0}$ of generation $1$ (which is the preimage of the critical tunnel $\cT_{c_{-m+1}} \subset \cT_{c_1}$ of generation $0$).}
\label{fig:map}
\end{figure}

Proceeding by induction, we obtain a description of the local geometry of the strong stable manifold of every point in $\Lambda$ that is contained in at most finitely many critical tunnels. See \figref{fig:map}. We conclude by observing that for $q \in \Lambda$, the deeper the nest of critical tunnels which contains $q$, the worse the forward regularity of $F$ at $q$ becomes. Consequently, the forward regularity of $F$ deteriorates completely at points contained in infinite nests of critical tunnels, and the portions of their stable manifolds that have controlled geometry get reduced to nothing more than the points themselves. The set $\Delta_{c_0}^c$ of all such points has zero measure (see \propref{inv dust}), and is referred to as {\it the critical dust of $c_0$}.

Replacing forward regularity with backward regularity, vertical with horizontal, and critical tunnels with valuable crescents in the above discussion, we obtain a description of the local geometry of the center manifold of every point in $\Lambda$ except those contained in infinitely many valuable crescents. These exceptional points have completely deteriorated backward regularity. They form a zero-measure set $\Delta_{c_1}^v$ referred to as {\it the valuable dust of $c_1$}.

\subsection{Summary of the proof of the main theorem}\label{sec:summary}


For $i \in \{0, 1\}$, let $\Phi_{c_i} : (\bbD_i, c_i) \to (\Phi_{c_i}(\bbD_i), 0)$ be the change of coordinates on $\bbD_i$ given in \eqref{eq:intro crit disk} that uniformizes the local dynamics of $F$ along the critical orbit $\{c_m\}_{m\in \bbZ}$. In particular,
\begin{equation}\label{eq:intro henon}
\Phi_{c_1}\circ F\circ \Phi_{c_0}^{-1}(x,y) = (x^2-y, x)
\matsp{for}
(x,y) \in \Phi_{c_0}(\bbD_0).
\end{equation}
In these coordinates, $W^{ss}_{\loc}(c_1)$ becomes a genuine vertical segment, and $W^c(c_0)$ becomes a genuine horizontal segment.

For $n \geq 1$, let $\cD^n$ be the renormalization domain of depth $n$ that contains $c_1$. A {\it renormalization piece of depth $n$} is the set
$$
\Lambda^n_i := \Lambda \cap F^{i-1}(\cD^n)
\matsp{for}
1 \leq i <R_n.
$$
Note that $\Lambda^n_{R_n+1} = \Lambda^n_1$. Hence, $\Lambda^n_{R_n} \ni c_0$. We show that $\displaystyle\bigcap_{n=1}^\infty \Lambda^n_1 = \{c_1\}$ (see \thmref{crit recur}). To simplify our discussion, we will assume that $\Lambda$ is totally disconnected (to make the argument work in the general case, we prove that any non-trivial connected component of $\Lambda$ is a $C^r$-curve that coincides with the center manifold, see \thmref{component arc}).

The main tool used in the proof of the Main Theorem is a structure called a {\it unicritical cover $\cC$ of $\Lambda$}. It is defined as follows. For any uniformly Pesin regular point $p \in \Lambda$, one can define a {\it regular chart} $\Phi_p : (\cB_p, p) \to (B_p, 0)$, which is a $C^r$-diffeomorphism from a quadrilateral $\cB_p \ni p$ (called a {\it regular box at $p$}) to a rectangle $B_p$ with $\|\Phi_p^{\pm 1}\|_{C^r} = O(1)$,
that satisfies the following properties:
\begin{itemize}
\item The local strong stable manifold $W^{ss}_{\loc}(p)$ and the center manifold $W^c(p)$ map to genuine vertical and horizontal line segments respectively.
\item If there is another uniformly Pesin regular point $p' \in \Lambda$ such that $F(\cB_p) \cap \cB_{p'} \ni q$, then a horizontal direction at $F^{-1}(q)$ (with respect to $\Phi_p$) is sent to a sufficiently horizontal direction at $q$ (with respect to $\Phi_{p'}$) under $DF$. Similarly, a vertical direction at $q$ is sent to a sufficiently vertical direction at $F^{-1}(q)$ under $DF^{-1}$.
\end{itemize}
The study of uniformly (partially) hyperbolic systems is greatly facilitated by the fact that the invariant set
can be entirely covered by a finite set of regular boxes, which endows the system with what is referred to as a {\it local product structure}.

In our setting, any union of regular boxes leaves out points in $\Lambda$ that are too close to the critical orbit. To completely cover $\Lambda$, we include critical tunnels and valuable crescents defined in \subsecref{subsec:local geom}. This leads to the following definition of a {\it unicritical cover of $\Lambda$}:
$$
\cC := \{\cB_p\}_{p \in \Lambda^P_{L, \epsilon}} \cup \{\cT_{c_m}\}_{m\in\bbZ}
$$
(where $L \geq 1$ is some sufficiently large regularity constant). By the compactness of $\Lambda$, we may extract a finite subcover $\cC_0 \subset \cC$. If $\Lambda$ is totally disconnected, then for any $n \geq N$ and $i \in \bbN$, there exists $Q \in \cC_0$ such that $\Lambda^n_i \Subset Q$.

For the Main Theorem, we need to ensure that a unicritical cover $\cC_0$ of $\Lambda$ for $F$ can fulfill a similar function as local product structures for uniformly (partially) hyperbolic systems. This requires the full use of our understanding of the forward and backward regularity of $F$ on $\Lambda$ summarized in \subsecref{subsec:local geom}.

For each $Q \in \cC_0$, we endow $\Lambda \cap Q$ with a total order induced by the strong-stable lamination as follows. First consider the case when $Q$ is either a regular box $\cB_p$ for some uniformly Pesin regular point $p \in \Lambda^P_{L, \epsilon}$, or a valuable crescent $\cT_{c_n}$ for some $n \in \bbN$. If $q \in \Lambda^+_{L, \epsilon} \cap Q$ is a uniformly forward regular point, then $W^{ss}_{\loc}(q)$ is a proper straight vertical curve in $Q$. Hence, the strong-stable lamination clearly induces a total order on $\Lambda^+_{L, \epsilon} \cap Q$ (say, their order of appearance from left to right).

Next, consider the case where $Q$ is a critical tunnel $\cT_{c_{-m}}$ for some $m \geq 0$. Let $\Lambda^+(\cT_{c_{-m}})$ be the set of all points $q \in \cT_{c_{-m}}$ such that $\cT_{c_{-m}}$ is the tunnel of the largest generation containing $q$ (i.e. there is no deeper tunnel $\cT_{c_{-k}} \subset \cT_{c_{-m}}$ with $k > m$ containing $q$).

If $q_0 \in \Lambda^+(\cT_{c_0})$, then its image $q_1 \in \cT_{c_1}$ is not contained in any critical tunnel. Hence, $q_1\in\Lambda^+_{L,\epsilon}$, and $W^{ss}_{\loc}(q_1)$ is a proper straight vertical curve in a neighborhood $\cU_{c_1}$ of $c_1$. This means that, in particular, the component $W^{ss}_{\cT_{c_1}}(q_1)$ of $W^{ss}_{\loc}(q_1) \cap \cT_{c_1}$ containing $q_1$ is proper in $\cT_{c_1}$. Define
$$
W^{ss}_{\loc}(q_0) := F^{-1}(W^{ss}_{\cT_{c_1}}(q_1)).
$$
It is easy to see that the strong-stable lamination
$$
\{W^{ss}_{\loc}(q) \; | \; F(q) \in \Lambda^+_{L, \epsilon} \cap \cT_{c_1}\}
$$
induces a total order on $\Lambda^+(\cT_{c_0})$. Pulling this back by the dynamics, we also obtain a total order on the set $\Lambda^+(\cT_{c_{-m}})$ for any $m \geq 0$.

We now explain how these total orders on separate subsets of $\Lambda$ fit together to form a single total order on $\Lambda \cap Q$ for any $Q \in \cC_0$. Suppose $Q$ is either a regular box or a valuable crescent. Let $\bfT_Q^0$ be the set of critical tunnels in $Q$ of generation $0$. For $\cT_{c_{-m}} \in \bfT_Q^0$, we consider a suitable thin proper nearly-vertical strip $S_{c_{-m}}$ that contains $\cT_{c_{-m}}$.
Then the collection 
$$
\{W^{ss}_{\loc}(q)\; | \; q\in\Lambda^+_{L, \epsilon}\cap Q \} \cup \{S_{c_{-m}} \; | \; \cT_{c_{-m}} \in \bfT_Q^0\}
$$
is pairwise disjoint. Hence, the total order on $\Lambda^+_{L, \epsilon}\cap Q$ can be extended to include $\Lambda^+(\cT_{c_{-m}})$ with $\cT_{c_{-m}} \in \bfT_Q^0$.

So far, for any $Q \in \cC_0$, we have a total order on the set of all points in $\Lambda \cap Q$ that are contained in at most one deeper critical tunnel $\cT_{c_{-m}} \subset Q$. Proceeding by induction, we obtain a total order on the set of all points in $\Lambda \cap Q$ contained in at most finitely many critical tunnels. In other words, we have a total order on the set $(\Lambda\setminus  \Delta_{c_0}^c) \cap Q$, where $\Delta_{c_0}^c$ denotes the critical dust of $c_0$ defined in \subsecref{subsec:local geom}. Since the intersection of any infinite nest of critical tunnels is a singleton, the total order extends to $\Delta_{c_0}^c \cap Q$ as well. For more details, see \secref{sec:tot order}.

Now that we have the unicritical cover structure $\cC_0$ with well-defined total orders on its elements, we are able to give a general outline of the proof of the Main Theorem. For $n \geq N$ with $N \in\bbN$ sufficiently large, let $\cV_1^n$ be a suitable thin tubular neighborhood of  $W^{ss}_{\loc}(c_1)$ in $\bbD_1 \ni c_1$ such that $\cV_1^n \supset \Lambda_1^n$. Denote $\cV_i^n := F^{i-1}(\cV_1^n)$ for $1 \leq i <R_n$.
We have $\cV^n_{R_n} = F^{R_n-1}(\cV^n_1) \ni c_0$. We need to show that under $F^{R_n-1}$:
\begin{enumerate}[1)]
\item There is a compression of $\cV^n_1$ in the vertical direction by a factor of $\lambda^{(1-\bepsilon)R_n}$.
\item A horizontal foliation over $\cV^n_1$ is mapped to a horizontal foliation close to $W^c(c_0)$.
\end{enumerate}
Then by applying $F$ one more time in a neighborhood $\cU_{c_0}$ of the critical point $c_0$, the vertically-compressed, horizontally-aligned rectangular set $\cV^n_{R_n} \subset \cU_{c_0}$ gets bent into the image of a H\'enon-like map. See \eqref{eq:intro henon} and \figref{fig:thm}.

\begin{figure}[h]
\centering
\includegraphics[scale=0.22]{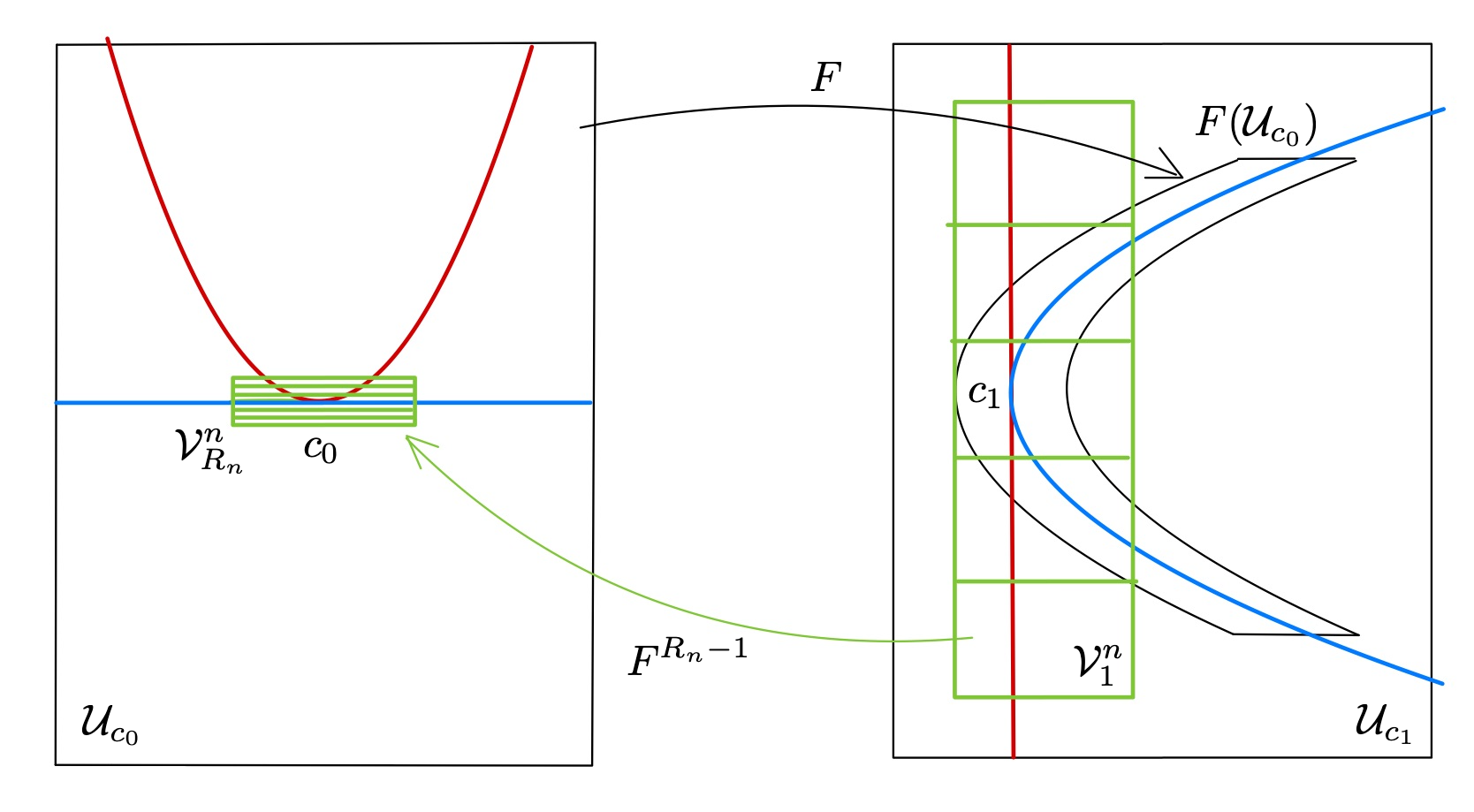}
\caption{An illustration of the statement of the Main Theorem.}
\label{fig:thm}
\end{figure}

There are two crucial differences between the unicritical cover structure $\cC_0$ and the local product structure. The first is that because the covering element $\cT_{c_0} \in \cC_0$ at $c_0$ is pinched, it is a much more delicate problem to ensure that a domain of interest stays contained inside elements in $\cC_0$ under the dynamics. The second is that unlike in the local product structure, horizontal-like and vertical-like directions are not everywhere preserved. Indeed, horizontal directions near $c_0$ switch to vertical-like directions near $c_1$. These difficulties are navigated using the total order on $\cC_0$, which enables us to prove the following crucial facts.
\begin{enumerate}[i)]
\item For $1 \leq i< R_n$, if $\Lambda^n_i \subset Q \in \cC_0$, then the extremal points of $\Lambda^n_i$ with respect to the total order on $\Lambda\cap Q$ are $c_{1+i}$ and $c_{1+R_n+i}$. Moreover, $\cV^n_i$ stay horizontally confined between the strong-stable manifolds of these extremal points.
\item We have
$$
\Lambda^{n+1}_1 \cap \cT_{c_{-R_n+1}} = \varnothing
\matsp{for}
n \geq N.
$$
\item We have
$$
\Lambda^{n+1}_{R_n} \cap \cT_{c_{R_n}} = \varnothing
\matsp{for}
n \geq N.
$$
\end{enumerate}
Facts i) and ii) are used to show that if $\Lambda^n_i \subset Q \in \cC_0$ for some $1 \leq i < R_n$, then $\cV^n_i$ is contained in a slight thickening $\hQ$ of $Q$. Fact iii) is used to show that after passing through the critical value $c_1$, horizontal directions in $\cV^n_i$ recover their alignment to the local neutral direction before the next closest approach to the critical point $c_0$. This means that the vertical directions in $\cV^n_i$ typically stay transverse to the neutral direction. Hence, they must contract under the dynamics. The Main Theorem follows. For more details, see Sections \ref{sec:order on piece} and \ref{sec:henon return}.

\subsection{Outline of the paper}\label{sec:organization}

To help the reader, we outline now the main content of each section and we refer to the main result proved in each one. Some of the sections  revisit results proved somewhere else.

In Section \ref{sec:quant pesin} we present the classical Pesin theory insisting in certain  quantitative  properties that are needed in our context, in particular, to deal with (forward/backward) ``regular finite pieces of orbits"  that are not necessary part of regular orbits in the  classical sense of 
Pesin theory. The main result is Theorem \ref{reg chart} that provides linearizing coordinates for (finite) regular points. 

In Section \ref{sec:homog} we revisit a quantitative notion of dissipation that was introduced in \cite{CP} and proved to be satisfied for odometers in \cite{CPT}. It allows us to show that the regularity conditions presented in Section \ref{sec:quant pesin} are satisfied.

In Section \ref{sec:crit orb} we obtain linearization coordinates at the critical point and value (see Proposition \ref{unif reg crit}) and we estimate the distance of the closest return of the critical point to itself (see Proposition \ref{no crit in crit nbh}).

In the first part of Section \ref{sec:reg near crit orb} we show how the finite (forward/backward) regularity of points in scaled neighborhood of the critical point is controlled (Propositions \ref{for reg rec} and \ref{back reg rec} for obtaining regularity and Propositions \ref{for reg no rec} and \ref{back reg no rec} for obstructing regularity). In the second part (Section \ref{subsec:crit tunnel}), we introduce the notions of crescents and tunnels (and the scaled ones related to the pre-iterates of the tunnels).

In Section \ref{sec:mild diss} we revisit the notion of mild dissipativity. Recall that this notion states that the stable manifold of any regular point separates the attracting disk where the dynamics is defined. In this section we  estimate quantitatively  how uniform pieces of local stable manifold separate forward iterates of the attracting disk.

In Section \ref{sec:inf ren} we recall a classical result about the unique ergodicity of odometers and we characterize the trivial connected components of the limit set. 

In Section \ref{sec:unicrit} we show that the odometer nearby the critical point is contained inside the pinched tunnel (Theorem \ref{pinching}). 

In Section \ref{sec:tot order} we show how scaled tunnels fits in the regular boxes and in the initial tunnel, providing first a local total order and later a global total order. More precisely, using regular boxes and scaled tunnels, in Section \ref{subsec:unicrit cover} we introduce the notion of unicritical cover and we show that such a cover is ``totally ordered" (Theorem \ref{thm:total order}).


\section{Quantitative Pesin Theory}\label{sec:quant pesin}

Let $r \geq 2$ be an integer. Consider a dissipative $C^r$-diffeomorphism $F : \Omega \to F(\Omega) \Subset \Omega$ defined on a Jordan domain $\Omega \subset \bbR^2$. Let $\Lambda \Subset \Omega$ be a totally invariant compact set, and let $\mu$ be an ergodic probability measure supported on $\Lambda$. Suppose that the Lyapunov exponents of $F$ with respect to $\mu$ are $0$ and $\log \lambda <0$. Let $\epsilon >0$ be a small constant such that $\bepsilon < 1$.


\subsection{Definitions of regularity}

Let $N, M \in \bbN\cup \{0,\infty\}$ and $L \geq 1$. A point $p \in \Lambda$ is {\it $N$-times forward $(L, \epsilon)_v$-regular along $E^+_p \in \bbP^2_p$} if for $s \in \{-r, r-1\}$, we have
\begin{equation}\label{eq:for reg}
L^{-1}\lambda^{(1+\epsilon)n}\leq \frac{(\Jac_pF^n)^s}{\|DF^n|_{E^+_p}\|^{s-1}} \leq L \lambda^{(1-\epsilon)n}
\matsp{for all}
1 \leq n \leq N.
\end{equation}
Similarly, $p$ is {\it $M$-times backward $(L, \epsilon)_v$-regular along $E^-_p \in \bbP^2_p$} if for $s \in \{-r, r-1\}$, we have
\begin{equation}\label{eq:back reg}
L^{-1}\lambda^{-(1-\epsilon)n}\leq \frac{(\Jac_pF^{-n})^s}{\|DF^{-n}|_{E^-_p}\|^{s-1}} \leq L \lambda^{-(1+\epsilon)n}
\matsp{for all}
1 \leq n \leq M.
\end{equation}
If \eqref{eq:for reg} and \eqref{eq:back reg} are both satisfied with $E_p^v := E_p^+ = E_p^-$, then we say that $p$ is {\it $(M, N)$-times $(L, \epsilon)_v$-regular along $E_p^v$}. If additionally, we have $M=N = \infty$, then we say that $p$ is {\it Pesin $(L, \epsilon)_v$-regular along $E_p^v$}.

We say that $p$ is {\it $N$-times forward $(L, \epsilon)_h$-regular along $\tiE_p^+ \in \bbP_p^2$} if for $s \in \{-r+1, r\}$, we have
\begin{equation}\label{eq:hor for reg}
L^{-1}\lambda^{(1+\epsilon)n}  \leq \frac{\Jac_p F^n}{\|D_pF^n|_{\tiE_p^+}\|^{s+1}} \leq L\lambda^{(1-\epsilon)n}
\matsp{for}
1 \leq n \leq N.
\end{equation}
Similarly, we say that $p$ is {\it $M$-times backward $(L, \epsilon)_h$-regular along $\tiE_p^- \in \bbP_p^2$} if for $s \in \{-r+1, r\}$, we have
\begin{equation}\label{eq:hor back reg}
L^{-1}\lambda^{-(1-\epsilon)n}  \leq  \frac{\Jac_p F^{-n}}{\|D_pF^{-n}|_{\tiE_p^-}\|^{s+1}} \leq L\lambda^{-(1+\epsilon)n}
\matsp{for}
1 \leq n \leq M.
\end{equation}
If both \eqref{eq:hor for reg} and \eqref{eq:hor back reg} hold with $E_p^h := \tiE_p^+ = \tiE_p^-$, then $p$ is {\it $(M, N)$-times $(L, \epsilon)_h$-regular along $E_p^h$}. If, additionally, we have $M = N =\infty$, then $p$ is {\it Pesin $(L, \epsilon)_h$-regular along $E_p^h$}.


\begin{prop}[Vertical forward regularity $=$ horizontal forward regularity]\label{hom transverse for reg}
Suppose $p$ is $N$-times forward $(L, \epsilon)_v$-regular along $E_p^v \in \bbP_p^2$. Let $E_p^h \in \bbP_p^2 \setminus \{E_p^v\}$. If $\measuredangle (E_p^v, E_p^h) > \theta$, then the point $p$ is $N$-times forward $(\bL/\theta^2, \epsilon)_h$-regular along $E_p^h$.

Conversely, if $p$ is $N$-times forward $(L, \epsilon)_h$-regular along $E_p^h \in \bbP_p^2$, then there exists $E_p^v \in \bbP_p^2$ such that $p$ is $N$-times forward $(\bL, \bepsilon)_v$-regular along $E_p^v$.
\end{prop}

\begin{prop}[Horizontal backward regularity $=$ vertical backward regularity]\label{hom transverse back reg}
Suppose $p$ is $M$-times backward $(L, \epsilon)_h$-regular along $E_p^h \in \bbP_p^2$. Let $E_p^v \in \bbP_p^2 \setminus \{E_p^h\}$. If $\measuredangle (E_p^h, E_p^v) > \theta$, then the point $p$ is $M$-times backward $(\bL/\theta^2, \epsilon)_v$-regular along $E_p^v$.

Conversely, if $p$ is $M$-times backward $(L, \epsilon)_v$-regular along $E_p^v \in \bbP_p^2$, then there exists $E_p^h \in \bbP_p^2$ such that $p$ is $M$-times backward $(\bL, \bepsilon)_h$-regular along $E_p^h$.
\end{prop}

\begin{prop}[Pesin regularity $=$ vertical forward regularity $+$ horizontal backward regularity $+$ transversality]\label{hom for back reg}
Suppose $p$ is $N$-times forward $(L, \epsilon)_v$-regular along $E_p^v \in \bbP_p^2$ and $M$-times backward $(L, \epsilon)_h$-regular along $E_p^h \in \bbP_p^2$ with $\theta := \measuredangle (E_p^v, E_p^h) > 0$. Let $\cL \geq 1$ be the minimum value such that $p$ is $(M, N)$-times $(\cL, \bepsilon)_v$-regular along $E_p^v$ and $(\cL, \bepsilon)_h$-regular along $E_p^h$. Then we have $k/\theta^2 < \cL < \bL/\theta^2$ for some uniform constant $k > 0$.
\end{prop}

Suppose $p_0 \in \Lambda$ is $(M, N)$-times $(L, \epsilon)_{v/h}$-regular along $E_{p_0}^{v/h} \in \bbP^2_{p_0}$. Define the {\it regularity} of $p_m$ as the smallest value $\cL_{p_m} \geq 1$ such that $p_m$ is $(M+m, N-m)$-times $(\cL_{p_m}, \epsilon)_{v/h}$-regular along $E_{p_m}^{v/h}$.

\begin{prop}[Decay in regularity]\label{decay reg}
For $-M \leq m \leq N$, we have
$$
\cL_{p_m} < L^2\lambda^{-2\epsilon |m|}.
$$
\end{prop}

\begin{rem}
By \propref{hom transverse for reg}, \ref{hom transverse back reg} and \ref{hom for back reg}, we see that regularity along a vertical direction and regularity along a horizontal direction are interchangeable conditions (at the cost of increasing the constants of regularity by some uniform amount). Henceforth, the vertical direction will be taken as the default unless otherwise stated.
\end{rem}


\subsection{Linearization along regular orbits}

For the remainder of this section, we assume that $p_0 \in \Lambda$ is $(M, N)$-times $(L, \epsilon)_v$-regular along $E_{p_0}^v \in \bbP^2_{p_0}$ for some $M, N, \in \bbN \cup \{\infty\}$ and $L \geq 1$.

For $l, w > 0$, denote
$$
\bbI(l) := (-l, l) \subset \bbR,
\hspace{5mm}
\bbB(l, w) := \bbI(l)\times \bbI(w) \subset \bbR^2
\matsp{and}
\bbB(l) := \bbB(l,l).
$$

\begin{thm}\label{reg chart}
For $-M \leq m \leq N$, let
$$
l_{p_m} := \bL^{-1}\lambda^{\bepsilon |m|}
\matsp{and}
U_{p_m} := \bbB(l_{p_m}).
$$
Then there exists a $C^r$-diffeomorphism $\Phi_{p_m} : (\cU_{p_m}, p_m) \to (U_{p_m}, 0)$ such that
$$
\|\Phi_{p_m}^{\pm 1}\|_{C^r} = O(\bL \lambda^{-\bepsilon |m|})
\comma
D\Phi_{p_m}(E^v_{p_m}) = E^{gv}_0,
$$
and $\Phi_{p_{n+1}} \circ F|_{\cU_{p_m}} \circ \Phi_{p_m}^{-1}$ extends to a globally defined $C^r$-diffeomorphism $F_{p_m} : (\bbR^2, 0) \to (\bbR^2, 0)$ satisfying the following properties.
\begin{enumerate}[i)]
\item We have $\displaystyle \|F_{p_m}^{\pm 1}\|_{C^r} = O(1)$.
\item The map $F_{p_m}$ is uniformly $C^1$-close to
$$
D_0F_{p_m} = A_m = \begin{bmatrix}
a_m & 0\\
0 & b_m
\end{bmatrix},
$$
with $\lambda^{1+\bepsilon} < b_m < \lambda^{1-\bepsilon}$ and $\lambda^\bepsilon < a_m < \lambda^{-\bepsilon}$.
\item We have
$$
F_{p_m}(x,y) = (f_{p_m}(x), e_{p_m}(x,y))
\matsp{for}
(x,y) \in \bbR^2,
$$
where $f_{p_m}:(\bbR, 0) \to (\bbR, 0)$ is a $C^r$-diffeomorphism, and $e_{p_m} : \bbR^2 \to \bbR$ is a $C^r$-map with $e_{p_m}(\cdot ,0) \equiv 0$.
\end{enumerate}
\end{thm}

The construction in \thmref{reg chart} is referred to as {\it a linearization of $F$ along the $(M,N)$-orbit of $p_0$ with vertical direction $E^v_{p_0}$}. For $-M \leq m \leq N$, we refer to $l_{p_m}$, $\cU_{p_m}$, $\Phi_{p_m}$ and $F_{p_m}$ as a {\it regular radius}, a {\it regular neighborhood}, a {\it regular chart} and a {\it linearized map at $p_m$} respectively. For $1 \leq n \leq N-m$, we denote
$$
F_{p_m}^n := F_{p_{m +n-1}} \circ \ldots \circ F_{p_{m+1}}\circ F_{p_m}.
$$
Let $X_m \subset U_{p_m}$ be a set. We denote
$$
X_{m+l} := F_{p_m}^l(X_m)
$$
for all $l \in \bbZ$ for which the right-hand side is well-defined. Similarly, for any direction $E_{z_m} \in \bbP^2_{z_m}$ at a point $z_m \in U_{p_m}$, we denote
$$
E_{z_{m+l}} := DF_{p_m}^l(E_{z_m}).
$$

\begin{rem}
Note that we have the freedom to make the regular radii used in \thmref{reg chart} uniformly smaller whenever necessary (i.e. replace $l_{p_m}$ by $\underline{l_{p_m}}$ for all $-M \leq m \leq N$).
\end{rem}

\begin{prop}\label{reg nbh size}
For $-M \leq m \leq N$, we have
$$
\diam(\cU_{p_m}) \asymp \bL^{-1}\lambda^{\bepsilon |m|}.
$$
\end{prop}

The following result states that restricted to the regular neighborhoods, iterates of $F$ are nearly linear.

\begin{prop}[Linear dynamics inside regular neighborhoods]\label{loc linear}
For any constant $k > 0$, the regular radii in \thmref{reg chart} can be chosen sufficiently small so that the following holds. Let $-M\leq m \leq N$ and $-M - m \leq l \leq N - m$. Suppose that $q_{m + i} \in \cU_{p_{m + i}}$ for $i \in [m, m+l] \cap \bbZ$.
Write $z_m := \Phi_{p_m}(q_m) \in U_{p_m}$. Then for all $v \in \bbR^2$, we have
$$
\|D_{z_m} F_{p_m}^l(v) - D_0 F_{p_m}^l(v)\| < k \|D_0F_{p_m}^l(v)\|
$$
and
$$
\|D_{q_m} F^l(v) - D_{p_m} F^l(v)\| < k \|D_{p_m}F^l(v)\|.
$$
Moreover,
$$
1-k< \frac{\Jac_{z_m}F_{p_m}^l}{\Jac_0F_{p_m}^l}, \frac{\Jac_{q_m}F^l}{\Jac_{p_m}F^l} < 1+k.
$$
\end{prop}


\subsection{Vertical and horizontal directions}

For $z \in \bbR^2$, denote the genuine vertical/horizontal direction at $z$ by $E^{gv/gh}_z \in \bbP^2_z$ respectively.

Let $-M \leq m \leq N$. For $q \in \cU_{p_m}$, write $z := \Phi_{p_m}(q)$. The {\it vertical/horizontal direction at $q$ in $\cU_{p_m}$} is defined as $E^{v/h}_q := D\Phi_{p_m}^{-1}(E^{gv/gh}_z)$. For $t\geq 0$, we say that $E_q \in \bbP^2_q$ is {\it $t$-vertical/horizontal in $\cU_{p_0}$} if
$$
\frac{\measuredangle(E_q, E_q^{v/h})}{\measuredangle(E_q, E_q^{h/v})} <t.
$$
Similarly, we say that a $C^1$-curve $\gamma \subset \cU_{p_0}$ is {\it $t$-vertical/horizontal in $\cU_{p_0}$} if
$$
\frac{\measuredangle(\gamma'(s), E_{\gamma(s)}^{v/h})}{\measuredangle(\gamma'(s), E_{\gamma(s)}^{h/v})} <t
\matsp{for}
s \in (0, |\gamma|).
$$

By the construction of regular charts in \thmref{reg chart}, vertical directions are invariant under $F$:
$$
\text{i.e.}\hspace{5mm}
DF(E_q^v) = E_{F(q)}^v
\matsp{for}
q \in \cU_{p_m}.
$$
Note that the same is not true for horizontal directions. However, the following result states that they are still nearly invariant under $F$.

\begin{prop}[Near invariance of horizontal directions]\label{hor near inv}
Let $-M\leq m \leq N$ and $-M - m \leq l \leq N - m$. Suppose that $q_{m + i} \in \cU_{p_{m + i}}
\matsp{for}
i \in [m, m+l] \cap \bbZ$. Let $\tiE_{q_{m+l}}^h := D_{q_m}F^l(E_{q_m}^h)$.
Write $z_m = (x_m, y_m) := \Phi_{p_m}(q_m)$ and
$$
\tiE_{z_{m+l}}^h := D_{z_m}F_{p_m}^l(E_{z_m}^{gh}) = D_{q_{m+l}}\Phi_{p_{m+l}}(\tiE_{q_{m+l}}^h).
$$
Then we have
$$
\measuredangle(\tiE_{z_{m+l}}^h, E_{z_{m+l}}^{gh}), \measuredangle(\tiE_{q_{m+l}}^h, E_{q_{m+l}}^h) < K|y_{m+l}|^{1-\bepsilon}
$$
for some uniform constant $K > 1$.
\end{prop}

\begin{prop}[Vertical alignment of forward contracting directions]\label{vert angle shrink}
Let $q_0 \in \cU_{p_0}$ and $\tiE_{q_0}^v \in \bbP^2_{q_0}$. Suppose for some $0 < n \leq N$, we have $q_i \in \cU_{p_ i}$ for $0 \leq i \leq n$. If $\nu := \|DF^n|_{\tiE_{q_0}^v}\| < \bL^{-1}\lambda^{\bepsilon n}$,
then
$$
\measuredangle(\tiE_{q_0}^v, E_{q_0}^v) < \bL\lambda^{-\bepsilon n}\nu + \bL\lambda^{(1-\bepsilon) n}.
$$
\end{prop}

\begin{prop}[Horizontal alignment of backward neutral directions]\label{hor angle shrink}
Let $q_0 \in \cU_{p_0}$ and $\tiE_{q_0}^h \in \bbP^2_{q_0}$. Suppose for some $0 < m \leq M$, we have $q_{-i} \in \cU_{p_ {-i}}$ for $0 \leq i \leq m$. If $\mu := \|DF^{-m}|_{\tiE_{q_0}^h}\| < \bL^{-1}\lambda^{-(1-\bepsilon) m}$, then
$$
\measuredangle(\tiE_{q_0}^h, E_{q_0}^h) < \bL\lambda^{(1-\bepsilon) m}(1+\mu).
$$
\end{prop}

Let $\cE : \cD \to T^1\cD$ be a unit vector field on a domain $\cD \subset \Omega$. Define
$$
DF_*(\cE)(p) := \frac{DF(\cE(p))}{\|DF(\cE(p))\|} \in T^1_{F(p)}F(\cD)
\matsp{for}
p \in \cD.
$$
If $\cE(p)$ is $t$-vertical for all $p \in \cD \cap \cU_{p_0}$, then $\cE$ is said to be {\it $t$-vertical in $\cU_{p_0}$}.

For $-N \leq m \leq N$, define $\cE^v_{p_m} : \cU_{p_m} \to T^1(\cU_{p_m})$ to be the unit vector field given by
$$
\cE^v_{p_m}(q) \in E^v_q
\matsp{for}
q \in \cU_{p_m}.
$$

\begin{prop}[Vertical alignment of non-horizontal vector fields under backward iterates]\label{vert fol conv}
Let $\cD_0 \subset \cU_{p_0}$ be a domain, and $0 \leq n \leq N$. Suppose
$$
\cD_i := F^i(\cD_0) \subset \cU_{p_i}
\matsp{for}
0\leq i \leq n.
$$
Let $\cE : \cD_n \to T^1(\cD_n)$ be a $C^{r-1}$-unit vector field. If $\cE$ is $t$-vertical in $\cU_{p_n}$ for some $t \geq 0$, then we have
$$
\|DF^{-n}_*(\cE) - \cE^v_{p_0}|_{\cD_0}\|_{C^{r-1}} \leq (1+t^2)(1+\|\cE\|_{C^{r-1}})\bL\lambda^{(1-\bepsilon)n}.
$$
\end{prop}

Let $\gamma \subset \bbR^2$ be a $C^r$-curve. We define
$$
\|\gamma\|_{C^r} := \|\phi_\gamma\|_{C^r},
$$
where $\phi_\gamma : [0, |\gamma|] \to \gamma$ is the parameterization of $\gamma$ by its arclength.

\begin{prop}[Horizontal alignment of non-vertical curves under forward iterates]\label{hor curv conv}
Let $\gamma_0 \subset \cU_{p_0}$ be a $C^r$-curve, and $0 \leq n \leq N$. Suppose
$$
\gamma_i := F^i(\gamma_0) \subset \cU_{p_i}
\matsp{for}
0 \leq i \leq n.
$$
If $\gamma_0$ is $t$-horizontal in $\cU_{p_0}$ for some $t \geq 0$, then $\Phi_{p_n}(\gamma_n)$ is the horizontal graph of some $C^r$-map $g_n : J_n \subset \bbI(l_{p_n}) \to \bbI(l_{p_n})$ with
$$
\|g_n\|_{C^r} \leq (1+t^{r+1})\|\gamma_0\|_{C^r}\bL\lambda^{(1-\bepsilon)n}.
$$
\end{prop}


\subsection{Truncated and pinched regular neighborhoods}\label{subsec:trunc pinch}

Let $\delta \in (\bepsilon, 1)$ be a small constant such that $\bdelta < 1$. For $n \geq 0$, the {\it $n$-times $\delta$-truncated regular neighborhood of $p_0$} is defined as
\begin{equation}\label{eq:trun neigh}
\cU_{p_0}^{\delta n} := \Phi_{p_0}^{-1}\left(U_{p_0}^{\delta n}\right)\subset \cU_{p_0},
\matsp{where}
U_{p_0}^{\delta n} := \bbB\left(\lambda^{\delta n}l_{p_0}, l_{p_0}\right).
\end{equation}

The purpose of truncating a regular neighborhood is to ensure that its iterated images stay inside regular neighborhoods.

\begin{lem}\label{trunc neigh fit}
For $0 \leq n \leq N$, we have
$F^i(\cU_{p_0}^{\bepsilon n}) \subset \cU_{p_i}$ for $0 \leq i \leq n$.
\end{lem}

For $\omega > 1$ and ${t} \in (0, l_{p_0})$, define
$$
T_{p_0}^\omega({t}) := \{(x, y) \in U_{p_0} \; | \; |x| < {t} \; \text{and} \; |y| < |x|^\omega\}.
$$
We refer to
$$
\cT_{p_0}^\omega({t}) := \Phi_{p_0}^{-1}(T_{p_0}^\omega({t})) \subset \cU_{p_0}
$$
as a {\it $\omega$-pinched regular neighborhood of length ${t}$ at $p_0$}. Note that $p_0 \not\in \cT_{p_0}^\omega({t})$. For $n \geq 0$, we also denote
$$
T_{p_0}^{\omega, \delta n}({t}) := T_{p_0}^\omega({t}) \cap U_{p_0}^{\delta n}
\matsp{and}
\cT_{p_0}^{\omega, \delta n}({t}):= \Phi_{p_0}^{-1}(T_{p_0}^{\omega, \delta n}({t}))=\cT_{p_0}^\omega({t}) \cap \cU_{p_0}^{\delta n}.
$$
If the length ${t}$ is not indicated, then it is assumed to be equal to $l_{p_0}$ (i.e. $\cT_{p_0}^\omega := \cT_{p_0}^\omega(l_{p_0})$).

The purpose of pinching a regular neighborhood is to ensure that its iterated preimages stay inside regular neighborhoods.

\begin{lem}\label{pinch neigh fit}
Let $0 \leq n \leq M$. If $\omega > (1+\bepsilon)/\delta$, then $F^{-i}(\cT_{p_0}^{\omega, \delta n}) \subset \cU_{p_{-i}}$ for $0 \leq i \leq n$.
\end{lem}

\begin{lem}\label{pinching neigh}
Let $\omega \in (1, \bepsilon^{-1}]$. For $q_0 \in \cU_{p_0} \setminus \{p_0\}$, write $z_0 = (x_0, y_0) := \Phi_{p_0}(q_0) \in U_{p_0}$. If $q_{-n} \in \cU_{p_{-n}}$ for some
$$
n > (1+\bepsilon)\omega\logl |x_0|,
$$
then $q_0 \in \cT_{p_0}^\omega$.
\end{lem}


\subsection{Regular boxes}\label{subsec:reg box}

Let $L'$ be a constant sufficiently uniformly larger than $L$. In this subsection, we assume that
\begin{equation}\label{eq:consistent box}
M, N > \logl \bLp.
\end{equation}

A  {\it regular box at $p_0$} is defined as
$$
\cB_{p_0} := \Phi_{p_0}^{-1}(B_{p_0}) \Subset \cU_{p_0},
\matsp{where}
B_{p_0} := \bbB(1/L').
$$
Denote
$$
\bbL^h_b := \{(x, b) \; | \; x \in \bbR\}.
$$
The {\it top-bottom boundary of $\cB_{p_0}$} is defined as
$$
\hpartial \cB_{p_0} := \Phi_{p_0}^{-1}(\hpartial B_{p_0}),
\matsp{where}
\hpartial B_{p_0} := \partial B_{p_0} \cap (\bbL^h_{1/L'} \cup \bbL^h_{-1/L'}).
$$

A Jordan curve $\Gamma : (0, 1) \to \cB_{p_0}$ is said to be {\it vertically proper in $\cB_{p_0}$} if $\Gamma$ extends continuously to $[0, 1]$, and $\Gamma(0)$ and $\Gamma(1)$ are in distinct components of $\hpartial \cB_{p_0}$. More generally, a Jordan curve $\Gamma$ in $\Omega$ is said to be vertically proper in $\cB_{p_0}$ if $\Gamma \cap \cB_{p_0}$ is.

Let $q \in \cU_{p_0}$. We say that $E_q \in \bbP^2_q$ is {\it sufficiently vertical/horizontal in $\cU_{p_0}$} if $E_q$ is $1/\bLp$-vertical/horizontal in $\cU_{p_0}$. Similarly, we say that a $C^1$-curve $\Gamma \subset \cU_{p_0}$ is {\it sufficiently vertical/horizontal in $\cU_{p_0}$} if $\Gamma$ is $1/\bLp$-vertical/horizontal in $\cU_{p_0}$.

\begin{prop}[Consistent vertical and horizontal directions inside regular boxes]\label{consist ver hor}
Let $\tiM, \tiN \in \bbN\cup\{0, \infty\}$. Suppose $\tip_0 \in \cB_{p_0}$ is $(\tiM, \tiN)$-times $(L, \epsilon)_v$-regular along $\tiE^v_{\tip_0} \in \bbP^2_{\tip_0}$. Denote the vertical/horizontal direction and the local vertical/horizontal manifold at $\tip_0$ associated with the linearization of $F$ along the $(\tiM, \tiN)$-orbit of $\tip_0$ with vertical direction $\tiE^v_{\tip_0}$ by $\tiE_{\tip_0}^{v/h} \in \bbP^2_{\tip_0}$ and $\tiW^{v/h}_{\loc}(\tip_0) \subset \cU_{\tip_0}$ respectively.
\begin{enumerate}[i)]
\item  If $\tiN > \logl \bLp$, then $\tiE^v_{\tip_0}$ and $\tiW^v(\tip_0)$ are sufficiently vertical in $\cU_{p_0}$.
\item If $\tiM > \logl \bLp$, then $\tiE^h_{\tip_0}$ and $\tiW^h(\tip_0)$ are sufficiently horizontal in $\cU_{p_0}$.
\end{enumerate}
\end{prop}

\begin{prop}[Verticality and horizontality of curves preserved inside regular boxes]\label{pres ver hor curv}
For $0 \leq k \leq K < \infty$, let $p^k \in \Omega$ be $(M, N)$-times $(L, \epsilon)_v$-regular along $E^v_{p^k}$. If $\Gamma_0 \subset \Omega$ is a $C^1$-curve such that
$\Gamma_k \subset \cB_{p^k}$ for all $0 \leq k \leq K$, then the following holds.
\begin{enumerate}[i)]
\item If $\Gamma_0$ is sufficiently horizontal in $\cU_{p^0}$, then $\Gamma_k$ is sufficiently horizontal in $\cU_{p^k}$ for all $0 \leq k \leq K$. If, additionally, $\Gamma_0$ is $C^r$ and $\|\Gamma_0\|_{C^r} < \bL$,
then $\|\Gamma_K\|_{C^r} < \bL$.
\item If $\Gamma_K \in \bbP^2_{q_K}$ is sufficiently vertical in $\cU_{p^K}$, then $\Gamma_k$ is sufficiently vertical in $\cU_{p^k}$ for all $0 \leq k \leq K$. Moreover,
\begin{equation}\label{eq:no stray}
|\Gamma_0| > \bL^{-1}|\Gamma_K|.
\end{equation}
If, additionally, $\Gamma_K$ is $C^r$ and $\|\Gamma_K\|_{C^r} < \bL$, then $\|\Gamma_0\|_{C^r} < \bL$.
\end{enumerate}
\end{prop}


\subsection{Stable and center manifolds}\label{stable-center manifolds}

If $N = \infty$, then \propref{vert angle shrink} implies that $E_{p_0}^v$ is the unique direction along which $p_0$ is infinitely forward $(L, \epsilon, \lambda)$-regular. In this case, we denote
$$
E_{p_0}^{ss} := E_{p_0}^v,
$$
and refer to this direction as the {\it strong stable direction at $p_0$}. Moreover, we define the {\it local strong stable manifold at $p_0$} as
$$
W^{ss}_{\loc}(p_0) := \Phi_{p_0}^{-1}(\{(0, y) \in U_{p_0}\}),
$$
and the {\it strong stable manifold at $p_0$} as
$$
W^{ss}(p_0) := \{q \in \Omega \; | \; F^n(q) \in W^{ss}_{\loc}(p_m) \; \; \text{for some} \; \; n \geq 0\}.
$$

If $M = \infty$, we denote
$$
E_{p_0}^c := E_{p_0}^h,
$$
and refer to this direction as the {\it center direction at $p_0$}. Moreover, we define the {\it (local) center manifold at $p_0$} as
$$
W^c(p_0) := \Phi_{p_0}^{-1}(\{(x, 0) \in U_{p_0}\}).
$$
Unlike strong-stable manifolds, center manifolds are not uniquely and canonically defined. However, the following result states that it still has a canonical jet.

\begin{prop}[Canonical $C^r$-jets of center manifolds]\label{center jet}
Suppose $M = \infty$. Let $\Gamma_0: (-t, t) \to \cU_{p_0}$ be a $C^r$-curve parameterized by its arclength such that $\Gamma_0(0) = p_0$, and for all $n \in \bbN$, we have
$$
\|DF^{-n}|_{\Gamma_0'(t)}\| < \lambda^{-(1-\bepsilon) n}
\matsp{for}
|t| < \lambda^{\epsilon n}.
$$
Then $\Gamma_0$ has a degree $r$ tangency with $W^c(p_0)$ at $p_0$.
\end{prop}

\begin{prop}\label{ss c geo}
If $p$ is infinite-time forward/backward $(L, \epsilon)$-regular. Then
$$
\left\|W^{ss/c}_{\loc}(p)\right\|_{C^r} < \bL.
$$
\end{prop}

\begin{prop}\label{ss c contin}
The strong stable manifold $W^{ss}(p)$ depends $C^r$-continuously on $p \in \Lambda^+_{L, \epsilon}$. The center manifold $W^c(p)$ depends $C^{r-\bepsilon}$-continuously on $p \in \Lambda^-_{L, \epsilon}$.
\end{prop}

\begin{rem}
Henceforth, we say that $p$ is infinite-time forward/backward $(L, \epsilon)$-regular (without specifying $v$/$h$ in the subscript) if $p$ is infinite-time forward/backward $(L, \epsilon)_{v/h}$-regular along $E^{ss/c}_p$. Furthermore, we say that $p$ is Pesin $(L, \epsilon)$-regular if $p$ is infinite-time forward and backward $(L, \epsilon)$-regular, and we have $\measuredangle(E^{ss}_p, E^c_p) > 1/L$ (see \propref{hom for back reg}). The sets of all infinite-time forward/backward/Pesin $(L, \epsilon)$-regular points in $\Lambda$ are referred to as the {\it forward, backward} and {\it Pesin $(L, \epsilon)$-regular blocks}, and are denoted by $\Lambda^+_{L, \epsilon}$, $\Lambda^-_{L, \epsilon}$ and $\Lambda^P_{L, \epsilon}$ respectively.
\end{rem}

\begin{prop}\label{classic pesin}
We have $\mu(\Lambda^P_{L, \epsilon}) \to 1$ as $L \to \infty$.
\end{prop}


\section{Homogeneity}\label{sec:homog}

Let $F$ be the diffeomorphism considered in \secref{sec:quant pesin}. For $\eta \in (0,1)$ with $\baeta < 1$, we say that $F$ is {\it $\eta$-homogeneous on $\Lambda$} if for all $p \in \Lambda$ and $E_p \in \bbP_p^2$, we have:
\begin{enumerate}[i)]
\item $ \lambda^{1+\eta} < \|D_pF|_{E_p}\| < \lambda^{-\eta}$ and
\item $\lambda^{1+\eta} < \Jac_p F < \lambda^{1-\eta}$.
\end{enumerate}
Homogeneity considerably simplifies the regularity conditions given in \eqref{eq:for reg}, \eqref{eq:back reg}, \eqref{eq:hor for reg} and \eqref{eq:hor back reg}. Let $N, M \in \bbN\cup \{0,\infty\}$; $L \geq 1$ and $\epsilon \in (\baeta, 1)$ with $\bepsilon < 1$. Then a point $p \in \Lambda$ is:
\begin{itemize}
\item $N$-times forward $(L, \epsilon)_v$-regular along $E^v_p \in \bbP^2_p$ if
$$
\|DF^n|_{E^v_p}\| \leq L \lambda^{(1-\epsilon)n}
\matsp{for}
1 \leq n \leq N;
$$
\item $M$-times backward $(L, \epsilon)_v$-regular along $E^v_p$ if
$$
\|DF^{-n}|_{E^v_p}\| \geq L^{-1} \lambda^{-(1-\epsilon) n}
\matsp{for}
1 \leq n \leq M;
$$
\item $N$-times forward $(L, \epsilon)_h$-regular along $E^h_p \in \bbP^2_p$ if
$$
\|DF^n|_{E^h_p}\| \geq L^{-1} \lambda^{\epsilon n}
\matsp{for}
1 \leq n \leq N;
$$
\item $M$-times backward $(L, \epsilon)_h$-regular along $E^h_p \in \bbP^2_p$ if
$$
\|DF^{-n}|_{E^h_p}\| \leq L \lambda^{-\epsilon n}
\matsp{for}
1 \leq n \leq M.
$$
\end{itemize}

\begin{prop}\label{get homog}
Suppose $F|_\Lambda$ is uniquely ergodic. Then for any $\eta>0$, there exists $N = N(\eta) \in \bbN$ such that if $n \geq N$, then the map $F^n$ is $\eta$-homogeneous on $\Lambda$.
\end{prop}


\subsection{Pliss moments}

Let $\eta, \epsilon, \delta \in (0,1)$ with $\baeta < \epsilon < \delta$ and $\bdelta < 1$.
Suppose that $F$ is $\eta$-homogeneous on $\Lambda$. We show that for such maps, regularity is a statistically common occurrence.

Let $p_0 \in \Lambda$. Suppose that along its forward/backward/full orbit, $p_0$ exhibits asymptotically regular behavior with marginal exponent $\epsilon$ (in the sense made rigorous in the propositions stated below). Loosely speaking, an integer $m$ is a Pliss moment for the orbit of $p_0$ if $p_m$ is $(1, \delta)$-regular.

\begin{prop}[Frequency of Pliss moments along forward orbit]\label{freq pliss for}
Let $p_1 \in \Lambda$. Suppose there exist $E_{p_1}^v$ or $E_{p_1}^h \in \bbP^2_{p_1}$ such that
$$
\liminf_{n \to \infty} \frac{1}{n}\logl\|DF^n|_{E_{p_1}^v}\| > 1-\epsilon
$$
or
$$
\limsup_{n \to \infty} \frac{1}{n}\logl\|DF^n|_{E_{p_1}^h}\| < \epsilon.
$$
Let $\{n_i\}_{i=1}^\infty \subset \bbN$ be the increasing set of all moments such that $p_{n_i}$ is $(n_i, \infty)$-times $(1, \delta)_{v/h}$-regular along $E_{p_{n_i}}^{v/h}$. Then
$$
\liminf_{i \to \infty}\frac{i}{n_i} > 1- \bepsilon/\delta
$$
If $p_1$ is infinitely forward $(L, \epsilon)_{v/h}$-regular along $E_{p_1}^{v/h}$ for some $L \geq 1$, then
$$
i > \left(1-\bepsilon/\delta\right)n_i + \delta^{-1}\logl L
\matsp{and}
n_i  > \left(1-\bepsilon/\delta\right)n_{i+1} + \delta^{-1}\logl L.
$$
\end{prop}

\begin{prop}[Frequency of Pliss moments along backward orbit]\label{freq pliss back}
Let $p_{-1} \in \Lambda$. Suppose there exists $E_{p_{-1}}^v$ or $E_{p_{-1}}^h \in \bbP^2_{p_{-1}}$ such that
$$
\liminf_{m \to -\infty} \frac{1}{m}\logl\|DF^m|_{E_{p_{-1}}^v}\| > 1-\epsilon
$$
or
$$
\limsup_{m \to -\infty} \frac{1}{m}\logl\|DF^m|_{E_{p_{-1}}^h}\| <\epsilon.
$$
Let $\{-m_i\}_{i=1}^\infty \subset -\bbN$ be the decreasing set of all moments such that $p_{-m_i}$ is $(\infty, m_i)$-times $(1, \delta)_{v/h}$-regular along $E_{p_{-m_i}}^{v/h}$. Then
$$
\liminf_{i \to \infty}\frac{i}{m_i} > 1- \bepsilon/\delta.
$$
If $p_{-1}$ is infinitely backward $(L, \epsilon)_{v/h}$-regular along $E_{p_{-1}}^{v/h}$ for some $L \geq 1$, then
$$
i> (1-\bepsilon/\delta)m_i + \delta^{-1}\logl L
\matsp{and}
m_i > (1-\bepsilon/\delta)m_{i+1} + \delta^{-1}\logl L.
$$
\end{prop}

\begin{lem}\label{conditional pliss full}
Let $p_0 \in \Lambda$ be Pesin $(L, \epsilon)_{v/h}$-regular along $E^{v/h}_{p_0} \in \bbP^2_{p_0}$, and let
$$
n > -\delta^{-1}\logl \bL.
$$
If $p_n$ is $(n, \infty)$-times $(1, \udelta)_{v/h}$-regular along $E^{v/h}_{p_n}$, then $p_n$ is Pesin $(1, \delta)_{v/h}$-regular along $E_{p_n}^{v/h}$. Similarly, if $p_{-n}$ is $(\infty, n)$-times $(1, \udelta)_{v/h}$-regular along $E^{v/h}_{p_{-n}}$, then $p_{-n}$ is Pesin $(1, \delta)_{v/h}$-regular along $E_{p_{-n}}^{v/h}$. 
\end{lem}

\begin{prop}[Frequency of Pliss moments along full orbit]\label{freq pliss full}
Let $p_0 \in \Lambda$. Suppose there exists $E_{p_0}^v$ or $E_{p_0}^h \in \bbP^2_{p_0}$ such that
$$
\liminf_{|m| \to \infty} \frac{1}{m}\logl\|DF^m|_{E_{p_0}^v}\| > 1-\epsilon
$$
or
$$
\limsup_{|m| \to \infty} \frac{1}{m}\logl\|DF^m|_{E_{p_0}^h}\| <\epsilon.
$$
Let $\{m_i\}_{i=-\infty}^\infty \subset \bbZ$ be the increasing set of all moments with
$$
m_n \geq 0
\matsp{and}
m_{-n} \leq 0
\matsp{for}
n \geq 0,
$$
such that $p_{m_i}$ is Pesin $(1, \delta)_{v/h}$-regular along $E_{p_{m_i}}^{v/h}$. Then
$$
\liminf_{|i| \to \infty}\frac{i}{m_i} > 1- \bepsilon/\delta.
$$
\end{prop}

The numbers $n_i$ in \propref{freq pliss for}, $-m_i$ in \propref{freq pliss back} and $m_i$ in \propref{freq pliss full} are referred to as {\it Pliss moments}.

Define
$$
\Lambda^\pm_\epsilon := \bigcup_{L \geq 1} \Lambda^\pm_{L, \epsilon}
\matsp{and}
\Lambda^P_\epsilon := \bigcup_{L \geq 1} \Lambda^P_{L, \epsilon}.
$$
Note that these sets are totally invariant. Also, recall that $\Lambda^\pm_{\Ly}$ and $\Lambda^P_{\Ly}$ denote the set of forward/backward and Pesin Lyapunov regular points in $\Lambda$ respectively.

\begin{prop}\label{lya reg}
We have
$
\Lambda^\pm_{\Ly}\subset \Lambda^\pm_\epsilon
$ and $
\Lambda^P_{\Ly} \subset \Lambda^P_{\epsilon}.
$
\end{prop}

\begin{prop}\label{get point reg}
Let $p_0\in \Lambda$ be a Lyapunov forward/backward/Pesin regular point. Then there exists $N = N(p_0) \geq 0$ such that for $n \geq N$, the point $p_0$ is infinitely forward/backward/Pesin $(1, \epsilon)$-regular for the map $F^n$.
\end{prop}

\begin{prop}\label{uni erg pos meas}
Suppose that $\mu$ is the unique invariant probability measure for $F|_\Lambda$. Then
$$
\mu(\Lambda^{\pm}_{1, \epsilon}) > \frac{\epsilon}{\epsilon+\eta}
\matsp{and}
\mu(\Lambda^P_{1, \epsilon}) > \frac{\epsilon - \eta}{\epsilon+\eta}.
$$
\end{prop}


\section{Critical Orbits}\label{sec:crit orb} 

Let $F : \Omega \to F(\Omega) \Subset \Omega$ be the diffeomorphism considered in \secref{sec:quant pesin} (not necessarily homogeneous). For $p \in \Lambda$ and $n \in \bbN$, let $r_p : \bbR/\bbZ \to \bbR^+$ and $\theta_p : \bbR/\bbZ \to \bbR/\bbZ$ be given by
$$
r_p(t)(\cos(2\pi\theta_p(t)), \sin(2\pi\theta_p(t))) := D_pF(\cos(2\pi t), \sin(2\pi t)).
$$
For
$$
E_p^t = \{l(\cos(2\pi t) , \sin(2\pi t)) \; | \; l \in \bbR\} \in \bbP^2_p,
$$
the {\it projective derivative of $F$ at $E_p^t$} is defined as
$$
\partial_\bbP F(E_p^t) := \theta_p'(t) = \frac{\Jac_pF}{\|D_pF|_{E_p^t}\|^2}.
$$

\begin{defn}\label{def general crit point}
A point $c \in \Lambda$ is a {\it (general) critical point} if there exists a tangent direction $E_c^* \in \bbP_c^2$ such that
$$
\partial_\bbP F^{\pm n}(E_c^*) \geq 1
\matsp{for}
n \in \bbN.
$$
In this case, $E_c^*$ is referred to as a {\it (general) critical direction at $c$}.
\end{defn}

Related to critical point, the following theorem is proved in \cite{CP2}.

\begin{thm}\label{dom split obs}
The set $\Lambda$ admits a dominated splitting if and only if it does not contain any critical point of $F$.
\end{thm}

\begin{rem}
In particular, if $F|_\Lambda$ is dissipative, it follows that it  contains a critical point of $F$ if and only if $F|_\Lambda$ is not uniformly partially hyperbolic. Moreover, as explained in \cite{CPT}, since odometers can not be partially hyperbolic, then they has to contain (general) critical points.
\end{rem}


\subsection{Critical points obstruct linearization}

Let $\epsilon \in (0,1)$ be a small constant such that $\bepsilon <1$.

\begin{prop}[No critical point in regular neighborhoods]\label{no crit in reg nbh}
Let $p_0 \in \Lambda$ be $(M, N)$-times $(L, \epsilon)_v$-regular along $E_{p_0}^v \in \bbP^2_{p_0}$ for some $M,N \in \bbN$ and $L \geq 1$. Let
$$
\{\Phi_{p_m} : \cU_{p_m} \to U_{p_m}\}_{m = -M}^{N}
$$
be a linearization of the $(M,N)$-orbit of $p_0$ with vertical direction $E_{p_0}^v$. Suppose there exists a point $q_0 \in \cU_{p_0}$ such that
$q_m \in \cU_{p_m}$ for all $-M\leq m \leq N.$
If 
$$
M, N > -K\logl L+K
$$
for some uniform constant $K \geq 1$, then $q_0$ cannot be a critical point.
\end{prop}

\begin{proof}
Suppose towards a contradiction that $q_0$ is a critical point with a critical direction $E_{q_0}^*$. Then
$\partial_\bbP F^N(E_{q_0}^*) \geq 1$
and this implies
$$
\|D_{q_0}F^N|_{E_{q_0}^*}\| \leq (\Jac_{q_0}F^N)^{1/2} < \lambda^{(1-\eta)N/2}.
$$
By \propref{loc linear} and \ref{vert angle shrink}, we see that
$$
\measuredangle(E_{q_0}^*, E_{q_0}^v)<\bL \lambda^{-\bepsilon N}\lambda^{(1-\eta)N/2} <k
$$
for some sufficiently small uniform constant $k > 0$. It follows that
$$
\|D_{q_0}F^{-M}|_{E_{q_0}^*}\| > \bL^{-1}\lambda^{-(1-\bepsilon)M}.
$$
Thus,
$$
\partial_\bbP F^{-M}(E_{q_0}^*) = \frac{\Jac_{q_0}F^{-M}}{\|D_{q_0}F^{-M}|_{E_{q_0}^*}\|^2} < \frac{\bL\lambda^{-(1+\epsilon)M}}{\bL^{-2}\lambda^{-2(1-\bepsilon)M}} < \bL\lambda^{(1-\bepsilon)M} < 1.
$$
This is a contradiction.
\end{proof}

\begin{prop}\label{no crit in pesin reg nbh}
Let $p_0 \in \Lambda$ be Pesin $(L, \epsilon)$-regular for some $L \geq 1$, and let
$
\{\Phi_{p_m} : \cU_{p_m} \to U_{p_m}\}_{m \in \bbZ}
$
be a linearization of the full orbit of $p_0$. Then the regular radii (specified in \thmref{reg chart}) can be chosen uniformly small so that $\cU_{p_m}$ does not contain a critical point for all $m \in \bbZ$.
\end{prop}

\begin{proof}
Let  $\{\Phi_{p_m} : \hcU_{p_m} \to \hU_{p_m}\}_{m \in \bbZ}$ be a starting uniformization of the orbit of $p_0$. Recall that
$$
\hU_{p_m} := \bbB(cL^{-C} \lambda^{\bepsilon |m|})
$$
for some uniform constants $c \in (0,1)$ and $C \geq 1$. Let $d \in (0,c)$ and $D > C$, and denote
$$
U_{p_m} := \bbB(dL^{-D} \lambda^{\bepsilon |m|})
\matsp{and}
\cU_{p_m} := \Phi_{p_m}^{-1}(U_{c_m}).
$$

Suppose towards a contradiction that $q_0$ is a critical point and $q_0 \in \cU_{p_m}$ for some $m \in \bbZ$. Then for $i \in \bbN$ such that $q_{-i} \in \hcU_{p_{m-i}}$, we have
$$
\dist(p_{m-i}, q_{-i}) < \lambda^{-(1+\bepsilon) i}dL^{-D}\lambda^{\bepsilon |m|}.
$$
Thus, we have $q_{-i} \in \hcU_{p_{m-i}}$ as long as
$$
cL^{-C} \lambda^{\bepsilon |m-i|} > dL^{-D} \lambda^{-(1+\eta) i }\lambda^{\bepsilon |m|},
$$
or more simply,
$$
(d/c)L^{C-D} < \lambda^{(1+\bepsilon)i}.
$$
Let $i = M$ be the first value for which the above inequality fails. Then
$$
M > (1+\bepsilon)^{-1}\left(\logl(d/c)-(D-C)\logl L \right).
$$

We also have
$$
\dist(p_{m+i}, q_i) < \lambda^{-\bepsilon i} dL^{-D}\lambda^{\bepsilon |m|}
\matsp{for}
i \geq 0.
$$
Thus, we have $q_i \in \hcU_{p_{m+i}}$ as long as
$$
cL^{-C} \lambda^{\bepsilon |m+i|} > \lambda^{-\bepsilon i}dL^{-D}\lambda^{\bepsilon |m|},
$$
or more simply,
$$
(d/c)L^{C-D} < \lambda^{\bepsilon i}.
$$
Let $i = N$ be the smallest value for which the above inequality fails. Then
$$
N > \bepsilon^{-1}\left(\logl(d/c)-(D-C)\logl L\right).
$$

Hence, by choosing $d$ and $D$ sufficiently smaller and larger than $c$ and $C$ respectively, the result follows from \propref{no crit in reg nbh}.
\end{proof}


\subsection{Regular critical values}

A point $c_1 \in \Lambda$ is an {\it $\epsilon$-regular critical value} if $c_1$ is both infinite-time forward and backward $(1, \epsilon)$-regular, and 
$$
E_{c_1}^* := E_{c_1}^{ss} = E_{c_1}^c.
$$
We refer to $c_0 := F^{-1}(c_1)$
as an {\it $\epsilon$-regular critical point}. Note that the strong stable manifold $W^{ss}(c_1)$ and the center manifold $W^c(c_1)$ have a tangency at $c_1$. If this tangency is of order $2$, then we say that $c_1$ is {\it of quadratic type}.

\begin{prop}\label{unique crit value}
Let $c_1 \in \Lambda$ be an $\epsilon$-regular critical value. Then $c_1$ is the unique general critical point in its orbit.
\end{prop}

\begin{prop}\label{get crit reg}
Suppose $F$ is $\eta$-homogeneous for some $\eta \in (0, \uepsilon)$. If $c_1 \in \Lambda$ is a regular critical orbit point, then there exists $N \in \bbN$ such that for $n \geq N$, the point $c_1$ is an $\epsilon$-regular critical value for the map $F^n$.
\end{prop}

\begin{proof}
The result follows immediately from \propref{get point reg}.
\end{proof}

For $a, b \in \bbR$, denote
$$
\bbL^h_b := \{(x, b) \; | \; x\in\bbR\}
\matsp{and}
\bbL^v_a := \{(a, y) \; | \; y\in\bbR\}.
$$
Additionally, let $f_a(x) := x^2 + a$, and denote
$$
\bbP^h_a := \{(x, f_a(x)) \; | \; x \in \bbR\}
\matsp{and}
\bbP^v_a := \{(f_a(y), y) \; | \; y \in \bbR\}.
$$

\begin{thm}[Uniformization along a regular critical orbit]\label{unif reg crit}
Let $c_1$ be an $\epsilon$-regular critical value of quadratic type with the critical direction $E_{c_1}^* \in \bbP^2_{c_1}$. Then there exist linearizations
$$
\{\Phi_{c_{1+n}} : \cU_{c_{1+n}} \to U_{c_{1+n}}\}_{n =0}^\infty
\matsp{and}
\{\Phi_{c_{-n}}: \cU_{c_{-n}} \to U_{c_{-n}}\}_{n=0}^\infty
$$
of the forward orbit of $c_1$ and the backward orbit of $c_0$ respectively that induce the following H\'enon-like transition at the critical moment:
$$
F_{c_0}(x,y) := \Phi_{c_1} \circ F|_{\cU_{c_0}} \circ \Phi_{c_0}^{-1}(x,y) = (x^2 - \lambda y, x)
\matsp{for}
(x,y) \in U_{c_0}.
$$
\end{thm}

\begin{proof}
Let $\{\Psi_{c_{1+n}}\}_{n = 0}^\infty$ and $\{\Psi_{c_{-n}}\}_{n=0}^\infty$ be some linearizations of the forward orbit of $c_1$ and backward orbit of $c_0$. Note that
$$
\Psi_{c_1}(W^{ss}_{\loc}(c_1)) \subset \bbL^v_0
\matsp{and}
\Psi_{c_0}(W^c(c_0)) \subset \bbL^h_0.
$$
We may assume that the regular radius $\rho_0 := l_{c_0}$ at $c_0$ is chosen uniformly smaller than the regular radius $\rho_1 := l_{c_1}$ at $c_1$ such that $F(\cU_{c_0}) \subset \cU_{c_1}$. Denote $I^i := (-\rho_i, \rho_i)$ for $i \in \{0,1\}$, and $\hF_0 := \Psi_{c_1}\circ F \circ \Psi_{c_0}^{-1}$.

The curve
$$
\hF_0(I^0 \times \{0\}) = \Psi_{c_1}(W^c(c_1))
$$
is a quadratic vertical graph that has a unique point of genuine vertical tangency given by
$$
v_0 = (0, 0) = \Psi_{c_1}(c_1) \in W^{ss}_{\loc}(c_1) \cap W^c(c_1).
$$
Consequently, for $|y| < \rho_0$, the curve $\hF_0(I^0\times\{y\})$ is a quadratic vertical graph that has a unique point $v_y$ whose tangent vector is genuine vertical. Moreover, there exist an interval $J^1 \subset I^1$ containing $0$; and $C^r$-maps $g_0 : I^0 \to I^0$ and $g_1 : J^1 \to I^1$ such that:
$$
\hF_0(g_0(y), y) = v_y = (u, g_1(u)).
$$

The vertical graph $\Gamma^v(g_0)$ can be straightened to $\{0\}\times I^0$ by a map of the form
$$
\Phi_{0, 0}(x,y) = (\phi_{0, 0}(x, y), y).
$$
Likewise, the horizontal graph $\Gamma^h(g_1)$ can be straightened to $J^1\times\{0\}$ by a map of the form
$$
\Phi_{1, 0}(u,w) = (u, \phi_{1, 0}(u,w)).
$$
Since the derivatives $D\Phi_{0, 0}$ and $D\Phi_{1, 0}$ preserve the horizontal and vertical directions respectively, we see that the former preserves horizontal tangencies, while the latter preserves vertical tangencies.

Observe that $F_{0, 0} := \Phi_{1,0} \circ \hF_0 \circ \Phi_{0,0}^{-1}$ maps $\{0\}\times I^0 \subset \bbL^v_0$ to $J^1 \times \{0\} \subset \bbL^h_0$. Define
$$
\psi_{1,1}(y) := \pi_h \circ F_{0,0}(0, y)
\matsp{for}
y \in I^0;
$$
and let
$$
\phi_{1,1}(u) := \lambda\psi_{1,1}^{-1}(u)
\matsp{for}
u \in J^1.
$$
Extend $\phi_{1,1}$ to a $C^r$-diffeomorphism of $I^1 \supset J^1$, and let $\Phi_{1, 1}(u,w) = (\phi_{1,1}(x), y)$. Then for $F_{0,1} := \Phi_{1,1} \circ F_{0,0}$, we have $F_{0,1}(0,y) = (-\lambda y,0)$.

The collection $\{F_{0,1}(I^0\times\{y\})\}_{y \in I^0}$ forms a $C^r$-foliation of $F_{0,1}(U_{c_0})$ by quadratic vertical curves in $U_{c_1}$. Extend this foliation to a $C^r$-foliation $\{\gamma^v_z\}_{z\in \hJ^1}$ of a domain $\hU_{c_1} \supset U_{c_1}$ by quadratic vertical curves, so that $\hU_{c_1} \cap \bbL^h_0 = \hJ^1$, and for $(u, w) \in U_{c_1}$, there exists a unique value $\tau(u, w) \in \hJ^1$ such that $(u, w) \in \gamma^v_{\tau(u,w)}$. Define
$$
\Phi_{1, 2}(u,w) := (u, \phi_{1,2}(u,w))
\matsp{for}
(u,w) \in U_{c_1},
$$
where
$$
\phi_{1,2}(u,w):= 
\left\{
\begin{array}{cl}
\sqrt{u-\tau(u,w)} &: w \geq 0\\
-\sqrt{u-\tau(u,w)} &: w < 0;
\end{array}
\right.
$$
and let $F_{0, 2} := \Phi_{1,2} \circ F_{0, 1}$. Then for $y \in I^0$, we have $F_{0,2}(I^0\times\{y\}) \subset \bbP^v_{-\lambda y}$.

Finally, let
$$
\Phi_{0, 3}(x,y) = (\phi_{0,3}(x,y), y)
$$
be the map which straightens the vertical foliation given by $\{F_{0,2}^{-1}(I^1\times\{w\})\}_{w \in I^1}$ to the genuine vertical foliation. Since $\Phi_{0,3}$ preserves the genuine horizontal foliation, the image of a genuine horizontal line under $F_{c_0} := F_{0,2} \circ \Phi_{0,3}^{-1}$ is the same as $F_{0,2}$. It is now straightforward to check that
$$
\Phi_{c_0} := \Phi_{0, 3} \circ \Phi_{0, 0} \circ \Psi_{c_0}
\matsp{and}
\Phi_{c_1} := \Phi_{1, 2} \circ \Phi_{1, 1} \circ \Phi_{1, 0} \circ \Psi_{c_1}
$$
provide the desired regular charts near $c_1$ and $c_0$ respectively.
\end{proof}


\subsection{Closest returns for regular critical orbits}

Let $c_1\in \Lambda$ be an $\epsilon$-regular critical value. Let $\delta \in (\bepsilon, 1)$ with $\bdelta < 1$.

\begin{prop}\label{no crit in crit nbh}
Let
$
\{\Phi_{c_m} : \cU_{c_m} \to U_{c_m}\}_{m \in \bbZ}
$
be a uniformization of the orbit of $c_0$ given in \thmref{unif reg crit}. Then the regular radii can be chosen uniformly small so that $c_0 \not\in \cU_{c_m} \setminus \{c_m\}$ for $m \in \bbZ$.
\end{prop}

\begin{proof}
Let 
$
\{\Phi_{c_m} : \hcU_{c_m} \to \hU_{c_m}\}_{m \in \bbZ}
$
be a starting uniformization of the orbit of $c_0$. Recall that $\hU_{c_m} := \bbB(k \lambda^{-\bepsilon |m|})$ for some uniform constant $k \in (0,1)$. Let $l \in (0,k)$, and denote
$$
U_{c_m} := \bbB(l \lambda^{\bepsilon |m|})
\matsp{and}
\cU_{c_m} := \Phi_{c_m}^{-1}(U_{c_m}).
$$

Suppose towards a contradiction that $c_0 \in \cU_{c_m} \setminus \{c_m\}$. Note that by choosing $l$ small, we can assume that $|m|$ is uniformly large. Let $1 \leq M, N \leq |m|$ be the largest values for which we have $c_i \in \hcU_{c_{m+i}}$ for $-M< i <N$. Arguing in a similar way as in the proof of \propref{no crit in pesin reg nbh}, we obtain
$$
M > \min\{(1+\bepsilon)^{-1}\logl(l/k), |m|\},
$$
and
$$
N > \min\{\bepsilon^{-1}\logl(l/k), |m|\}.
$$
Thus, by choosing $l$ sufficiently smaller than $k$, the result follows from \propref{no crit in reg nbh}.
\end{proof}

\begin{prop}[Slow recurrence of the critical orbit]\label{return crit disc}
There exists a uniformly small constant $r >0$ such that the following holds. If
$$
\bbD_{c_m}(r\lambda^{\delta |m|}) \cap \bbD_{c_0}(r) \neq \varnothing
$$
for some $m \in \bbZ$, then
$$
|m| > \frac{1}{\bepsilon}\logl \left(\dist(c_m, c_0)\right).
$$
\end{prop}

\begin{proof}
For $m \in \bbZ$, let $\Phi_{c_m}$, $\hcU_{c_m}$ and $\cU_{c_m}$ be the objects defined in the proof of \propref{no crit in crit nbh}. Then
$$
\dist(c_m, c_0) > \radius(\cU_{c_m}) > l_{c_0}\lambda^{\bepsilon |m|}.
$$
Let $r := \underline{l_{c_0}}$. Then we have
$$
\dist(c_m, c_0) < r(1+\lambda^{\delta |m|}) < \underline{l_{c_0}}.
$$
Thus,
$$
|m| > \bepsilon^{-1}\left(\logl \left(\dist(c_m, c_0)\right)- \logl l_{c_0}\right) = \bepsilon^{-1}\logl\left( \dist(c_m, c_0)\right).
$$
\end{proof}

\begin{lem}\label{no pliss near crit}
There exists a uniform constant $r>0$ such that the following holds. Let $0 \leq n < m$. If $\dist(c_{\pm m}, c_0) < r\lambda^{\delta n}$, then we have $m > \bepsilon^{-1} n$. Moreover, for all $k \in \bbZ$ such that $|k| \leq \bepsilon^{-1} n$, the value $\pm m + k$ is not a Pliss moment in the forward/backward orbit of $c_{1/0}$.
\end{lem}

\begin{proof}
We prove the claim for $c_m$. The proof of the statement for $c_{-m}$ is identical.

The first claim follows immediately from \propref{return crit disc}. Denote $p_0 := c_m$, and let 
$$
\{\hPhi_{p_l} : \hcU_{p_l} \to \hU_{p_l}\}_{-m < l < \infty}
$$
be a linearization of the $(m-1, \infty)$-orbit of $p_0$. Then for $k \in \bbZ$ such that $|k| \leq n$, we have
$$
\diam(\hcU_{p_k}) \asymp \lambda^{\bepsilon k} > r\lambda^{\epsilon' n}.
$$
Hence, $c_0 \in \hcU_{p_k}$. Arguing as in the proof of \propref{no crit in crit nbh}, we obtain a contradiction.
\end{proof}


\section{Regularity Near a Critical Orbit}
\label{sec:reg near crit orb}

Let $F$ be the diffeomorphism considered in \secref{sec:quant pesin}. Let $\eta, \epsilon, \delta \in (0, 1)$ be small constants satisfying $\baeta < \epsilon$, $\bepsilon < \delta$ and $\bdelta < 1$. Suppose that $F$ is $\eta$-homogeneous on $\Lambda$. Moreover, suppose $F$ has an $\epsilon$-regular critical value $c_1 \in \Lambda$ of quadratic type with the critical direction $E_{c_1}^* \in \bbP^2_{c_0}$. Let
$$
\{\Phi_{c_m} : (\cU_{c_m}, c_m) \to (U_{c_m}, 0)\}_{m \in \bbZ}
$$
be the uniformization of the orbit of $c_0$ given in  \thmref{unif reg crit}.

\begin{prop}[Recovery of forward regularity]\label{for reg rec}
Let $p_0 \in \cU_{c_0}$ and $\tiE_{p_0}^v \in \bbP^2_{p_0}$. Write $z_0 = (x_0, y_0) := \Phi_{c_0}(p_0)$. Let $-n \leq 0$ be a Pliss moment in the backward orbit of $c_0$. If
$p_{-i} \in \cU_{c_{-i}}$ for $0\leq i \leq n$, and we have
$$
\measuredangle(\tiE_{p_0}^v, E^v_{p_0}) > k|x_0|
\matsp{and}
n > (1+\bepsilon)\delta^{-1}\logl |x_0|,
$$
for some uniform constant $k \in (0,1)$, then $p_{-n}$ is $n$-times forward $(O(1), \delta)_v$-regular along $\tiE_{p_{-n}}^v$.
\end{prop}

\begin{proof}
Denote
$$
\tiE^v_{z_0} := D\Phi_{c_0}(\tiE^v_{p_0})
\matsp{and}
\theta_{-i} := \measuredangle(\tiE^v_{z_{-i}}, E_{z_{-i}}^{gh})
\matsp{for}
0 \leq i \leq n.
$$
By \thmref{reg chart} ii) and \propref{loc linear}, there exist uniform constants $C, c > 0$ such that if $\theta_{-i} < C$, then
$$
\theta_{-i} > c\lambda^{-(1-\bepsilon)i}\theta_0,
$$
and if $\theta_{-i} > C$, then
$$
|\theta_{-i - 1} - \pi/2| < \lambda^{(1-\bepsilon)} |\theta_{-i} - \pi/2|. 
$$

Let $-j \leq 0$ be the last moment such that $\theta_{-j} < C$. Observe 
$$
C > \theta_{-j} > ck\lambda^{-(1-\bepsilon)j}|x_0|.
$$
This implies
$$
j < (1+\bepsilon)\logl |x_0| < (1-\bepsilon)\delta n.
$$

Observe that for $-n \leq m \leq -j$, we have 
$$
\measuredangle\left(\tiE^v_{p_m}, E^v_{p_m}\right) < \nu
$$
for some sufficiently small $\nu > 0$. Thus, by \propref{loc linear} and the fact that $-n$ is Pliss, we obtain
$$
\|DF^l|_{\tiE^v_{p_{-n}}}\| < K\lambda^{(1-\bepsilon) l}
\matsp{for}
0\leq l \leq n-j,
$$
where $K > 0$ is a uniform constant. The result now follows from the fact that
$$
\lambda^{(1-\bepsilon) (n-j)} < \lambda^{(1-\delta) n}.
$$
\end{proof}

\begin{prop}[Obstruction to forward regularity]\label{for reg no rec}
There exists a uniform constant $L_0 \geq 1$ such that the following holds. For $L \geq L_0$, denote $L' := L^{2+\bdelta}$. Let $p_1 \in \bbD_{c_1}(1/L')$; $\tiE_{p_1}^v \in \bbP^2_{p_1}$ and $n \in \bbN$. If
$p_{-i} \in \cU_{c_{-i}}$ for $0\leq i \leq n$, and we have
$$
\measuredangle(\tiE_{p_1}^v, E^v_{p_1}) < 1/L'
\matsp{and}
-(1+\bdelta)\logl L < n <  -\frac{1-\bdelta}{\delta}\logl L,
$$
then $p_{-n}$ is not $n$-times forward $(L, \delta)_v$-regular along $\tiE_{p_{-n}}^v$.
\end{prop}

\begin{proof}
Denote
$$
z_{-i} := \Phi_{c_{-i}}(p_{-i})
\comma
\tiE_{z_{-i}}^v := D\Phi_{c_{-i}}(\tiE_{p_{-i}}^v)
\matsp{and}
\theta_{-i} := \measuredangle(\tiE_{z_{-i}}^v, E_{z_{-i}}^{gh}).
$$
By Theorems \ref{unif reg crit} and \ref{reg chart} ii), and \propref{loc linear}, we have
$$
\theta_{-i} < K\lambda^{-(1+\bepsilon)i}/L'
$$
for some uniform constant $K \geq 1$.

If $\theta_{-i} < \lambda^{\bepsilon i}\nu$, where $\nu >0$ is a sufficiently small uniform constant, then
$$
\measuredangle(\tiE^v_{p_{-i}}, E^h_{p_{-i}}) < \bar\nu.
$$
Thus, by \propref{loc linear}, we obtain
$$
\|DF^{-i}|_{\tiE^v_{p_0}}\| < \lambda^{-\bepsilon i}.
$$

Let $j >0$ be the smallest integer such that
$$
K\lambda^{-(1+\bepsilon)j}/L' > \lambda^{\bepsilon j}\nu.
$$
Simplifying, we obtain
\begin{equation}\label{eq:for reg no rec}
j > -(1-\bepsilon)\logl L'.
\end{equation}

Suppose $n < j$. By the assumed lower bound on $n$, we have
$
L^{-1} > \lambda^{(1-\bdelta)n}.
$
It follows that
$$
\|DF^n|_{\tiE^v_{p_{-n}}}\| > \lambda^{\bepsilon n} > L\lambda^{(1-\delta)n}.
$$
Hence, $p_{-n}$ is not $n$-times forward $(L, \delta)_v$-regular along $\tiE^v_{p_{-n}}$.

Suppose $n > j$. Observe that
$$
\|DF^n|_{\tiE^v_{p_{-n}}}\| > \lambda^{\bepsilon j}\lambda^{(1+\eta)(n-j)}=\lambda^{-(1-\bepsilon)j}\lambda^{(\delta + \eta)n}\lambda^{(1-\delta)n}.
$$
By \eqref{eq:for reg no rec} and the assumed upper bound on $n$, we have
$$
\lambda^{-(1-\bepsilon)j}\lambda^{(\delta + \eta)n}> (L')^{(1-\bepsilon)^2} L^{-\frac{(1-\bdelta)(\delta+\eta)}{\delta}} > L.
$$
Hence, we again see that $p_{-n}$ is not $n$-times forward $(L, \delta)_v$-regular along $\tiE^v_{p_{-n}}$.
\end{proof}

\begin{prop}[Recovery of backward regularity]\label{back reg rec}
Let $p_1 \in \cU_{c_1}$ and $\tiE_{p_1}^h \in \bbP^2_{p_1}$. Write $z_1 = (x_1, y_1) := \Phi_{c_1}(p_1).$ Let $n \geq 1$ be a Pliss moment in the forward orbit of $c_1$. If
$p_i \in \cU_{c_i}$ for $1\leq i \leq n$, and we have
$$
\measuredangle(\tiE_{p_1}^h, E^v_{p_1}) > k|x_1|
\matsp{and}
n > (1+\bepsilon)\delta^{-1}\logl |x_1|
$$
for some uniform constant $k \in (0,1)$, then $p_n$ is $n$-times backward $(O(1), \delta)_h$-regular along $\tiE_{p_n}^h$.
\end{prop}

\begin{proof}
Denote
$$
\tiE_{z_1}^h := D\Phi_{c_1}(\tiE_{p_1}^h)
\matsp{and}
\theta_i := \measuredangle(\tiE_{z_i}^h, E_{z_i}^{gv})
\matsp{for}
1 \leq i \leq n.
$$
By \thmref{reg chart} ii), \propref{loc linear} and \lemref{hor near inv}, there exist uniform constants $C, c > 0$ such that if $\theta_i < C$, then
$$
\theta_i > c\lambda^{-(1-\bepsilon)i}\theta_1,
$$
and if $\theta_i > C$, then
$$
|\theta_{i+ 1}-\pi/2| < \lambda^{(1-\bepsilon)} |\theta_i-\pi/2| + \lambda^{(1-\bepsilon)i}.
$$

Let $j \geq 1$ be the last moment such that $\theta_j < C$. Observe 
$$
C > \theta_j > ck\lambda^{-(1-\bepsilon)j}|x_1|.
$$
This implies
$$
j < (1+\bepsilon)\logl |x_1| < (1-\bepsilon)\delta n.
$$

Observe that for $j \leq m \leq n$, we have 
$$
\measuredangle\left(\tiE^h_{p_m}, E^h_{p_m}\right) < \nu
$$
for some sufficiently small $\nu > 0$. Thus, by \propref{loc linear} and the fact that $n$ is Pliss, we obtain
$$
\|DF^{-l}|_{\tiE^h_{p_n}}\| < K\lambda^{-\bepsilon l}
\matsp{for}
0\leq l < n-j,
$$
where $K > 0$ is a uniform constant. Thus,
$$
\|DF^{-n+j -l}|_{\tiE^h_{p_n}}\| < K\lambda^{-\bepsilon (n-j)}\lambda^{-(1+\eta)l}
\matsp{for}
0\leq l < j.
$$
The result now follows from the fact that
$$
\lambda^{-(1+\eta)j} < \lambda^{-(1-\bepsilon)\delta n}.
$$
\end{proof}

\begin{prop}[Obstruction to backward regularity]\label{back reg no rec}
There exists a uniform constant $L_0 \geq 1$ such that the following holds. For $L \geq L_0$, denote $L' := L^{2+\bdelta}$. Let $p_1 \in \bbD_{c_1}(1/L')$; $\tiE_{p_1}^h \in \bbP^2_{p_1}$ and $n \in \bbN$. If
$p_i \in \cU_{c_i}$ for $1\leq i \leq n$, and we have
$$
\measuredangle(\tiE_{p_1}^h, E^v_{p_1}) < 1/L'
\matsp{and}
-(1+\bdelta)\logl L < n <  -\frac{1-\bdelta}{\delta}\logl L,
$$
then $p_n$ is not $n$-times backward $(L, \delta)_h$-regular along $\tiE_{p_n}^h$.
\end{prop}

For $p \in \bbR^2$ and $t > 0$, let
$$
\bbD_p(t) := \{q \in \bbR^2 \; | \; \dist(q, p)< t\}.
$$

\begin{prop}\label{measure in crit disk}
For $t > 0$, denote
$$
\Lambda_{c_0}^t := \Lambda \cap \bigcup_{n=0}^\infty\bbD_{c_{- n}}(t\lambda^{\delta n}).
$$
Then we have
$
\mu\left(\Lambda_{c_0}^t\right) \to 0
$ as $
t \to 0.
$
\end{prop}

\begin{proof}
Fix a sufficiently small constant $t_0 > 0$ such that
$$
\bbD_{c_{-n}}(t_0\lambda^{\delta n}) \subset \cU_{c_{-n}}
\matsp{for all}
n \geq 0.
$$
Let $K \geq 1$. We claim that there exists $N = N(K)\in \bbN$ such that if $p_0 \in \Lambda \cap \bbD_{c_{-n}}(t_0\lambda^{\delta n})$ for some $n \geq N$, then $p_0$ is not infinitely forward $(K, \udelta)$-regular.

Suppose towards a contradiction that $p_0$ is forward regular. Then by \propref{decay reg}, $p_{n+1}$ is infinitely forward $(K\lambda^{-\delta n}, \udelta)$-regular. Write $z_0 =(x_0, y_0):= \Phi_{c_0}(p_0)$. By \thmref{reg chart}, \propref{loc linear} and $\eta$-homogeneity, we have
$$
|x_0| < t_0\lambda^{(\delta -\eta)n}
\matsp{and}
|y_0| < t_0\lambda^n.
$$
By \thmref{unif reg crit},
$$
\dist(p_{n+1}, c_1) < \lambda^{(\delta -\baeta)n}.
$$
It follows that $p_{n+m} \in \cU_{c_m}$ for all $1\leq m \leq M$, where $M \in \bbN$ is the smallest integer such that
$$
\lambda^{(\delta -\baeta)n}\lambda^{-\eta M} < \lambda^{\bepsilon M}.
$$
This implies that $M > n/\bepsilon$. By \propref{vert angle shrink}, we have
$$
\measuredangle(E^{ss}_{p_{n+1}}, E^v_{p_{n+1}}) < \lambda^{n/\bepsilon} < \dist(p_{n+1}, c_1)
$$
for all $n$ sufficiently large. Observe that
$$
\frac{1-\bdelta}{(2+\bdelta)\udelta n}\logl (\dist(p_{n+1}, c_1)) \to \infty
\matsp{as}
n \to \infty.
$$
Thus, by \propref{for reg no rec}, $p_0$ is not infinitely forward $(K, \udelta)$-regular for any finite constant $K$.

For $t \in (0, t_0)$, let $N(t)\in \bbN \cup\{\infty\}$ be the largest integer such that for all $q \in \Lambda_{c_0}^t$, there exists $n = n(t, q) \geq N(t)$ such that $q \in \bbD_{c_{-n}}(t_0\lambda^{\delta n})$. Clearly, $N(t) \to \infty$ as $t \to 0$. Thus, by above, we have
$$
\Lambda_{c_0}^t \cap \Lambda^+_{K(t), \udelta} = \varnothing,
$$
with $K(t) \to \infty$ as $t \to 0$. The result now follows immediately from \propref{classic pesin}.
\end{proof}


\subsection{Critical tunnels and valuable crescents}\label{subsec:crit tunnel}

Let ${t} \in (0, l_{c_0}]$. Denote
$$
\delta > \delta' := (1- \bdelta) \delta > \delta'' := (1- \bdelta) \delta'.
$$
For $n \geq 0$, let $j_n \geq 0$ be the largest integer less than $n$ such that $-j_n$ is a Pliss moment in the backward orbit of $c_0$. Define the {\it ($1/\delta''$-pinched, $\delta'$-truncated) critical tunnel at $c_{-n}$ (of length ${t}$)} as
$$
\cT_{c_{-n}} = \cT_{c_{-n}}({t}) := F^{-n}\left(\cT_{c_0}^{1/\delta'', \delta' j_n}({t})\right)
$$
(see \subsecref{subsec:trunc pinch}). Note that $F(\cT_{c_{-n-1}}) \subseteq \cT_{c_{-n}}$, and the equality holds unless
$$
n = j_n > \frac{\logl {t}}{\delta'}.
$$

Similarly, for $n \geq 1$, let $i_n \geq 1$ be the largest integer less than $n$ such that $i_n$ is a Pliss moment in the forward orbit of $c_1$. Define the {\it ($1/\delta''$-pinched, $\delta'$-truncated) valuable crescent at $c_n$ (of length ${t}$)} by
$$
\cT_{c_n} = \cT_{c_n}({t}) := F^{n-1}\left(\cT_{c_1}({t})\cap \cU_{c_1}^{\delta' i_n}\right)
\matsp{where}
\cT_{c_1}({t}) := F(\cT_{c_0}({t})).
$$
Note that $F^{-1}(\cT_{c_n}) \subseteq \cT_{c_{n-1}}$, and the equality holds unless
$$
n = i_n > \frac{2\logl {t}}{\delta'}.
$$

\begin{lem}\label{diff length same tunnel}
For ${t} \in (0, l_{c_0}]$, let
$$
N := \frac{(1+\bepsilon)\logl {t}}{\delta'}.
$$
Then
$$
\cT_{c_{-n}}({t}) = \cT_{c_{-n}}(l_{c_0})
\matsp{for}
n > N
$$
and
$$
\cT_{c_{1+n}}({t}) = \cT_{c_{1+n}}(l_{c_0})
\matsp{for}
n > 2N.
$$
\end{lem}

\begin{prop}\label{tunnel fit}
For $m \in \bbZ$, we have $\cT_{c_m} \subset \cU_{c_m} \setminus \{c_m\}.$
\end{prop}

\begin{proof}
The result follows immediately from \lemref{trunc neigh fit} (for $m \geq 1$) and \lemref{pinch neigh fit} (for $m \leq 0$).
\end{proof}

\begin{prop}\label{tunnel size}
For $m \in \bbZ$, we have
$
\diam(\cT_{c_m}) < \lambda^{\delta |m|}.
$
\end{prop}

\begin{proof}
Suppose $m = -n \leq 0$. By \propref{tunnel fit}, we have
$$
\cT_{c_{-n}} = \Phi_{c_{-n}}^{-1}\circ F_{c_0}^{-n}\left(\cT_{c_0}^{1/\delta'', \delta' j_n}({t})\right).
$$
Let $q_{-n} \in \cT_{c_{-n}}$. Write
$
z_{-n} =(x_{-n}, y_{-n}) := \Phi_{c_{-n}}(q_{-n}).
$
By \thmref{reg chart} ii) and \propref{loc linear}, we have
$$
|x_{-n}| < \lambda^{-\bepsilon n} |x_0| < \lambda^{-\bepsilon n} \lambda^{\delta' n} < \lambda^{(1-\bepsilon)\delta' n}.
$$
Moreover,
$$
|y_{-n}| < \lambda^{-(1+\bepsilon) n} |y_0| < \lambda^{-(1+\bepsilon) n}|x_0|^{1/\delta''} < \lambda^{-(1+\bepsilon) n}\lambda^{(1+\bdelta) n} <\lambda^{\bdelta n}.
$$
The result now follows from the fact that
$$
\|\Phi_{c_{-n}}^{-1}\|_{C^r} = O(\lambda^{-\bepsilon n}).
$$

A similar argument proves the case when $m \geq 1$.
\end{proof}

\begin{prop}[Invariance of tunnels and crescents]\label{tunnel inv}
Let $m \in \bbZ$, and $n > |m|$. Then the following statements hold.
\begin{enumerate}[i)]
\item If
$$
F^{-1}(\cT_{c_{-n+1}}) \cap \cT_{c_m} \neq \varnothing,
$$
then
$$
n > \bepsilon^{-1} |m|
\matsp{and}
\cT_{c_{-n}} = F^{-1}(\cT_{c_{-n+1}}).
$$
\item If
$$
F(\cT_{c_{n-1}}) \cap \cT_{c_m} \neq \varnothing,
$$
then
$$
n > \bepsilon^{-1} |m|
\matsp{and}
\cT_{c_n} = F(\cT_{c_{n-1}}).
$$
\end{enumerate}
\end{prop}

\begin{proof}
Suppose
$$
F^{-1}(\cT_{c_{-n+1}}) \cap \cT_{c_{-m}} = \varnothing
$$
for some $0 \leq m < n$.

Note that
$$
\dist(c_{-n+m}, c_0) < \diam(\cT_{c_0} \cap \cU_{c_0}^{\bdelta m})+\lambda^{-\eta m}\diam(F^{-1}(\cT_{c_{-n+1}})) < \lambda^{\bdelta m}.
$$
The result now follows immediately from \lemref{no pliss near crit}.

The other cases can be proved in a similar way.
\end{proof}

If $p \in \cT_{c_{-n}}({t})$ for some $n \geq 0$, then the {\it ${t}$-critical depth of $p$} is defined as the largest integer $0 \leq d_- := \cd_{t}(p) \leq \infty$ such that $p \in \cT_{c_{-d_-(p)}}({t}).$ Similarly, if $q \in \cT_{c_n}({t})$ for some $n \geq 1$, then the {\it ${t}$-valuable depth of $p$} is defined as the largest integer $1 \leq d_+ := \vd_{t}(p) \leq \infty$ such that
$p \in \cT_{c_{d_+(p)}}({t}).$

Define the {\it critical dust of $c_0$} as
$$
\Delta^c_{c_0} := \bigcap_{n = 0}^\infty \bigcup_{i=n}^\infty \cT_{c_{-i}}({t}) = \{p \in \Lambda \; | \; \cd_{t}(p) = \infty\}.
$$
Similarly, define the {\it valuable dust of $c_1$} as
$$
\Delta^v_{c_1} := \bigcap_{n = 1}^\infty \bigcup_{i=n}^\infty \cT_{c_i}({t}) = \{p \in \Lambda \; | \; \vd_{t}(p) = \infty\}.
$$
Lastly, define the {\it full critical dust of $c_0$} as
$$
\Delta^{fc}_{c_0} := \Delta^c_{c_0} \cup \Delta^v_{c_1}.
$$

\begin{prop}\label{crit dust infty depth}
We have
$$
\Delta^c_{c_0} = \bigcap_{{t} > 0}\{p \in \Lambda \; | \; \cd_{t}(p) = \infty\}
$$
and
$$
\Delta^v_{c_1} = \bigcap_{{t} > 0}\{p \in \Lambda \; | \; \vd_{t}(p) = \infty\}.
$$
\end{prop}

\begin{proof}
The result follows immediately from \lemref{diff length same tunnel}.
\end{proof}

\begin{prop}\label{inv dust}
The sets $\Delta^{c/fc}_{c_0}$ and $\Delta_{c_1}^v$ are totally invariant. Consequently, $\mu(\Delta_{c_0}^{fc}) = 0$.
\end{prop}

\begin{proof}
Let $p \in \Delta^c_{c_0}$. If $p \in \cT_{c_{-n}}$ with $n > 0$, then $p \in \cT_{c_{-n+1}}$. Thus,
$F(\Delta^c_{c_0}) \subset \Delta^c_{c_0}.$
Let $k \geq 0$ be the smallest integer such that $p \in \cT_{c_{-k}}$. Then by \propref{tunnel inv}, we have
$F^{-1}(\cT_{c_{-n}}) = \cT_{c_{-n-1}}.$
Thus,
$F^{-1}(\Delta^c_{c_0}) \subset \Delta^c_{c_0}.$
The total invariance of $\Delta^v_{c_1}$ can be proved in a similar way.
\end{proof}

The {\it tunnel escape moment} of $p_0 \in \cT_{c_0}$ is the unique integer $-j <0$ such that
$$
p_{-j+1} \in \cT_{c_{-j+1}}
\matsp{and}
p_{-j} \not\in \cT_{c_{-j}}.
$$
Similarly, the {\it crescent escape moment} of $p_1 \in \cT_{c_1}$ is the unique integer $i > 1$ such that
$$
p_{i-1} \in \cT_{c_{i-1}}
\matsp{and}
p_i \not\in \cT_{c_i}.
$$
Write $(x_1, y_1) := \Phi_{c_1}(p_1).$
If $\tiE^v_{p_1}\in \bbP^2_{p_1}$ is $k|x_1|^{1/2}$-vertical in $\cU_{c_1}$ for some sufficiently small uniform constant $k>0$, then we say that $\tiE^v_{p_1}$ is {\it sufficiently stable-aligned in $\cU_{c_1}$}. Similarly, if $\tiE^h_{p_0}\in \bbP^2_{p_0}$ is $k|x_0|^2$-horizontal in $\cU_{c_0}$, then we say that $\tiE^h_{p_1}$ is {\it sufficiently center-aligned in $\cU_{c_0}$}.

\begin{prop}[Forward/backward regularity implies stable/center alignment]\label{stable center align}
Let $p_0 \in \cT_{c_0}$ and $\tiE^{v/h}_{p_0} \in \bbP^2_{p_0}$. Write
$$
z_0 = (x_0, y_0) := \Phi_{c_0}(p_0)
\matsp{and}
\tiE^{v/h}_{z_0} := D\Phi_{c_0}(\tiE^{v/h}_{p_0}).
$$
Let
$$
N > K\logl |x_0|
$$
for some sufficiently large uniform constant $K \geq 1$. If $p_1 \in \cU_{c_1}$ is $N$-times forward $(1, \delta)_v$-regular along $\tiE^v_{p_1}$, then $\tiE^v_{p_1}$ is sufficiently stable-aligned in $\cU_{c_1}$. Similarly, if $p_0$ is $N$-times backward $(1, \delta)_h$-regular along $\tiE^h_{p_0}$, then $E^h_{p_0}$ is sufficiently center-aligned in $\cU_{c_0}$. 
\end{prop}

\begin{prop}\label{tunnel escape}
Let $p_0 \in \cT_{c_0}$. Write
$z_0 = (x_0, y_0) := \Phi_{c_0}(p_0).$
If $\tiE^v_{p_1}$ is sufficiently stable-aligned in $\cU_{c_1}$, then the following statements hold.
\begin{enumerate}[i)]
\item If $-j <0$ is the tunnel escape moment of $p_0$, then $p_{-j}$ is $j$-times forward $(O(1), \delta)_v$-regular along $\tiE_{p_{-j}}^v$.
\item The direction $\tiE^v_{p_{-n}}$ is sufficiently vertical in $\cU_{c_{-n}}$ for all $n > (1+\bepsilon)\logl |x_0|$ such that $p_{-i} \in \cU_{c_{-i}}$ for $i \leq n$.
\item We have $\|DF^n|_{\tiE^v_{p_{-n}}}\| > \lambda^{\bepsilon n}$ for all $0 \leq n < (1-\bepsilon)\logl |x_0|$.
\end{enumerate}
\end{prop}

\begin{proof}
By \lemref{freq pliss back}, we have
$$
j > (1-\bepsilon)\frac{\logl |x_0|}{\delta'} > (1+\bdelta)\delta^{-1}\logl |x_0|.
$$
The first claim now follows from \lemref{tunnel fit} and \propref{for reg rec}.

Write
$$
\tiE^v_{z_0} := D\Phi_{c_0}(\tiE^v_{p_0})
\matsp{and}
\theta^{h/v}_0 := \measuredangle(\tiE^v_{z_0}, E^{gh/gv}_{z_0}).
$$

By \thmref{reg chart} and \propref{loc linear}, there exists a uniform constant $k >0$ such that for $0\leq n < j$, if we have
$\theta^{h/v}_{-n} < k,$
then
$$
\lambda^{\mp (1 \pm \bepsilon)} < \frac{\theta^{h/v}_{-n-1}}{\theta^{h/v}_{-n}} < \lambda^{\mp(1 \pm \bepsilon)}.
$$
Note that $\tiE^v_{p_{-n}}$ is sufficiently horizontal/vertical in $\cU_{c_{-n}}$ if
$\theta^{h/v}_{-n} < l\lambda^{\bepsilon n}$
for some uniform constant $l > 0$. Since
$\theta^h_0 \asymp |x_0|$
we see that for some uniform constant $K \geq 1$, we have
$$
\theta^h_{-n} < K\lambda^{-(1 + \bepsilon) n} |x_0| < l\lambda^{\bepsilon n} \bK \lambda^{-(1+\bepsilon)n }|x_0| < l\lambda^{\bepsilon n}
$$
for
$$
0\leq n < (1-\bepsilon)\logl |x_0|.
$$
The third claim follows from \propref{loc linear}.

Observe that if
$\theta^h_{-n} > k,$
then
$\theta^v_{-n -N} < k$
for some uniform constant $N \geq 0$. Arguing similarly as above, we conclude the second claim.
\end{proof}

\begin{prop}\label{crescent escape}
Let $p_1 \in \cT_{c_1}$. Write
$z_1 = (x_1, y_1) := \Phi_{c_1}(p_1).$
If $\tiE^h_{p_0}$ is sufficiently center-aligned in $\cU_{c_0}$, then the following statements hold.
\begin{enumerate}[i)]
\item If $i > 1$ be the crescent escape moment of $p_1$, then $p_i$ is $i$-times backward $(O(1), \delta)_h$-regular along $\tiE^h_{p_i}$.
\item The direction $\hE^h_{p_n}$ is sufficiently horizontal in $\cU_{c_n}$ for all $n > (1+\bepsilon)\logl |x_1|$ such that $p_j \in \cU_{c_j}$ for $j \leq n$.
\item We have $\|DF^{-n}|_{\hE^h_{p_n}}\| > \lambda^{-(1-\bepsilon) n}$ for all $0 \leq n < (1-\bepsilon)\logl |x_1|$.
\end{enumerate}
\end{prop}

For $i \in \{0, 1\}$, let $\Gamma_i \subset \cU_{c_i}$ be a $C^1$-curve parameterized by its arclength $|\Gamma_i|$. We say that $\Gamma_1$ is {\it sufficiently stable-aligned in $\cU_{c_1}$} if the direction of $\Gamma_1'$ is sufficiently stable-aligned in $\cU_{c_1}$. Similarly, $\Gamma_0$ is {\it sufficiently center-aligned in $\cU_{c_0}$} if the direction of $\Gamma_0'$ is sufficiently center-aligned in $\cU_{c_0}$.


\section{Mild Dissipativity}\label{sec:mild diss}

Let $F$ the diffeomorphism as considered in \secref{sec:quant pesin}. Suppose additionally that $F$ is mildly dissipative. Let $L \geq 1$ and $\epsilon \in (0,1)$ with $\bepsilon < 1$.

A point $p \in \Lambda^+_{L, \epsilon}$ is said to be {\it proper forward regular} if $W^{ss}(p)$ is proper in $\Omega$. In this case, let $W^{ss}_\Omega(p)$ be the connected component of $W^{ss}(p) \cap \Omega$ containing $p$. Lastly, denote the set of properly forward regular points in $\Lambda^+_{L, \epsilon}$ and $\Lambda^P_{L, \epsilon}$ by $\hLambda^+_{L, \epsilon}$ and $\hLambda^P_{L, \epsilon}$ respectively.

\begin{lem}\label{h compact}
The sets $\hLambda^+_{L, \epsilon}$ and $\hLambda^P_{L, \epsilon}$ are compact.
\end{lem}

\begin{proof}
Note that since $(L, \epsilon)$-regularity is a closed condition, the sets $\Lambda^+_{L, \epsilon}$ and $\Lambda^P_{L, \epsilon}$ are compact. The result now follows immediately from \propref{ss c contin}.
\end{proof}

For $p \in \Lambda^P_{L, \epsilon}$, let
$\Phi_p : \cU_p \to U_p$
be a regular chart at $p$. Recall that a regular box at $p$ is given by
$$
\cB_p := \Phi_p^{-1}(\bbB(1/L_1))
\matsp{for some}
L_1 := \bL
$$
(see \subsecref{subsec:reg box}). For $n \geq 0$, let $\cB_p^n$ be the connected component of $\cB_p \cap F^n(\Omega)$ containing $p$. Denote
$B_p^n := \Phi_p(\cB_p^n).$

\begin{prop}[Iterated image of $\partial \Omega$ as top-bottom boundary]\label{enc box}
There exists a uniform constant $M\in \bbN$ with
$$
M > -\logl \overline{L_1}
$$
such that the following holds. Let $p_0 \in \hLambda^P_{(L, \epsilon)}$. Then for $n \geq M$, the set $\partial B_{p_0}^n \cap B_{p_0}$ is the union of two $C^r$ Jordan arcs $\gamma_+^n$ and $\gamma_-^n$. Moreover, $\gamma_\pm$ is the horizontal graph of a $C^r$-function $g^n_\pm : (-1/L_1, 1/L_1) \to \bbR$ such that
$$
\lambda^{(1+\bepsilon)n}< |g^n_\pm(x)| < \lambda^{(1-\bepsilon)n}
\matsp{for}
-1/L_1< x <1/L_1
$$
and
$$
\|(g^n_\pm)'\|, \|(g^n_\pm)''\| < \lambda^{(1-\bepsilon) n}.
$$
\end{prop}

\begin{proof}
Let $\kappa \in (\bepsilon, 1)$ be a uniformly small constant such that $\bkappa < 1$. Let $\{-m_i\}_{i=1}^\infty$ be the decreasing sequence of all negative Pliss moments in the full orbit of $p_0$ such that $p_{-m_i} \in \hLambda^P_{1, \kappa}$. Choose $M = m_i$ with $i$ sufficiently large such that
$M > -\logl \bL,$
and consider
$
q_0 := p_{-M} \in \hLambda^P_{1, \kappa}.
$

Let
$
\{\Phi_{q_m} : \cU_{q_m}\to U_{q_m}\}_{m \in \bbZ}
$
be a linearization of the full orbit of $q_0$. Then
$$
\radius(U_{q_m}) = l_{q_m} \asymp \lambda^{\bkappa |m|}.
$$
For $a >0$ and $n \geq 0$, denote
$$
\cU_{q_m}^n(a) := \Phi_{q_m}^{-1}(U_{q_m}^n(a))
\matsp{where}
U_{q_m}^n(a) := \bbB(al_{q_m}\lambda^{\bkappa n}, l_{q_m}).
$$

By \lemref{h compact}, there exists a uniform constant $N_0 < M$ such that for $n \geq N_0$, the curve
$$
\cC_n := F^n(W^{ss}_\Omega(q_0))
$$
is contained in $W^{ss}_{\loc}(q_n)$. By choosing a new domain $\tiOmega$ such that $\Omega \supset \tiOmega \supset F(\Omega)$ if necessary, we may assume that for some small uniform constant $a_0 \in (0, 1)$, the set $\partial F^{N_0}(\Omega) \cap \cU^0_{q_{N_0}}(a_0)$ is the union of two $C^r$-Jordan arcs $\Gamma_+^0$ and $\Gamma_-^0$. Moreover, by \thmref{reg chart} and \propref{loc linear}, there exists a uniform constant $K \geq 1$ such that the following holds. For $n \geq 0$, the curve
$$
\gamma_\pm^n := \Phi_{q_{N_0+n}}(\Gamma_\pm^n)
\matsp{where}
\Gamma_\pm^n := F^n(\Gamma_\pm^0 \cap \cU_{q_{N_0}}^n(a_0))
$$
is the horizontal graph of a $C^r$-map $g_\pm^n : \bbI(a_0 l_0\lambda^{\bkappa n}) \to \bbR$ satisfying
$$
\|g_\pm\|_{C^r} < K\lambda^{(1-\bkappa)n}
$$
and
$$
g_\pm^n(x) > K^{-1}\lambda^{(1+\bkappa) n}
\matsp{for}
x \in \bbI(a_0l_0 \lambda^{\bkappa n}).
$$

Observe that
$$
\lambda^{\bkappa M}, \lambda^{(1\pm \bkappa) M} < 1/\bL,
\matsp{and}
\|\Phi_{q_M}^{-1}\|_{C^r} < \bL.
$$
Let
$
\{\Phi_{p_m} : \cU_{p_m} \to U_{p_m}\}_{m \in \bbZ}
$
be a linearization of the full orbit of $p_0$. Then
$
\|\Phi_{p_0}\|_{C^r} < \bL.
$
Since
$$
\gamma_\pm = \Phi_{p_0}\circ \Phi_{q_M}^{-1}(\gamma_\pm^{M-N_0}),
$$
the result follows.
\end{proof}

By \propref{enc box}, we can form dynamically meaningful top-bottom boundaries of a regular box by choosing them as subarcs of $F^M(\partial \Omega)$ for some sufficiently large $M \in \bbN$. The following result states that if $F$ is homogeneous, then there is a uniform bound on how large $M$ has to be (so that there is a uniform lower bound on the resulting height of the regular box).

\begin{prop}\label{hom enc box}
If $F$ is $\eta$-homogeneous for some $\eta \in (0, \uepsilon)$, then we can choose $M$ in \propref{enc box} to be
$$
M = -\logl \overline{L_1}.
$$
\end{prop}

\begin{proof}
Consider the sequence of Pliss moments $\{-m_i\}_{i=1}^\infty$ considered in the proof of \propref{enc box}. By \propref{conditional pliss full} and \propref{freq pliss back}, we have
$$
m_1 < -\kappa^{-1}\logl \bL = -\logl \bL,
$$
and if
$$
m_i > -\logl \bL,
$$
then
$$
m_i > (1-\bepsilon)m_{i+1} + \logl \bL.
$$
Hence, we may choose
$$
M = -\logl \overline{L_1}.
$$
\end{proof}

For $p \in \hLambda^P_{L, \delta}$ and $a \in [0, 1]$, define
$$
\bcB_p(a) := \Phi_p^{-1}\left(B_p^M \cap U_p(a)\right),
$$
where $M$ is given in \propref{enc box} (or \propref{hom enc box} if $F$ is homogeneous). Define an {\it enclosed regular box at $p$} as
\begin{equation}\label{eq:pesin rect}
\bcB_p := \bcB_p(1/2),
\end{equation}
and the {\it outer boundary of $\bcB_p$} as
\begin{equation}\label{eq:outer boundary}
\bpartial \bcB_p := \partial \cB^M_p(1) \cap \partial F^M(\Omega).
\end{equation}
A Jordan curve $\Gamma : (0, 1) \to \bcB_p$ is said to be {\it vertically proper in $\bcB_p$} if $\Gamma$ extends continuously to $[0, 1]$, and $\Gamma(0)$ and $\Gamma(1)$ are in distinct components of $\bpartial \bcB_p$. More generally, a Jordan curve $\Gamma$ in $\Omega$ is said to be vertically proper in $\bcB_p$ if $\Gamma \cap \bcB_p(1)$ is.

\begin{prop}\label{proper in box}
Let $p \in \hLambda^P_{L, \epsilon}$. If $q \in \bcB_p$ is $N$-times forward $(L, \epsilon)$-regular for some
$$
N > -\logl \bL,
$$
then $W^v(q)$ is proper in $\bcB_p$.
\end{prop}

\begin{proof}
By \propref{consist ver hor}, $W^v(q)$ is sufficiently vertical in $\cU_p$. Moreover,
$$
\diam(W^v(q)) > \overline{\diam(\bcB_p)}.
$$
The result follows.
\end{proof}

Define
$$
\Lambda^\pm_\epsilon := \bigcup_{L \geq 1} \Lambda^\pm_{L, \epsilon}
\matsp{and}
\Lambda^P_\epsilon := \bigcup_{L \geq 1} \Lambda^P_{L, \epsilon}.
$$
Let $\hLambda^+_\epsilon \subset \Lambda^+_\epsilon$ be the set of all mildly forward regular points in $\Lambda^+_\epsilon$.

\begin{prop}\label{proper for reg}
Suppose that $F|_\Lambda$ is uniquely ergodic. Then we have
$$
\hLambda^+_\epsilon = \Lambda^+_\epsilon
\matsp{and}
\hLambda^P_\epsilon = \Lambda^P_\epsilon.
$$
\end{prop}

\begin{proof}
By replacing $F$ by $F^N$ for some sufficiently large $N\in \bbN$, we may assume that $F$ is $\eta$-homogeneous for some sufficiently small constant $\eta \in (0, \uepsilon)$ (see \propref{get homog}).

Consider the following open set
$$
\bfU := \bigcup_{p \in \hLambda^P_{1, \bepsilon}} \bcB_p.
$$
By \propref{uni erg pos meas},
$$
\alpha := \mu(\bfU \cap \Lambda) > \frac{\epsilon - \eta}{\epsilon+\eta} \simeq 1.
$$

Let $p_0 \in \Lambda^+_{L, \epsilon}$ for some $L \geq 1$. Let $\{n_i\}_{i=1}^\infty$ and $\{m_i\}_{i=1}^\infty$ be the increasing sequences of all moments in the forward orbit of $p_0$ such that $p_{n_i} \in \bfU$ and $p_{m_i} \in \Lambda^+_{1, \bepsilon}$. By unique ergodicity, we have
$$
\lim_{i \to \infty} \frac{i}{n_i} = \alpha.
$$
Thus, by \propref{freq pliss for}, there exist $i, j \geq 1$ such that $n_i = m_j$. The result now follows from \propref{proper in box}.
\end{proof}


\section{Infinite Renormalizability}\label{sec:inf ren}

Let $F$ be the mildly dissipative diffeomorphism considered in \secref{sec:mild diss}. Suppose additionally that $F$ is infinitely renormalizable. Then there exists a nested sequence $\Omega =: \cD^0 \Supset \cD^1 \Supset \ldots$ of Jordan domains and a sequence of natural numbers $1 =: R_0 < R_1 < \ldots$ such that $\cD^n$ is $R_n$-periodic and
$$
r_{n-1} := R_n/R_{n-1} \geq 2
\matsp{for}
n \in \bbN.
$$
Recall that the renormalization limit set is given by
$$
\Lambda := \bigcap_{n=1}^\infty \bigcup_{i=0}^{r_n-1} F^i(D^n).
$$

Consider the standard odometer $S: \bbO \to \bbO$ defined by
$$
\bbO := \prod_{n \in \bbN} (\bbZ / R_n \bbZ),
$$
and
$$
S(a_1, \ldots, a_{n-1}, a_n, a_{n+1}, \ldots) := (0, \ldots, 0, a_n+1, a_{n+1}, \ldots),
$$
where $n \geq 1$ is the smallest integer such that $a_n \neq R_n-1$. Denote
$$
\bbO_0^n:= \{(0, \ldots, 0, a_{n+1}, a_{n+2}, \ldots) \; | \; a_i \in \bbZ/R_i\bbZ \hspace{3mm}\text{for}\hspace{3mm}i \geq n+1\}.
$$
Lastly, let $\mu_0$ be the unique invariant probability measure for $T|_\bbO$. Note that
$$
\mu_0(\bbO_0^n) = \prod_{i=1}^n \frac{1}{R_i}.
$$

\begin{prop}
There exists a unique continuous semi-conjugacy $\pi : \Lambda \to \bbO$ such that
$$
\pi \circ F|_\Lambda = S\circ \pi
$$
and
$$
\pi(D^n) = \bbO_0^n
\matsp{for}
n \in \bbN.
$$
Moreover, for every $x \in \bbO$, the fiber $\pi^{-1}(x)$ is a connected component of $\Lambda$.
\end{prop}

\begin{prop}\label{uni erg}
For $\mu_0$-almost every $x \in \bbO$, the fiber $\pi^{-1}(x) \subset \Lambda$ is trivial. Consequently, $F|_\Lambda$ has a unique invariant probability measure $\mu := \pi^*\mu_0$.
\end{prop}

\begin{proof}
Let $\mu$ be an ergodic probability measure on $\Lambda$. Note that $\mu$ cannot have two negative Lyapunov exponents since it is not supported on a sink. By \propref{classic pesin}, for any $\kappa>0$, there exists $L \geq 1$ such that $\mu(\hLambda^P_{L, \epsilon}) > 1-\kappa$ (recall that if $p \in \hLambda^P_{L,\epsilon}$, then $p$ is Pesin $(L, \epsilon)$-regular, and $W^{ss}_\Omega(p)$ is proper in $\Omega$).

Let $p \in \hLambda^P_{L, \epsilon}$. For $n \geq 0$, let $\cB_p^n$ be the connected component of $\cB_p \cap F^n(\Omega)$ containing $p$, where $\cB_p$ is a regular box at $p$. Suppose $q^n_r, q^n_l \in \hLambda^P_{L, \epsilon} \cap \cB_p^n$ are points located to the right and left of $W^{ss}_\Omega(p_0) \cap \bcB_p^n$ respectively. Then we can form a rectangle $\cB_p^n(q^n_r, q^n_l) \subset \cB^n_p$ enclosed by $W^{ss}_\Omega(q^n_r)$ and $W^{ss}_\Omega(q^n_l)$. If $q^n_r$ and $q^n_l$ are sufficiently close to $p$, then by \propref{enc box}, we have
$$
\diam(\cB_p^n(q^n_r, q^n_l)) = O(\lambda^{(1-\epsilon)n}).
$$
Let $Z_p$ be the connected component of $\Lambda$ containing $p$. If $q^n_r, q^n_l \notin Z_p$ for all $n$, then $Z_p \subset \cB_p^n(q^n_r, q^n_l)$. This implies that $Z_p = \{p\}$.

Let $q \in \hLambda_{L, \epsilon}$, and let $Z_q$ be the connected component of $\Lambda$ that contains $q$. We claim that if $q$ cannot be approached by points in $\hLambda_{L, \epsilon} \setminus Z_q$ from both sides of $W^{ss}_\Omega(q)$, then there exists a rectangle $R_q \subset \cB_q$ containing $q$ such that $R_q \cap \hLambda_{L, \epsilon} \subset Z_q$. The set of all such components $Z_q$ is countable, and the result follows.

If $q$ is isolated in $\hLambda^P_{L, \epsilon}$, or can only be approached from within $\hLambda^P_{L, \epsilon}$ by points contained in $Z_q$, then clearly the claim holds. If $q$ can only be approached by points in $\hLambda^P_{L, \epsilon} \setminus Z_q$ from one side of $W^{ss}_\Omega(q)$, then we can form a rectangle $R_q$ so that one of its sides is contained in $W^{ss}_\Omega(q)$, and $R_q \cap\hLambda_{L, \epsilon} \subset W^{ss}_\Omega(q)$. Since $W^{ss}_\Omega(q) \cap \hLambda_{L, \epsilon} \subset Z_q$, the claim again holds.
\end{proof}

\begin{prop}\label{non hyp}
The unique invariant measure $\mu$ for $F|_\Lambda$ is not hyperbolic. That is, the Lyapunov exponents of $F|_{\Lambda}$ must be $0$ and $\log \lambda$ for some $\lambda \in (0,1)$.
\end{prop}

\begin{proof}
Suppose $\mu$ is hyperbolic. Then it has a period $p \geq 1$. Up to take an iterate, one can assume that $p=1$. Then for almost every points $x, y \in \Lambda$, there exists a homoclinic point $z$ such that
$$
\dist(f^n(x),f^n(z)) \to 0
\matsp{and}
\dist(f^{-n}(y),f^{-n}(z))\to 0
\matsp{as}
n \to \infty.
$$
This is a contradiction, since one can choose $x,y$ to belong to different connected components of $\Lambda$.
\end{proof}


\section{Regular Unicriticality}\label{sec:unicrit}

Let $F$ be the mildly dissipative, infinitely renormalizable diffeomorphism considered in \secref{sec:inf ren}. Let $\eta, \epsilon, \delta \in (0,1)$ be small constants such that
$\baeta < \epsilon$;
$\bepsilon < \delta$ and
$\bdelta<1.$
Suppose that $F$ is $(\delta, \bepsilon)$-regularly unicritical. Let $\{c_m\}_{m\in\bbZ} \subset \Lambda$ be a regular critical orbit with
$
c_1 \in \cD^n
$ for all
$n \in \bbN.
$
Denote
$
\cD^n_i := F^{i-1}(\cD^n)
$ for $
i \in \bbN.
$

By Propositions \ref{uni erg}, \ref{get homog} and \ref{get crit reg}; and \remref{ren unicrit is unicrit}, we can replace $F$ by $F^{R_N}|_{\cD^N}$ for some sufficiently large $N \in \bbN$ such that the following properties hold:
\begin{enumerate}[i)]
\item $F$ is $\eta$-homogeneous on $\Lambda \cap \cD^N$;
\item $c_1 \in \cD^n_1$ is an $\epsilon$-regular critical value for $F$; and
\item $F$ is $(\delta, \bepsilon)$-regular unicritical on $\Lambda \cap \cD^n_1$.
\end{enumerate}
Let
$$
\{\Phi_{c_m} : (\cU_{c_m}, c_m) \to (U_{c_m}, 0)\}_{m \in \bbZ}
$$
be a uniformization of the orbit of the $c_1$ given in \thmref{unif reg crit}. Let $t_0 \in (0,1)$ be a sufficiently small uniform constant so that
$$
\bbD_{c_{-i}}(t_0 \lambda^{\bepsilon i}) \subset \cU_{c_{-i}}
\matsp{for}
i \geq 0.
$$
Let $L_0 := L(t_0) \geq 1$ be the regularity factor given in \defnref{def unicrit} ii).

By \propref{hom transverse back reg}, $c_0$ is infinitely backward $(O(1), \bepsilon)_v$-regular along
$$
E^v_{c_0} := D_0\Phi_{c_0}^{-1}(E^{gv}_0).
$$

\begin{prop}\label{crit vert reg}
Let $t \in (0, t_0)$. After enlarging $L(t) \geq 1$ by a uniform amount if necessary, the following condition holds. If
$$
c_{-n} \notin \bigcup_{i = 0}^{n-1} \bbD_{c_{-i}}({t}\lambda^{\bepsilon i})
$$
for some $n \in \bbN$, then $c_{-n}$ is $n$-times forward $(L(t), \delta)_v$-regular along $E^v_{c_{-n}}$.
\end{prop}

\begin{proof}
By the regular unicriticality condition, $c_{-n}$ is $n$-times forward $(L(t), \delta)_v$-regular along some direction $\tiE^v_{c_{-n}} \in \bbP^2_{c_{-n}}$. Let $0 \leq j \leq n$ be the largest integer such that $-j$ is a Pliss moment in the backward orbit of $c_0$. This means that $c_{-j}$ is $(\infty, j)$-times $(K, \bepsilon)_v$-regular along $E_{c_{-j}}^v$ for some uniform constant $K \geq 1$. By \lemref{freq pliss back}, we have
$
j > (1-\bepsilon) n.
$

It follows from \propref{decay reg} that $c_{-n}$ is $n$-times forward $(K\lambda^{-\bepsilon n}, \bepsilon)_v$-regular along $E_{c_{-k}}^v$. Using \propref{vert angle shrink}, we see that
$$
\measuredangle(E^v_{c_{-n}}, \tiE_{c_{-n}}^v) < \overline{L(t)} \lambda^{-\bepsilon n} \lambda^{(1-\delta)n}.
$$
If
$$
\measuredangle(E^v_{c_{-n}}, \tiE_{c_{-n}}^v) < k \lambda^{(1+\eta)\bepsilon n}
$$
for some uniformly small constant $k >0$, then
$$
\measuredangle(E^v_{c_{-i}}, \tiE_{c_{-i}}^v) < k
\matsp{for all}
j \leq i \leq n.
$$
Thus, if
$
n > C\logl L(t)$
for some uniform constant $C \geq 1$, then we see that $c_{-n}$ is $n$-times forward $(\overline{L(t)}, \delta)_v$-regular along $E^v_{c_{-n}}$.
\end{proof}


\subsection{Critical discs}\label{subsec:crit disc}

Let $\epsilon_i$ with $i \in \{0, 1, 2\}$ be small constants with
$
\epsilon_0 := \epsilon
$ and $
\bepsilon_i < \epsilon_{i+1} < \underline{\delta}
$ for $
0 \leq i < 2
$
such that
$$
\diam(\cU_{c_m}) > {t_0}\lambda^{\epsilon_1 |m|}.
$$
For ${t} \in (0,{t_0})$ and $m \in \bbZ$, the {\it $({t}, \epsilon_2)$-disc at $c_m$} is defined as
$$
\bbD^{t}_{c_m} := \bbD_{c_m}({t}\lambda^{\epsilon_2 |m|}) \setminus \{c_m\} \subset \cU_{c_m} \setminus \{c_m\}.
$$

\begin{lem}\label{shrink disc inv}
For $i \geq 1$, we have $F(\bbD^{t}_{c_{-i}}) \subset \bbD^{t}_{c_{-i+1}}$.
\end{lem}

\begin{proof}
This follows immediately from the $\eta$-homogeneity of $F$.
\end{proof}

Let $p \in \Omega$. If $p \in \bbD^{t}_{c_{-n}}$ for some $n \geq 0$, let $e_n(p) >0$ be the smallest integer such that
$
F^{-e_n(p)}(p) \notin \bbD^{t}_{c_{-n - e_n(p)}},
$
and let $0 \leq \de(p) \leq \infty$ be the largest integer such that
$
p \in \bbD^{t}_{c_{-\de(p)}}.$
If, on the other hand, $p \notin \bbD^{t}_{c_{-n}}$ for all $n \geq 0$, then denote $\de(p) = \varnothing$.

\begin{lem}\label{max escape time}
Let $p \in \bbD^{t}_{c_{-n}}$. Then
$$
e_n(p) < \bepsilon_2^{-1}\logl\left(\dist(p, c_{-n})\right).
$$
\end{lem}

\begin{proof}
By $\eta$-homogeneity, we have
$$
\dist(F^{-i}(p), c_{-n-i}) > \lambda^{\eta i}\dist(p, c_{-n}).
$$
Since
$$
\radius(\bbD^{t}_{c_{-n-i}}) := {t}\lambda^{\epsilon_2 (n+i)},
$$
the claim follows.
\end{proof}

\begin{lem}\label{return crit dist lem}
If $c_{-n} \in \bbD^{t}_{c_0}$ for some $n > 0$, then
$$
\dist(c_{-n}, c_0) > \left(\radius(\bbD^{t}_{c_{-n}})\right)^{\bepsilon_1}.
$$
\end{lem}

\begin{proof}
By \propref{return crit disc}, we have
$$
\lambda^{\bepsilon_1 n} < \dist(c_{-n}, c_0).
$$
The claim now follows from the fact that
$$
\left(\radius(\bbD^{t}_{c_{-n}})\right)^{\bepsilon_1} = {t}^{\bepsilon_1}\lambda^{\epsilon_2 \bepsilon_1 n} < \lambda^{\bepsilon_1 n}.
$$
\end{proof}

\begin{lem}\label{return dist lem}
Let $p_0 \in \bbD^{t}_{c_0}$. If $p_{-n} \in \bbD^{t}_{c_0}$ for some $0< n < e_0(p_0)$, then
$$
\dist(p_{-n}, c_0) > \left(\dist(p_0, c_0)\right)^{\bepsilon_1}.
$$
\end{lem}

\begin{proof}
Note that by \lemref{return crit dist lem}
\begin{align*}
\dist(p_{-n}, c_0) &> \dist(c_{-n}, c_0) - \dist(p_{-n}, c_{-n}) \\
&> \left(\radius(\bbD^{t}_{c_{-n}})\right)^{\bepsilon_1} - \radius(\bbD^{t}_{c_{-n}})\\
&\asymp \lambda^{\epsilon_2\bepsilon_1 n}.
\end{align*}
By $\eta$-homogeneity, we have
$$
\dist(p_0, c_0) < \lambda^{-\eta n}\dist(p_{-n}, c_{-n}) < \lambda^{-\eta n}\radius(\bbD^{t}_{c_{-n}}) \asymp \lambda^{(\epsilon_2 - \eta) n}.
$$
The result follows.
\end{proof}

\begin{prop}\label{full escape}
Let $p_0 \in \bbD^{t}_{c_0}$. If $\de(p_0) < \infty$, then there exists a smallest integer $0 \leq \ee(p) < \infty$ such that $d\left(F^{-\ee(p)}(p)\right) = \varnothing$. Moreover, $\ee(p_0) < (1+\bepsilon_1)e_0(p_{\de(p_0)})$.
\end{prop}

\begin{proof}
Denote
$
q_0 := p_{\de(p_0)} \in \bbD^{t}_{c_0}.
$
Note that $\de(q_0)$ may be greater than $0$.

Consider a sequence $\{n_i\}_{i=0}^j$ defined as follows. Let $n_0 := 0$. For $i \geq 0$, if
$
\ee(q_{-n_i}) > e_0(q_{-n_i}),
$
then let
$$
n_{i+1}' := n_i + e_0(q_{-n_i}),
$$
and let
$
n_i = n_i^0 < n_i^1 < \ldots < n_i^k \leq n_i'
$
be a sequence of moments such that
$
q_{-n_i^l} \in \bbD^{t}_{c_0}
$ for $
0 \leq l \leq k.
$
Define $n_i = n_i^l$, where $0\leq l \leq k$ is the integer such that $n_i^l + e_0(q_{-n_i^l})$ is maximized.

We claim that $n_{i+1} > n_i$. Otherwise, we have
$
q_{-n} \not\in \bbD^{t}_{c_0}
$ for $
n_i < n \leq n_{i+1}',
$
and hence
$
\de(q_{-n_{i+1}'}) > \ee_0(q_{-n_i}).
$
Denote
$$
b_i := n_{i+1}' - \de(q_{-n_{i+1}'}).
$$
Then by \lemref{shrink disc inv}, we have
$
q_{-b_i} \in\bbD^{t}_{c_0}$
and
$$
b_i + \ee_0(q_{-b_i}) > n_{i+1}' = n_i + \ee_0(q_{-n_i}).
$$
If $b_i > n_{k-1} \geq 0$ for some $k \leq i$, then this contradicts the maximality of $n_k + \ee_0(-n_k)$. If $b_i < 0$, then
$$
\de(q_{-n_{i+1}'}) = n_{i+1}' - b_i > n_{i+1}'.
$$
Hence,
$$
\de(p_0) = \de(q_{-\de(p_0)}) \geq \de(q_{-n_{i+1}'}) - n_{i+1}' + \de(p_0) > \de(p_0),
$$
which is a contradiction. Thus, the claim holds.

Let
$
l_i := \dist(c_{-n_i}, c_0).
$
By \lemref{return dist lem}, we have
$$
l_{i+1} >l_i^{\bepsilon_1} > l_0^{\bepsilon_1^{i+1}}.
$$
This immediately implies that $j < \infty$.

Denote
$
m_{i+1} := n_{i+1}-n_i.
$
By $\eta$-homogeneity, we have
$$
{t}\lambda^{\epsilon_2 m_i}=\diam(\bbD^{t}_{c_{-m_i}}) > \dist(c_{-n_i}, c_{-m_i}) >\lambda^{\eta m_i}l_{i-1}.
$$
So
$$
m_{i+1} < \bepsilon_2^{-1} \logl(l_i) < \bepsilon_2^{-1}\bepsilon_1^i\logl(l_0).
$$

Let
$
m_{j+1} := \ee_0(q_{-n_j}).
$
Then by \lemref{max escape time},
$$
m_{j+1} < \bepsilon_2^{-1} l_j < \bepsilon_2^{-1}\bepsilon_1^j\logl(l_0).
$$
Finally, observe that by \propref{return crit disc}, we have
$$
\ee_0(q_0) > n_1 > \bepsilon_1^{-1}\logl(l_0).
$$
Consequently,
$$
\ee(p_0) \leq \ee(q_0) = \sum_{i=0}^{j+1} m_i < (1+\bepsilon_1) \ee_0(q_0)
$$
as claimed.
\end{proof}


\subsection{Pinching at the critical point}

\begin{thm}\label{pinching}
We have
$$
\Lambda \cap \bbD_{c_0}(1/\overline{L_0}) \subset \cT_{c_0}^{1/\bdelta},
$$
where $\cT_{c_0}^{1/\bdelta} \subset \cU_{c_0}$ is a $1/\bdelta$-pinched regular neighborhood at $c_0$.
\end{thm}

\begin{proof}
Let $p_0 \in \Lambda \cap \bbD_{c_0}(1/\overline{L_0})$. Write
$
z_0 = (x_0, y_0) := \Phi_{c_0}(p_0).
$
Suppose towards a contradiction that $p_0 \notin \cT_{c_0}^{1/\bdelta}$. For $m \in \bbZ$, consider the $({t_0}, \bepsilon)$-disc $\bbD_{c_m}^{t_0}$ at $c_m$. By \thmref{reg chart} and \propref{loc linear}, we have
$$
|y_{-i}| > k\lambda^{-(1-\bepsilon)i}|y_0| > k\lambda^{-(1-\bepsilon)i}|x_0|^{1/\bdelta}
\matsp{for}
0 \leq i < \ee_0(p_0)
$$
for some uniform constant $k > 0$. If
$$
k\lambda^{-(1-\bepsilon)i}|x_0|^{1/\bdelta} < {t_0}\lambda^{\bepsilon i}
$$
then it follows by straightforward computation that
\begin{equation}\label{eq:pinching}
i < \bdelta^{-1}\logl|x_0|.
\end{equation}

First, suppose that $\de(p_0) = 0$. By \propref{full escape} and \eqref{eq:pinching}, we have
\begin{equation}\label{eq:pinching 2}
N:= \ee(p_0) < (1+\bepsilon)\ee_0(p_0) < \bdelta^{-1}\logl|x_0| < \infty.
\end{equation}
We also conclude from
$$
\lambda^{-(1-\bepsilon)i} \overline{L_0}^{-1}< k\lambda^{-(1-\bepsilon)i}|y_0| < {t_0}\lambda^{\bepsilon i}
$$
that
\begin{equation}\label{eq:pinching 0}
N > -\logl \overline{L_0}.
\end{equation}
If $p_0 = c_{-n}$ for some $n \in \bbN$, then for $l \in \bbZ$, denote
$
\tiE_{p_l}^v := E^v_{c_{-n+l}}.
$
Otherwise, denote
$
\tiE_{p_l}^v := E^s_{p_l}.
$
In the former case, $p_{-N} = c_{-(n+N)}$ is $(n+N)$-times forward $(L_0, \delta)_v$-regular along $\tiE_{p_{-N}}^v$ by \propref{crit vert reg}. Consequently, by \propref{decay reg}, we see that $p_1 = c_{-n+1}$ is $(n-1)$-times forward $(L_0\lambda^{-\bdelta N}, \delta)_v$-regular along $\tiE_{p_{-n+1}}^v$. Similarly, in the latter case, $p_{-N}$ is infinitely forward $(L_0, \delta)_v$-regular along $\tiE_{p_{-N}}^v$, and thus, $p_1$ is infinitely forward $(L_0\lambda^{-\bdelta N}, \delta)_v$-regular along $\tiE_{p_{-n+1}}^v$.

Let $M \in \bbN$ be the largest integer such that
$
p_{1+i} \in \cU_{c_{1+i}}
$ for $
0 \leq i < M.
$
Then by \thmref{reg chart} and \propref{loc linear}, we have
$
\lambda^{\bepsilon M} < |x_M| < \lambda^{-\bepsilon M}|x_1|.
$
We also have
\begin{equation}\label{eq:pinching 3}
|y_0| < \lambda^{(1-\bepsilon)N}.
\end{equation}

If $|y_0| > |x_0|$, then $|x_1| \asymp |y_0|$. Hence,
$$
\lambda^{\bepsilon M} < \lambda^{-\bepsilon M}|y_0| < \lambda^{(1-\bepsilon)N}.
$$
It follows that $N < \bepsilon M$. If $|y_0| < |x_0|$, then $|x_1| = O(|x_0|)$. So
$$
\lambda^{\bepsilon M} < \lambda^{-\bepsilon M}|x_0| < \lambda^{-\bepsilon M}|y_0|^{\bdelta} < \lambda^{\bdelta N}.
$$
Hence, we again have $N < \bepsilon M$.

Let
$$
E^v_{p_1} := D_{z_1}\Phi_{c_1}^{-1}(E^{gv}_{z_1}).
\matsp{and}
\theta_1 := \left|\measuredangle(\tiE_{p_1}^v, E^v_{p_1})\right|.
$$
By \propref{vert angle shrink}, we have
$$
\theta_1 < L_0\lambda^{-\bdelta N}\lambda^{(1-\bdelta)M} < L_0\lambda^{\bepsilon^{-1}N} < |y_0|^{1/\bepsilon}.
$$
Moreover, by \thmref{unif reg crit}, we have
$
\theta_0 < O\left(|x_0|+\theta_1\right).
$
If $|y_0| < |x_0|$, then
$
\theta_0 < O\left(|x_0|\right).
$
This, \eqref{eq:pinching 2}, \eqref{eq:pinching 0} and \propref{for reg no rec} contradict the $N$-times forward $(L_0, \delta)_v$-regularity of $p_{-N}$ along $\tiE^v_{p_{-N}}$.

If $|y_0| > |x_0|$, then 
$
\theta_0 < O\left(|y_0|\right).
$
By \thmref{reg chart} and \propref{loc linear}, we have
\begin{equation}\label{eq:pinching 4}
\theta_{-i} < K\lambda^{-(1+\bepsilon)i}\theta_0 < \bK\lambda^{-(1+\bepsilon)i}|y_0|.
\end{equation}
Let $j \geq 0$ be the smallest integer less than $\ee_0(p_0)$ such that 
$
\theta_{-j} > \lambda^{\bepsilon j}\nu.
$
for some small $\nu >0$. By \eqref{eq:pinching 4}, we have
$$
j > (1-\bepsilon)(\logl |y_0| + \logl(\bK)).
$$
Thus, by \eqref{eq:pinching 3}, we see that
$
j > kN
$
for some uniform constant $k > 0$. Finally, by \propref{loc linear} and \eqref{eq:pinching 0}, we have
$$
\|DF^{-N}|_{\tiE^v_{p_0}}\| < \lambda^{-\bepsilon j}\lambda^{-(1+\eta) (N-j)} < \lambda^{-(1-k+\bepsilon)N} = \lambda^{(k-\bdelta)N}\lambda^{-(1-\delta)N} < L_0^{-1}\lambda^{-(1-\delta)N}.
$$
This contradicts the $N$-times forward $(L_0, \delta)_v$-regularity of $p_{-N}$ along $\tiE^v_{p_{-N}}$.

Finally, consider the case when $p_0 \in \bbD_{c_0}(1/\overline{L_0})$, and $\de(p_0) = n > 0$. This means that $c_{-n} \in \bbD^{t_0}_{c_0}$. Write
$
(x_0, y_0) := \Phi_{c_0}(c_{-n}).
$
It follows by \propref{return crit disc} that
$$
\radius(D^{t_0}_{c_{-n}}) < |x_0|^{1/\bepsilon}.
$$
Thus, by slightly increasing $\bdelta$ a uniform amount if necessary, we obtain
$
p_0 \in D^{t_0}_{c_{-n}} \subset \cT_{c_0}^{1/\bdelta}.
$
\end{proof}


\subsection{Refinement of regular unicriticality}\label{sec:refine uni}

Let
$
{t_1} := 1/\overline{L_0} < \underline{t_0}
$
be a sufficiently small uniform constant such that so that
$$
\bbD^{t_1}_{c_m} \subset \cU_{c_m} \setminus \{c_m\}
\matsp{for}
m \in \bbZ
\matsp{and}
\Lambda \cap \bbD^{t_1}_{c_0} \subset \cT_{c_0}^{1/\udelta}.
$$
Let ${L_1} := L({t_1}) \geq 1$ be the regularity factor given in \defnref{def unicrit} ii). By increasing ${L_1}$ if necessary, we may assume that  ${L_1}^{-\epsilon} < {t_1}.$

Denote
$$
\delta > \delta' := (1- \bdelta) \delta > \udelta.
$$
For ${t} \in (0, {t_0}]$ and $n \geq 0$, let $\cT_{c_{-n}}({t})$ and $\cT_{c_{1+n}}({t})$ be the $1/\udelta$-pinched, $\delta'$-truncated critical tunnel and valuable crescent at $c_{-n}$ and $c_{1+n}$ respectively of length ${t}$ (see \subsecref{subsec:crit tunnel}). 

Denote
$$
\hE^v_{c_{-n}} := D_0\Phi_{c_{-n}}^{-1}(E^{gv}_0) = DF^{-n}(E^v_{c_0})
$$
and
$$
\hE^h_{c_{1+n}} := D_0\Phi_{c_{1+n}}^{-1}(E^{gh}_0) = DF^{1+n}(E^v_{c_0}).
$$
Let $p \in \Lambda$. For $L \geq 1$, we say that $p$ is {\it forward $(L, \delta)$-regular} (without specifying time or direction) if either
\begin{itemize}
\item $p = c_{-n}$ for some $n \geq 0$, and $p$ is $n$-times forward $(L, \delta)_v$-regular along $\hE^v_{c_{-n}}$; or
\item $p \neq c_{-n}$ for all $n \geq 0$, and $p$ is infinitely forward $(L, \delta)_v$-regular along $\hE^v_p := E^s_p$.
\end{itemize}
Similarly, we say that $p$ is {\it backward $(L, \delta)$-regular} if either
\begin{itemize}
\item $p = c_n$ for some $n \geq 1$, and $p$ is $n$-times backward $(L, \delta)_h$-regular along $\hE^h_{c_n}$; or
\item $p \neq c_n$ for all $n \geq 1$, and $p$ is infinitely backward $(L, \delta)_h$-regular along $\hE^h_p := E^c_p$.
\end{itemize}

\begin{thm}[Geometric criterion for regularity]\label{ref uni}
For ${t} \in (0, {L_1}^{-1}]$, let $L := 1/\ut$. Then for all $p \in \Lambda$, the following statements hold.
\begin{enumerate}[a)]
\item If
$\displaystyle
p \notin \bigcup_{i = 0}^\infty \cT_{c_{-i}}({t}),
$
then $p$ is forward $(L, \delta)$-regular. If $p$ is also contained in $\cT_{c_1}({t})$, then $p$ is forward $(1, \delta)$-regular.
\item If
$\displaystyle
p \notin \bigcup_{i = 0}^\infty \cT_{c_{1+i}}({t}),
$
then $p$ is backward $(L, \delta)$-regular. If $p$ is also contained in $\cT_{c_0}({t})$, then $p$ is backward $(1, \delta)$-regular.
\item If
$\displaystyle
p \notin \bigcup_{i = -\infty}^\infty \left(\cT_{c_i}({t})\cup \{c_i\}\right)
$
then $p$ is Pesin $(L, \delta)$-regular.
\end{enumerate}
\end{thm}

Let $\epsilon_i$ with $i \in \{0, 1, 2\}$ be constants satisfying
$
\epsilon_0 := \epsilon$ and $
\bepsilon_i < \epsilon_{i+1} < \udelta
$ for $
0 \leq i < 2,
$
so that we have
$$
\diam(\cU_{c_m}) > {t_0}\lambda^{\epsilon_1 |m|}
\matsp{and}
\diam(\bbD^{t_1}_{c_m}) := {t_1} \lambda^{\epsilon_2 |m|}
\matsp{for all}
m \in \bbZ.
$$

Let $p_0 \in \Omega$. If $p_0 \in \bbD^{{t_1}}_{c_{-n}}$ for some $n \geq 0$, let $e_n(p_0) >0$ be the smallest integer such that
$
p_{-e_n(p_0)} \notin \bbD^{{t_1}}_{c_{-n - e_n(p_0)}},
$
and let $0 \leq \de(p_0) \leq \infty$ be the largest integer such that
$
p_0 \in \bbD^{{t_1}}_{c_{-\de(p)}}.
$
In this case, $0 \leq \ee(p_0) < \infty$ is the integer given by \propref{full escape} (taking ${t} = {t_1}$). If, on the other hand, $p_0 \notin \bbD^{t_1}_{c_{-n}}$ for all $n \geq 0$, then denote $\de(p_0) = \varnothing$. 

If $p_0 \in \bbD^{2{t_1}}_{c_{-n}}$ for some $n \geq 0$, let $\hd(p_0), \he_n(p_0) \geq 0$ be the smallest integers so that
$
p_0 \in \bbD^{2{t_1}}_{c_{-\hd(p_0)}}
$ and $
p_{-\he_n(p_0)} \notin \bbD^{2{t_1}}_{c_{-n-\he_n(p_0)}}.
$
Otherwise, denote $\hd(p_0) = \varnothing$.

For ${t} \in (0, {t_1}]$, define the ${t}$-critical depth $\cd_{t}(p_0)$ and the ${t}$-valuable depth $\vd_{t}(p_0)$ of $p_0$ as in \subsecref{subsec:crit tunnel}. Denote
$$
\cT_{c_m} := \cT_{c_m}({t_1})
\matsp{for}
m \in \bbZ;
$$
and
$$
\cd(p_0) := \cd_{{t_1}}(p_0)
\matsp{and}
\vd(p_0) := \vd_{{t_1}}(p_0).
$$

\begin{lem}\label{disc tunnel inv}
Let $p_{-n} \in \bbD^{2{t_1}}_{c_{-n}}$ for some $n \geq 0$. Then one of the following statements hold.
\begin{enumerate}[i)]
\item We have $\cd(p_{-n}) \in \{\varnothing, \infty\}$, and
$
\cd(p_1) = \cd(p_{-n}).
$
\item We have $\cd(p_{-n}) = -n$, and
$
\cd(p_1) = \varnothing.
$
\item We have $\bepsilon_1^{-1}n < \cd(p_{-n}) < \infty$, and
$
\cd(p_1) = \cd(p_{-n}) -n-1.
$
\end{enumerate}
\end{lem}

\begin{proof}
Suppose that $p_1 \in \cT_{c_{-k}}$ for some $k \geq 0$. Then by \propref{return crit disc}, we have $k > \bepsilon_1^{-1}n$. Observe that, for some uniform constant $K \geq 1$, we have
\begin{align*}
\dist(c_{-k+1}, c_0) &< K\dist(c_{-k}, p_1) + \dist(p_0, c_0)\\
&< K\dist(c_{-k}, p_1) + \lambda^{-\eta n}\dist(p_{-n}, c_{-n})\\
&< \lambda^{\bdelta k}+ 2{t_1}\lambda^{(\epsilon_2-\eta) n}\\
&< \lambda^{\bepsilon_2 n}.
\end{align*}
Thus, by \lemref{no pliss near crit}, we have
$$
p_{-n} \in F^{-n-1}(\cT_{-k}) = \cT_{c_{-k-n-1}}.
$$
\end{proof}

\begin{lem}\label{shallow tunnel free}
Let $p \in \Omega$. If $\hd(p) \neq \varnothing$, then $\de(c_{-\hd(p)}) = \varnothing$.
\end{lem}

\begin{proof}
Denote $\hd := \hd(p)$. Suppose $c_{-\hd} \in \bbD^{t_1}_{c_{-n}}$ for some $n \geq 0$. By \lemref{disc tunnel inv}, we see that $\hd > \bepsilon_1^{-1} n$. Thus,
\begin{align*}
\dist(p, c_{-n}) &< \dist(p, c_{-\hd}) + \dist(c_{-\hd}, c_{-n}) \\
&< 2{t_1}\lambda^{\epsilon_2 \hd} + {t_1}\lambda^{\epsilon_2 n}\\
&< 2{t_1}\lambda^{\epsilon_2 n}.
\end{align*}
So $p \in \bbD^{2{t_1}}_{c_{-n}}$, which is a contradiction.
\end{proof}

Let
$\tiN_1 := -\delta^{-1}\logl \overline{L_1},$
and let $N_1$ be the first Pliss moment larger than $\tiN_1$ in the forward orbit of $c_1$. By \propref{freq pliss for}, we have
$
\tiN_1 \leq N_1 < (1+\bepsilon_1)\tiN_1.
$
Denote
$$
{t_2} := \radius(U_{c_1}^{\bepsilon_1 N_1}) \asymp \lambda^{\bepsilon_1 N_1} = {L_1}^{-\bepsilon_1} < {t_1},
$$
and let
$$
\bfT_{c_n} := F^{n-1}(\cT_{c_1}({t_2}))
\matsp{for}
0\leq n \leq N_1.
$$

\begin{lem}\label{edge fast escape}
Let $p_1 \in \cT_{c_1} \setminus \bfT_{c_1}$.
If $\de(p_1) = \varnothing$, then
$$
\he(p_1) < -\bepsilon_1 \logl {L_1}.
$$
\end{lem}

\begin{proof}
Letting $t = 2{t_1}$ and applying \lemref{max escape time}, we have
$$
\he_0(p_0) < \bepsilon_2^{-1}\logl {t_2} = -\bepsilon_1 \logl {L_1}.
$$
The result now follows from \propref{full escape}.
\end{proof}

\begin{lem}\label{1 tunnel recover}
Let $p_{N_1} \in \bfT_{c_{N_1}}$ be $n$-times forward $({\overline{L_1}}, \delta)_v$-regular along $\tiE^v_{p_{N_1}} \in \bbP^2_{p_{N_1}}$ with
$
n > -\logl {\overline{L_1}}.
$
Then $p_1$ is $(n+N_1)$-times forward $(1, \delta)_v$-regular along $\tiE^v_{p_1}$. Moreover, $\tiE^v_{p_1}$ is sufficiently stable-aligned in $\cU_{c_1}$.
\end{lem}

\begin{proof}
Since $p_1 \in \cU_{c_1}^{\bepsilon_1 N_1}$, we see that $p_i \in \cU_{c_i}$ for $1 \leq i \leq M$ with
$$
M > \epsilon_1^{-1}(\bepsilon_1 N_1) > \bepsilon_1^{-1} N_1.
$$
By \propref{vert angle shrink}, we have for some uniform constant $K \geq 1$:
$$
\measuredangle(\tiE^v_{p_{N_1}}, E^v_{p_{N_1}}) < K\lambda^{-\epsilon_1 N_1}{\overline{L_1}} \lambda^{(1-\delta)n} < K{L_1}^{\bepsilon_1}{\overline{L_1}} \lambda^{(1-\delta)n} < {\overline{L_1}}^{-1} < \lambda^{\bepsilon_1 N_1}.
$$
Thus, $\tiE^v_{p_{N_1}}$ is sufficiently vertical in $\cU_{c_{N_1}}$.

By \propref{loc linear}, we conclude that $p_1$ is $N_1$-times forward $(1, \bepsilon_1)_v$-regular along $\tiE^v_{p_1}$. Since
$$
\lambda^{(1-\bepsilon_1) N_1}{\overline{L_1}}\lambda^{(1-\delta)i} ={\overline{L_1}}\lambda^{(\delta-\bepsilon_1) N_1} \lambda^{(1-\delta) (N_1+i)} < \lambda^{(1-\delta) (N_1+i)}
\matsp{for}
i \geq 0,
$$
the result follows.
\end{proof}

\begin{lem}\label{1 tunnel inv}
Let $p_1 \in \bfT_{c_1}$. Then one of the following statements hold.
\begin{enumerate}[i)]
\item We have $\cd(p_1) \in \{\varnothing, \infty\}$, and
$
\cd(p_{N_1}) = \cd(p_1).
$
\item We have $\bepsilon_1^{-1} N_1< \cd(p_1) < \infty$, and
$
\cd(p_{N_1}) = \cd(p_1) - N_1+1.
$
\end{enumerate}
\end{lem}

\begin{proof}
Since $N_1$ is a Pliss moment, and
$$
\dist(c_0, c_{m-N_1}) < \lambda^{\bepsilon_1 N_1},
$$
it follows from \lemref{no pliss near crit} that either
$
\cd(p_1) = \varnothing
$ or $
 \cd(p_1) > \bepsilon_1^{-1} N_1.
$

Suppose $p_{N_1} \in \cT_{c_{-n}}$ for some $n > \bepsilon_1^{-1} N_1$. Then by \propref{tunnel size}, we have
$$
\dist(c_{-n-N_1}, c_0) < \diam(\cT_{c_{-n-N_1}})+ \diam(\bfT_{c_1}) < \lambda^{\delta (n+N_1)} + \lambda^{\bepsilon_1 N_1}< \lambda^{\bepsilon_1 N_1}.
$$
Thus, by \lemref{no pliss near crit}, the result follows.
\end{proof}

\begin{lem}\label{cutoff for reg}
Let $p_1 \in \Lambda \cap \cT_{c_1}$ be a point such that $\cd(p_1) = \varnothing$,
and let $\tiE^v_{p_1} \in \bbP^2_{p_1}$ be sufficiently stable-aligned in $\cU_{c_1}$. Then for all $n \geq 1$ such that $d_{\cT}(p_{-n}) = \varnothing$, the point $p_{-n}$ is $n$-times forward $({L_1}^{1+\bepsilon}, \delta)_v$-regular along $\tiE^v_{p_{-n}}$.
\end{lem}

\begin{proof}
Let $-j_0 < 0$ be the tunnel escape moment of $p_0$, and $\he_0 := \he_0(p_0)$. By \propref{tunnel escape} and \lemref{shallow tunnel free}, we see that for $j_0 \leq i \leq \he_0$ such that $\hd(p_{-i}) = i$, the point $p_{-i}$ is $i$-times $(O({L_1}), \delta)_v$-regular along $\tiE^v_{p_{-i}}$.

Let $m_1 > j_0$ be the smallest integer such that $p_{-m_1} \in \cT_{c_0}$. We claim that $\hd(p_{-m_1+1}) = m_1+1$. Otherwise, there exists $0 < m_1' < j_0$ such that
$$
\hd(p_{-m_1'-i}) = -i
\matsp{for}
0\leq i \leq m_1.
$$
Then by \lemref{disc tunnel inv}, we must have
$$
j_0 > \he_0(p_{-m_1'}) > m_1,
$$
which is a contradiction.

If $p_{-m_1} \notin \bfT_{c_0}$, then by \lemref{edge fast escape}, we see that for
$$
0 \leq i \leq l \leq \he_1 := \he(p_{-m_1}),
$$
we have
$$
\|DF^i|_{\tiE^v_{p_{-m_1-l}}}\| < \lambda^{-\eta i}= \lambda^{-(1-\delta+\eta)i}\lambda^{(1-\delta)i} \leq {L_1}^{\bepsilon_1} \lambda^{(1-\delta)i}.
$$
Moreover, by \propref{tunnel escape} and \lemref{shallow tunnel free}, the point $p_{-m_1+1}$ is $(m_1-1)$-times $(O({L_1}), \delta)_v$-regular along $\tiE^v_{p_{-m_1+1}}$. Hence, $p_{-m_1-l}$ is $(m_1+l)$-times $({L_1}^{1+\bepsilon_1}, \delta)_v$-regular along $\tiE^v_{p_{-m_1-l}}$.

Suppose instead that $p_{-m_1} \in \bfT_{c_0}$, so that $p_{-m_1 + N_1} \in \bfT_{c_{N_1}}$. By \propref{tunnel escape}, \lemref{shallow tunnel free} and \lemref{1 tunnel inv}, we see that $p_{-m_1+N_1}$ is $(m_1-N_1)$-times $(O({L_1}), \delta)_v$-regular along $\tiE^v_{p_{-m_1+N_1}}$. By \lemref{1 tunnel recover}, the point $p_{-m_1+1}$ is $(m_1-1)$-times $(1, \delta)_v$-regular along $\tiE^v_{p_{-m_1+1}}$, and in particular, is sufficiently stable-aligned in $\cU_{c_1}$.

Proceeding by induction, the result follows.
\end{proof}

\begin{lem}\label{full tunnel for reg}
If
$$
p \in \Lambda \setminus \bigcup_{i \geq 0}^\infty \cT_{c_{-i}}({L_1}^{-\bepsilon}),
$$
then $p$ is forward $({L_1}^{1+\bepsilon}, \delta)$-regular. If $p$ is also contained in $\bfT_{c_1} = \cT_{c_1}({L_1}^{-\bepsilon})$, then $p$ is forward $(1, \delta)$-regular.
\end{lem}

\begin{proof}
Let $p_0 \in \Lambda \cap \cT_{c_0}$. If $p_0 \notin \bfT_{c_0}$, then by \lemref{edge fast escape}, we see that for
$$
0 \leq i \leq l \leq \he(p_0),
$$
we have
$$
\|DF^i|_{\tiE^v_{p_{-l}}}\| < \lambda^{-\eta i}= \lambda^{-(1-\delta+\eta)i}\lambda^{(1-\delta)i} \leq {L_1}^{\bepsilon_1} \lambda^{(1-\delta)i}.
$$
If $\de(p_0) = \varnothing$, then this implies that $p_{-l}$ is forward $({L_1}^{1+\bepsilon_1}, \delta)$-regular. Otherwise, $d_1 := \de(p_0) > 0$, and by \lemref{return dist lem}, we have
$$
\dist(p_{d_1}, c_0) < \dist(p_0, c_0)^{1/\bepsilon_1}.
$$
Proceeding by induction, we obtain a sequence $0 =: d_0 < d_1 < \ldots < d_M$ with $M \geq 0$ such that $d_{i+1} := \de(p_{d_i})$, and either $\de(p_{d_M}) = \varnothing$, or $p_{d_M} \in \bfT_{c_0}$. In the former case, we can argue as before to conclude that $p_{d_M-l}$ is forward $({L_1}^{1+\bepsilon_1}, \delta)$-regular for all $0 \leq l \leq \he(p_{d_M})$.

It remains to prove the result for $p_0 \in \bfT_{c_0}$ satisfying $\cd(p_0) < \infty$. By \lemref{disc tunnel inv} ii) and \lemref{1 tunnel inv}, we have
$$
\cd(p_{d_{\cT}(p_0)+N_1}) = \varnothing.
$$
Thus, by replacing $p_{N_1}$ with $p_{\cd(p_{N_1})+N_1}$, we may assume without loss of generality that $\cd(p_{N_1}) = \varnothing$.

If there exists $M \geq 0$ such that $p_M \in \bfT_{c_{N_1}}$ and $\de(p_M) = \varnothing$, then the result follows by \lemref{1 tunnel recover} and \lemref{cutoff for reg}. Suppose this is not the case. Then there exists an infinite sequence $0\leq M_1 < \ldots$ such that
$$
p_{M_i} \in \bfT_{c_{N_1}},
\hspace{5mm}
\cd(p_{M_i}) = \varnothing
\matsp{and}
\de(p_{M_i}) = M_{i+1}-M_i-N_i.
$$

Let $E^i_{p_{M_i}} := E^v_{p_{M_i}}$, so that
$$
D\Phi_{c_{N_1}}\left(E^i_{p_{M_i}}\right) = E^{gv}_{\Phi_{c_{N_1}}(p_{M_i})}.
$$
By \lemref{1 tunnel recover} and \lemref{cutoff for reg}, $p_0$ is $M_i$-times forward $({L_1}^{1+\bepsilon}, \delta)_v$-regular along $E^i_{p_0}$. Hence, by \propref{vert angle shrink}, $E^i_{p_0}$ converge exponentially fast as $i \to \infty$ to $E^s_{p_0}$, along which $p_0$ is forward $({L_1}^{1+\bepsilon}, \delta)_v$-regular.
\end{proof}

Let $p_0 \in \cT_{c_0}$. Write
$
(x_0, y_0) := \Phi_{c_0}(p_0).
$
Denote
$$
\til := (1+\bepsilon_1)\logl |x_0|.
$$
By \lemref{freq pliss for}, there exists a Pliss moment $-l < 0$ in the backward orbit of $c_0$ such that
$
\til \leq l < (1+\bepsilon_1)\til.
$
Such a moment is referred to as a {\it direction aligning moment for $p_0$}.

\begin{lem}\label{direction align lem}
Let $p_0 \in \cT_{c_0}$, and let $-l < 0$ be a direction aligning moment for $p_0$. If $\tiE^v_{p_1} \in \bbP^2_{p_1}$ is sufficiently stable-aligned in $\cU_{c_1}$, then $\tiE^v_{p_{-n}}$ is sufficiently vertical in $\cU_{c_{-n}}$ for $l \leq n < \he_0(p_0)$. If, additionally, we have
$$
p_{-n}\in \bfT_{c_1}
\matsp{and}
\hd(p_{-n}) = n,
$$
then $\tiE^v_{p_{-n}}$ is sufficiently stable-aligned in $\cU_{c_1}$.
\end{lem}

\begin{proof}
The first claim follows immediately from \propref{tunnel escape} ii).

Write
$
(x_1, y_1) := \Phi_{c_1}(p_{-n+1})
$ and $
(\tix_1, \tiy_1) := \Phi_{c_1}(c_{-n+1}).
$
By \lemref{return crit dist lem}, we have
\begin{equation}\label{eq:direction align lem}
|x_1-\tix_1|^{\bepsilon_1} < |\tix_1|.
\end{equation}
Consequently, $c_{-n+1} \in \cU_{c_1}^{\bepsilon_1 N_1}$. Moreover, by \lemref{shallow tunnel free} and \lemref{1 tunnel inv}, we have $\de(c_{-m+N_1}) = \varnothing$. Arguing as in the proof of \lemref{1 tunnel recover}, we see that $c_{-n+1}$ is forward $(1, \delta)$-regular, and hence, $\hE^v_{c_{-n+1}}$ is sufficiently stable-aligned in $\cU_{c_1}$. Since $\tiE^v_{p_{-n+1}}$ is sufficiently vertical in $\cU_{c_{-n+1}}$, \eqref{eq:direction align lem} implies that $\tiE^v_{p_{-n+1}}$ is also sufficiently stable-aligned in $\cU_{c_1}$.
\end{proof}

\begin{lem}\label{no direction align}
Let  $p_0 \in \cT_{c_0}$, and let $-l <0$ be a direction aligning moment for $p_0$. If $p_{-n} \in \cT_{c_0}$ for some $1 \leq n \leq \he_0(p_0)$, then either $n > l$ or $l > \he_0(p_{-n})$.
\end{lem}

\begin{proof}
The result follows from a similar argument as the one used in the proof of \lemref{disc tunnel inv}.
\end{proof}

\begin{proof}[Proof of \thmref{ref uni} a)]
Let ${t} \in (0, {L_1}^{-1}]$. Consider $\bfp \in \Lambda$ such that
$$
M := \cd(\bfp) \geq 0
\matsp{and}
\cd_{t}(\bfp) = \varnothing.
$$
Then we have
$$
p_1 := F^{M+1}(\bfp) \in \bfT_{c_1} \setminus \cT_{c_1}({t}),
$$
and by \lemref{disc tunnel inv} and \lemref{1 tunnel inv},
$$
\cd(p_1) = \cd(p_{N_1}) = \varnothing =  \cd_{t}(p_1).
$$
Thus, it follows from \lemref{1 tunnel recover} that $p_1$ is forward $(1, \delta)_v$-regular along $\hE^v_{p_1}$, and $\hE^v_{p_1}$ is sufficiently stable-aligned in $\cU_{c_1}$.

Let $-l_0 < 0$ be a direction aligning moment for $p_0$. By \lemref{direction align lem}, for
$
l_0 \leq n \leq \he_0(p_0),
$
the direction $\hE^v_{p_{-n}}$ is sufficiently vertical in $\cU_{c_{-n}}$. Denote the vertical direction at $p_{-n}$ in $\cU_{c_{-n}}$ by $\tiE^0_{p_{-n}}$, so that
$$
D\Phi_{c_{-n}}(\tiE^0_{p_{-n}}) = E^{gv}_{\Phi_{c_{-n}}(p_{-n})}.
$$

Let $m_0 := 0$, and for $i \geq 1$, let
$$
m_{i-1} + l_{i-1} < m_i< m_{i-1}+\he_0(p_{-m_{i-1}})
$$
be an integer such that
$$
p_{-m_i} \in \bfT_{c_0} \setminus \cT_{c_0}({t})
\matsp{and}
\hd(p_{-m_i}) = -m_i+m_{i-1}.
$$
Write
$
z^i = (x^i, y^i) := \Phi_{c_0}(p_{-m_i}).
$
Then we have by \lemref{return dist lem},
\begin{equation}\label{eq:direction align further}
|x^i| > |x^{i-1}|^{\bepsilon_1}.
\end{equation}
Let
$$
-l_i = -(1+\bepsilon_1)\logl |x^i|,
$$
be a direction aligning moment for $p_{-m_i}$. By \lemref{direction align lem}, $\hE^v_{p_{-m_i+1}}$ is sufficiently stable-aligned in $\cU_{c_1}$, and for 
$
l_i \leq n \leq \he_0(p_{-m_i}),
$
the direction $\hE^v_{p_{-m_i-n}}$ is sufficiently vertical in $\cU_{c_{-n}}$.

We claim that the initially given point $p_{-M} = \bfp$ is forward $(L, \delta)$ regular along $\hE^v_{p_{-M}}$ with $L = 1/\underline{{t}}$. Let
$$
0 :=m_0 < m_1 < \ldots < m_I \leq M
$$
with $I \geq 0$ be a maximal sequence of returns to $\bfT_{c_0} \setminus \cT_{c_0}({t})$ as specified above.

For $0 \leq i < I$, denote
$
n_i := m_{i+1}-m_i.
$
Since for $1 \leq n \leq n_i - l_i$, the direction $\hE^v_{p_{-m_{i+1}+n}}$ is sufficiently vertical in $\cU_{c_{-n_i+n}}$, and $c_{-n_i}$ is forward $(1, \delta)$-regular, we conclude that $p_{-m_{i+1}}$ is $(n_i - l_i)$-times forward $(1, \delta)$-regular. It follows that $p_{-m_{i+1}}$ is $n_i$-times forward $(L_i, \delta)$-regular, where
$$
L_i < \lambda^{-(1-\delta - \eta) l_i}.
$$

Denote
$
n_I := M - m_I.
$
If $n_I > l_I$, then for $0\leq n \leq n_I -l_I$, the direction $\hE^v_{p_{-m_{i+1}+n}}$ is sufficiently vertical in $\cU_{c_{-n_i+n}}$. Moreover, since
$
\hd(p_{-M}) = n_i,
$
the point $c_{-n_i}$ is forward $({L_1}, \delta)$-regular. Thus, $p_{-M}$ is $(n_I-l_i)$-times forward $({L_1}, \delta)$-regular. It follows that (whether $n_I > l_I$ or not) $p_{-M}$ is $n_I$-times forward $({L_1} L_I, \delta)$-regular, where
$
L_I < \lambda^{-(1-\delta - \eta) l_I}.
$

Lastly, observe that by \eqref{eq:direction align further}, we have
$
l_i < \bepsilon_1 l_{i-1}
$ for $
1 \leq i \leq I,
$
and hence
\begin{equation}\label{eq:total nonalign time}
\sum_{i=0}^I l_i < Kl_0 < -K\logl |x^0| < K \logl {t}.
\end{equation}

Concatenating the above regularities, we conclude that $p_{-M}$ is forward $({L_1} L', \delta)_v$-regular along $\tiE^v_{p_{-M}}$, where
$$
{L_1} L' < {t}^{-1} \lambda^{-(1-\delta-\eta)\sum_{i=0}^I l_i} < 1/\underline{{t}}.
$$
\end{proof}

Let $N \geq 0$ and $K \geq 1$. We say that a point $p_0 \in \Omega$ is {\it $N$-times consistently forward $(K, \delta)_v$-regular along $\tiE^v_{p_0} \in \bbP^2_{p_0}$} if, for all $0\leq i \leq N$, the point $p_i$ is $(N-i)$-times forward $(K, \delta)_v$-regular along $\tiE^v_{p_i}$.

\begin{lem}\label{consistent for to back}
If for some $N \geq 0$ and $K\geq 1$, a point $p_0 \in \Omega$ is $N$-times consistently forward $(K, \delta)_v$-regular along $\tiE^v_{p_0} \in \bbP^2_{p_0}$, then $p_i$ is $(i, N-i)$-times $(K, \delta)_v$-regular along $\tiE^v_{p_i}$.
\end{lem}

\begin{lem}\label{back reg ver to hor}
Let $p_0 \in \Omega$ be $M$-times backward $(K, \delta)_v$-regular along $\tiE^v_{p_0} \in \bbP^2_{p_0}$ for some $M \geq 0$ and $K \geq 1$. If $\tiE^h_{p_{-M}} \in \bbP^2_{p_{-M}}$ is a direction satisfying
$$
\measuredangle(\tiE^h_{p_{-M}}, \tiE^v_{p_{-M}}) > k
$$
for some uniform constant $k >0$, then $p_0$ is $M$-times backward $(\bK, \delta)_h$-regular along $\tiE^h_{p_0}$. Moreover, we have
$$
\measuredangle(\tiE^h_{p_0}, \tiE^v_{p_0}) > \bK^{-1}.
$$
\end{lem}

Let $\delta_1 = \udelta$, so that $\cT_{c_0}$ is a $1/\underline{\delta_1}$-pinched neighborhood. Denote
$$
\delta_1' := (1-\bdelta_1)\delta_1.
$$
For ${t} \in (0, {L_1}^{-1}]$ and $n \geq 0$, let $\hcT_{c_{-n}}({t})$ be the $n$-times $1/\delta_1$-pinched, $\delta_1'$-truncated tunnel of length ${t}$. Numbers $-l < -j < 0$ are the {\it slow and fast tunnel escape moment for $p_0 \in \cT_{c_0}({t})$} respectively if $l$ and $j$ are the smallest integers such that
$$
p_{-l} \notin \hcT_{c_{-l}}({t})
\matsp{and}
p_{-j} \notin \cT_{c_{-j}}({t}).
$$
If $p \in \hcT_{c_{-n}}({t})$ for some $n \geq 0$, then denote by
$$
d:= d_{\hcT_-({t})}(p) \in \bbN \cup \{0, \infty\}
$$
the largest integer such that $p \in \hcT_{c_{-d}}({t})$. Otherwise, denote
$$
d_{\hcT_-({t})}(p) = \varnothing.
$$

\begin{lem}\label{trans back to tunnel}
Let ${t} \in (0, {L_1}^{-1}]$ and $L := 1/\underline{{t}}$. For $p_0 \in \Lambda \cap \cT_{c_0}({t})$, let $-l<-j < 0$ be the slow and fast tunnel escape moment for $p_0$. Suppose $\tiE^v_{p_0}, \tiE^h_{p_0}\in \bbP^2_{p_0}$ are directions such that $\tiE^v_{p_1}$ is sufficiently stable-aligned in $\cU_{c_1}$; the point $p_{-l}$ is $M$-times backward $(\bL, \delta)_h$-regular along $\tiE^h_{p_{-l}}$ for some $M \geq 0$; and
$$
\measuredangle(\tiE^v_{p_{-l}}, \tiE^h_{p_{-l}}) > \bL^{-1}.
$$
Then for $n \in \{0, j\}$, the point $p_{-n}$ is $(M+l-n)$-times backward $(1, \delta)_h$-regular along $\tiE^h_{p_0}$. Moreover,
$$
\measuredangle(\tiE^h_{p_{-j}}, \tiE^v_{p_{-j}}) > k
$$
for some uniform constant $k > 0$, and $\tiE^h_{p_0}$ is sufficiently center-aligned in $\cU_{c_0}$.
\end{lem}

\begin{proof}
Write
$
z_0 = (x_0, y_0) := \Phi_{c_0}(p_0)$; $E^v_{p_0} := D_{z_0}\Phi_{c_0}^{-1}(E^{gv}_{z_0})$ and $
\tiE^h_{z_0} := D\Phi_{c_0}(\tiE^h_{p_0}).
$
Then
$$
l = (1+\bdelta)\delta_1^{-1}\logl |x_0|
\matsp{and}
j < \udelta^{-1}\logl |x_0|.
$$
By \propref{tunnel escape} and \propref{vert angle shrink}, we have
\begin{equation}\label{eq:trans direct}
\measuredangle(\tiE^v_{p_{-l}}, E^v_{p_{-l}}) < \lambda^{(1-\bdelta) l} < |x_0|^{(1-\bdelta)/\delta_1} < \bL^{-(1-\bdelta)/\delta_1} \ll \bL^{-1}
\end{equation}
Hence,
$$
\measuredangle(\tiE^h_{p_{-l}}, E^v_{p_{-l}}) > \bL^{-1}.
$$

Denote
$$
\theta^{v/h}_0 := \measuredangle(\tiE^h_{z_0}, E^{gv/gh}_{z_0}).
$$
For $0\leq i \leq l$, if $\theta^v_{-i} < k$ for some uniform constant $k > 0$, then by \thmref{reg chart} and \propref{loc linear}, we have
\begin{equation}\label{eq:trans back 1}
\theta^v_{-i+1} > \lambda^{-(1-\bepsilon)}\theta^v_{-i}.
\end{equation}
Similarly, if $\theta^h_{-i} < k$, then
\begin{equation}\label{eq:trans back 2}
\theta^h_{-i+1} > \lambda^{1-\bepsilon}\theta^h_{-i}.
\end{equation}
Let $I_1, I_2$ be the smallest integers such that
$
\theta^v_{-l+I_1} > k
$ and $
\theta^h_{-l+I_2} < k.
$
Inequalities \eqref{eq:trans back 1} and \eqref{eq:trans back 2} imply
$$
I_1 < -\logl\bL < K\logl |x_0|
\matsp{and}
I_2 < I_1 + K
$$
for some uniform constant $K \geq 1$.

Let $N > 0$ be the largest integer such that
$
\theta_{-N}^h < \lambda^{\bepsilon (l-N)}.
$
By a straightforward computation, we obtain
$$
N > l - I_2 - \bepsilon \logl |x_0| > \delta_1^{-1}\logl |x_0|.
$$
Consequently,
$$
l-N < \bdelta_1 N
\matsp{and}
N - j > (1-\udelta^{-1}\delta_1)N=(1-\bdelta_1)N.
$$
Note that for $0 \leq i \leq N$, the direction $\tiE^h_{p_{-i}}$ is sufficiently horizontal in $\cU_{c_{-i}}$.

By \propref{loc linear}, $p_{-n}$ with $n \in \{0, j\}$ is $(N-n)$-times backward $(1, \bepsilon)_h$-regular along $\tiE^h_{p_{-n}}$. For $0 \leq i < l - N$, we have
\begin{align*}
\|D_{p_{-j}}F^{-(N+i-j)}\| &< \lambda^{-\bepsilon (N-j)} \lambda^{-(1+\eta) (l-N)}\\
&< \lambda^{-\bdelta_1 N}\\
&< \lambda^{-\bdelta_1 N} \lambda^{\delta(N+i-j)}\lambda^{-\delta (N+i-j)}\\
&<  \lambda^{\delta(1-\bdelta_1) N}\lambda^{-\delta (N+i-j)}\\
&< {t}^{1/\bdelta_1}\lambda^{-\delta (N+i-j)}\\
&< \bL^{-1}\lambda^{-\delta(N+i-j)}.
\end{align*}
Concatenating with the $M$-times backward $(\bL, \delta)_h$-regularity of $p_{-l}$ along $\tiE^h_{p_{-l}}$, we see that the point $p_{-n}$ with $n \in \{0, j\}$  is $(M+l-n)$-times backward $(1, \delta)_h$-regular along $\tiE^h_{p_{-n}}$.

The other claims follow immediately from \lemref{back reg ver to hor} and \propref{stable center align}.
\end{proof}

\begin{proof}[Proof of \thmref{ref uni} b)]
For ${t} \in (0, {L_1}^{-1}]$, let $L := 1/\underline{{t}}$ be the forward regularity factor given in part a). Let $p_0 \in \Lambda$ be a point satisfying
$
\vd_{t}(p_0) = \varnothing.
$
We prove that $p_0$ is backward $(\bL, \delta)$-regular.

First, suppose
$
\cd_{t}(p_0) = \varnothing,
$
so that $p_0$ is forward $(L, \delta)_v$-regular along $\hE^v_{p_0}$. Let $N \geq 0$ be an integer such that 
$
\cd_{t}(p_{-N}) = \varnothing;
$
and let $\tiE^N_{p_{-N}} \in \bbP^2_{p_{-N}}$ be a direction satisfying
$$
\measuredangle(\tiE^N_{p_{-N}}, \hE^v_{p_{-N}}) > k
$$
for some uniform constant $k > 0$. We claim that $p_0$ is $N$-times backward $(\bL, \delta)_h$-regular along $\tiE^N_{p_0}$.

Let $0\leq m_1 < \ldots < m_I < N$ for some $I \geq 0$ be the sequence of all integers satisfying
$$
p_{-m_i} \in \cT_{c_1}({t})
\matsp{and}
\cd_{t}(p_{-m_i}) = \varnothing.
$$
Let $-l_i < -j_i < 0$ be the slow and fast tunnel escape moment for $p_{-m_i-1}$, and let $n_i > 1$ be the crescent escape moment for $p_{-m_i}$. By \propref{tunnel inv}, we have
$$
\cd_{t}(p_{-m_i-l_i}) = \cd_{t}(p_{-m_i-j_i})  = \cd_{t}(p_{-m_i + n_i})= \varnothing,
$$
and
$$
m_i + l_i < m_{i+1}-n_{i+1}.
$$
Hence,
$$
\cd_{t}(p_{-m}) = \varnothing
\matsp{for}
m_i + j_i \leq k \leq m_{i+1}-n_{i+1}.
$$
By \lemref{consistent for to back}, $p_{-m_I-l_I}$ is $(N-m_I - l_I)$-times backward $(L, \delta)_v$-regular along $\hE^v_{-m_I-l_I}$. Applying \lemref{back reg ver to hor}, \lemref{trans back to tunnel} and \propref{crescent escape}, we see that $p_{-m_I+n_I}$ is $(N - m_I +n_I)$-times backward $(O(1), \delta)_h$-regular along $\tiE^N_{p_{-m_I+n_I}}$, and
$$
\measuredangle(\tiE^N_{p_{-m_I+n_I}}, \hE^v_{p_{-m_I+n_I}}) > k.
$$

Proceeding by induction, we conclude that $p_{-m_1+n_1}$ is $(N - m_1 +n_1)$-times backward $(O(1), \delta)_h$-regular along $\tiE^N_{p_{-m_1+n_1}}$, and
$$
\measuredangle(\tiE^N_{p_{-m_1+n_1}}, \hE^v_{p_{-m_1+n_1}}) > k.
$$
By \lemref{consistent for to back} and  \lemref{back reg ver to hor}, $p_0$ is $(m_1-n_1)$-times backward $(\bL, \delta)_h$-regular along $\tiE^N_{p_0}$. Concatenating the regularities, it follows that $p_0$ is $N$-times backward $(\bL, \delta)_h$-regular along $\tiE^N_{p_0}$.

Now, suppose
$$
0 \leq M := \cd_{t}(p_0) < \infty.
$$
Denote $m_0 := M$. Let $-l_0 < -j_0 <0$ be the slow and fast tunnel escape time for $p_{m_0} \in \cT_{c_0}({t})$. Then
$$
\cd_{t}(p_{m_0+1}) = \cd_{t}(p_{m_0-j_0}) =  \cd_{t}(p_{m_0-l_0}) = \varnothing.
$$
Observe that $j_0 > m_0$. Let $m_1$ be the largest integer less than $m_0$ such that $p_{m_1} \in \cT_{c_0}({t})$, and if $-l_1 < -j_1 <0$ are the slow and fast tunnel escape time for $p_{m_1}$, then $j_1 > m_1$. Clearly, we have
$$
p_{m_1} \in \cT_{c_{m_1-m_0}}({t})
\matsp{and}
\cd_{t}(c_{m_1-m_0}) = \varnothing.
$$
Moreover, by \propref{tunnel inv}, we have
$$
0 > m_1 - l_1 > m_0 - j_0.
$$

Proceeding by induction, we obtain the maximal sequence $M =: m_0 > m_1 > \ldots > m_I >0$ with $I \geq 0$, such that for $0 < i \leq I$, we have
$$
p_{m_i} \in \cT_{c_{m_i-m_{i-1}}}({t})
\matsp{and}
\cd_{t}(c_{m_i- m_{i-1}+1}) = \varnothing;
$$
and if $-l_i < -j_i < 0$ are the slow and fast tunnel escape time for $p_{m_i}$, then
$$
j_i > m_i
\matsp{and}
0 > m_i - l_i > m_{i-1} - j_{i-1}.
$$

Let $\hE^i_{p_{m_i}}$ be the vertical direction at $p_{m_i}$ in $\cU_{c_0}$, so that
$$
D\Phi_{c_0}(\hE^i_{p_{m_i}}) = E^{gv}_{\Phi_{c_0}(p_{m_i})}.
$$
Then by the assumption on the critical depth of $c_{m_i-m_{i-1}+1}$, the direction $\hE^{i-1}_{p_{m_i+1}}$ is sufficiently stable-aligned in $\cU_{c_1}$. Hence, by \lemref{tunnel escape}, $p_{m_i - l_i}$ is $l_i$-times forward $(1, \delta)_v$-regular along $\hE^{i-1}_{p_{m_i-l_i}}$.

Let $N \geq l_0 - m_0$ be an integer such that 
$
\cd_{t}(p_{-N}) = \varnothing;
$
and let $\tiE^N_{p_{-N}} \in \bbP^2_{p_{-N}}$ be a direction satisfying
$$
\measuredangle(\tiE^N_{p_{-N}}, \hE^v_{p_{-N}}) > k
$$
for some uniform constant $k > 0$. We claim that $p_0$ is $N$-times backward $(\bL, \delta)_h$-regular along $\tiE^N_{p_0}$.

Denote
$$
\hE^{-1}_{p_0} := \hE^v_{p_0}
\matsp{and}
k_0 := N - (l_0 - m_0);
$$
and 
$$
k_i := (j_{i-1} - m_{i-1}) -(l_i  - m_i)> 0
\matsp{for}
1\leq i \leq I.
$$
Suppose that for $0 \leq i \leq I$, the point $p_{m_i-l_i}$ is $k_i$-times backward $(\bL, \delta)_h$-regular along $\tiE^N_{p_{m_i-l_i}}$, and 
$$
\measuredangle(\tiE^N_{p_{m_i-l_i}}, \hE^{i-1}_{p_{m_i-l_i}}) > \bL^{-1}.
$$
For $i = 0$, this was proved in the previous case.

Observe that
$$
p_{m_{i-1} -n} \in \cT_{c_{-n}}({t})
\matsp{for}
n < j_{i-1};
$$
and
$$
\cd_{t}(c_{-m_{i-1}+m_i-n}) = \varnothing
\matsp{for}
n \in \{-1, j_i, l_i\}.
$$
Hence, $\tiE^{i-1}_{p_{m_i+1}}$ is sufficiently stable-aligned in $\cU_{c_1}$, and $p_{m_i-l_i}$ is $l_i$-times forward $(1, \delta)_v$-regular along $\tiE^{i-1}_{p_{m_i+1}}$. By \lemref{trans back to tunnel}, $p_{m_i-j_i}$ is $(k_i +l_i-j_i)$-times backward $(1, \delta)_h$-regular along $\tiE^N_{p_{m_i-j_i}}$, and 
$$
\measuredangle(\tiE^N_{p_{m_i-j_i}}, \hE^i_{p_{m_i-j_i}}) > k.
$$
Since
$$
\cd_{t}(c_{-m_i+m_{i+1}-l_{i+1}}) = \varnothing,
$$
it follows from the previous case that $p_{m_{i+1}-l_{i+1}}$ is $k_{i+1}$-times backward $(\bL, \delta)_h$-regular along $\tiE^N_{p_{m_{i+1}-l_{i+1}}}$, and 
$$
\measuredangle(\tiE^N_{p_{m_{i+1}-l_{i+1}}}, \hE^i_{p_{m_{i+1}-l_{i+1}}}) > \bL^{-1}.
$$

By induction, we see that $p_{m_I-l_j}$ is $(N-(j_I-m_I))$-times backward $(1, \delta)_h$-regular along $\tiE^N_{p_{m_I-j_I}}$, and 
$$
\measuredangle(\tiE^N_{p_{m_I-j_I}}, \hE^I_{p_{m_I-j_I}}) > k.
$$
Since
$
\cd_{t}(c_{-m_I}) = \varnothing,
$
by applying the previous case again, we conclude that $p_0$ is $N$-times backward $(\bL, \delta)_h$-regular along $\tiE^N_{p_{m_I-j_I}}$.

Lastly, suppose that
$
\cd_{t}(p_0) = \infty.
$
Let $D \geq 0$ be an integer such that $p_0 \in \cT_{c_{-D}}({t})$. Note that
$
d_{\hcT_-({t})}(c_{-D}) \neq \infty.
$
Moreover, since
$
\vd_{t}(p_0) =\varnothing,
$
we may assume by \propref{loc linear} that
$
\vd_{t}(c_{-D}) = \varnothing.
$

Let $N > 0$ be an integer such that $-(D+N) <0$ is the slow tunnel escape moment for $p_D \in \cT_{c_0}$. Observe that
$$
\dist(p_D, c_0) < \lambda^{-\eta D} \dist(p_0, c_{-D}) < \lambda^{(1+\bdelta)\delta D}
$$
for $D$ sufficiently large by \lemref{tunnel size}. Consequently,
$$
N > (1-\bdelta)\delta_1^{-1}\delta D > \bdelta_1^{-1} D.
$$

Let $\hE^v_{c_0} \in \bbP^2_{c_0}$ be the vertical direction at $c_0$ in $\cU_{c_0}$, so that
$$
D\Phi_{c_0}(\hE^v_{c_0}) = E^{gv}_0.
$$
Let $\tiE^N_{c_{-(D+N)}} \in \bbP^2_{c_{-(D+N)}}$ be a direction satisfying
$$
\measuredangle(\tiE^N_{c_{-(D+N)}}, \hE^v_{c_{-(D+N)}}) > k.
$$
By the previous two cases, we see that $c_{-D}$ is $N$-times backward $(\bL, \delta)_h$-regular along $\tiE^N_{c_{-D}}$. By \propref{loc linear}, there exists $\tiE^N_{p_0} \in \bbP^2_{p_0}$ such that $p_0$ is $N$-times backward $(\bL, \delta)_h$-regular along $\tiE^N_{p_0}$.

We proved that for $N$ arbitrarily large, there exists $\tiE^N_{p_0} \in \bbP^2_{p_0}$ such that $p_0$ is $N$-times backward $(\bL, \delta)_h$-regular along $\tiE^N_{p_0}$. By \propref{hor angle shrink}, we conclude that as $N \to \infty$, the direction $\tiE^N_{p_0}$ converges exponentially fast to $E^c_{p_0}$, along which $p_0$ is infinitely backward $(\bL, \delta)_h$-regular.
\end{proof}

\begin{proof}[Proof of \thmref{ref uni} c)]
For ${t} \in (0, {L_1}^{-1}]$, let $L := 1/\underline{{t}}$. Let $\bfp \in \Lambda$ be a point satisfying
$$
\vd_{t}(\bfp) = \cd_{t}(\bfp) = \varnothing.
$$
By \thmref{ref uni} a), $\bfp$ is forward $(L, \delta)_v$-regular along $\hE^v_{\bfp}$. We show that $\bfp$ is also backward $(\bL, \delta)_v$-regular along $\hE^v_{\bfp}$.

Let $-m <0$ be the first moment in the backward orbit of $\bfp$ such that
$$
p_0 := F^{-m}(\bfp) \in \cT_{c_0}({t}).
$$
Write
$
z_0 = (x_0, y_0) := \Phi_{c_0}(p_0).
$
Let $-l<-j < 0$ be the slow and fast tunnel escape moment of $p_0$ respectively, and let $n >0$ be the crescent escape moment of $p_1$. Recall that
$$
l > \udelta^{-1}\logl |x_0| > \udelta^{-1}\logl L
\matsp{and}
j, n > \delta^{-1}\logl |x_0| > -k\delta^{-1}\logl L
$$
for some uniform constant $k >0$. Observe that
$
\cd_{t}(p_i) = \varnothing
$ for $
1\leq i \leq m.
$
Moreover,
$$
\cd_{t}(p_i) = \vd_{t}(p_i) = \varnothing
\matsp{for}
i \in \{-j, -l\}.
$$

By \thmref{ref uni} a) and \lemref{consistent for to back}, $p_m$ is $m$-times backward $(L, \delta)_v$-regular along $\hE^v_{p_m}$. Arguing as in the proof of \thmref{ref uni} b), we see that $p_0$ is $l$-times backward $(\bL, \delta)_v$-regular along $\hE^v_{p_0}$.

Let $m_1 > 1$ be a Pliss moment in the forward orbit of $c_1$ so that
$$
Kn \leq m_1 < (1+\bepsilon) Kn,
$$
where $K \geq 1$ is a sufficiently large uniform constant. If $n < m < m_1$, then arguing as in the proof of \lemref{1 tunnel recover}, we see that $\hE^v_{p_m}$ is sufficiently vertical in $\cU_{c_m}$. Since $c_m$ is $m$-times backward $(\lambda^{-\bepsilon m}, \epsilon)_v$-regular along $E^v_{c_m}$, we have
$$
\|DF^{-m}|_{\hE^v_{p_m}}\| > \lambda^{\bepsilon m}\lambda^{-(1-\epsilon) m} > \lambda^{-(1-\bepsilon) m}.
$$
If $m > m_1$, then again, $\hE^v_{p_{m_1}}$ is sufficiently vertical in $\cU_{c_{m_1}}$, and we have
$$
\|DF^{-m_1}|_{\hE^v_{p_{m_1}}}\| > \lambda^{-(1-\bepsilon) m_1} > \lambda^{-(\delta - \bepsilon)Kn}\lambda^{-(1-\delta)m_1} > \bL\lambda^{-(1-\delta)m_1}.
$$
By \propref{stable center align}, $\hE^v_{p_1}$ is sufficiently stable-aligned in $\cU_{c_1}$, and hence, $\hE^v_{p_{-j}}$ is sufficiently vertical in $\cU_{c_{-j}}$. Thus,
$$
\|DF^{-(l-j)}|_{\hE^v_{p_{-j}}}\| > \lambda^{-(1-\bepsilon) (l-j)} > \lambda^{-(\delta - \bepsilon)(l-j)}\lambda^{-(1-\delta)(l-j)} > L^{\udelta^{-1}}\lambda^{-(1-\delta)(l-j)}.
$$

Concatenating the above estimates, we conclude that $\bfp = p_m$ is $(m+l)$-times backward $(\bL, \delta)_v$-regular along $\hE^v_{\bfp}$. Moreover, since
$$
\|D_{\bfp}F^{-(m+l)}|_{\hE^v_{\bfp}}\| > \lambda^{-(1-\delta) (m+l)},
$$
we see that if $p_{-(m+l)}$ is $M$-times backward $(\bL, \delta)_v$-regular along $\hE^v_{p_{-(m+l)}}$, then $\bfp$ is $(m+l+M)$-times backward $(\bL, \delta)_v$-regular along $\hE^v_{\bfp}$. The result follows by induction.
\end{proof}


\subsection{Irregular points}

\begin{prop}\label{trap irreg}
Let $\Delta_{c_0}^{c/fc}$ and $\Delta_{c_1}^v$ be the critical/full critical dust of $c_0$ and the valuable dust of $c_1$ respectively. For $K \geq 1$, let $\Lambda_K^{+/-/P}$ be the set of all infinitely forward/backward/Pesin $(K, \delta)$-regular points in $\Lambda$. Then we have
$$
\Lambda \setminus \left(\bigcup_{K \geq 1} \Lambda_K^+\right) \subset \Delta_{c_0}^c
\matsp{and}
\Lambda \setminus \left(\bigcup_{K \geq 1} \Lambda_K^-\right) \subset \Delta^v_{c_1}.
$$
Moreover,
$$
\Lambda \setminus \left(\bigcup_{K \geq 1} \Lambda_K^P\right) \subset \Delta_{c_0}^{fc} \cup \{c_m\}_{m \in \bbZ}.
$$
\end{prop}

\begin{proof}
We prove the last containment, as the other two can be proved by a similar and simpler argument.

If
$$
p_0 \notin \Delta_{c_0}^{fc} := \Delta_{c_0}^c \cup \Delta^v_{c_1},
$$
then
$
\cd(p_0) \neq \infty
$ and $
\vd(p_0) \neq \infty.
$
By replacing $p_0$ by $p_{-\vd(p_0)-1}$ if necessary, we can assume that
$
\vd(p_0) = \varnothing.
$
If
$
d := \cd(p_0) \neq \varnothing,
$
let $-j<0$ be the tunnel escape moment for $p_d$. Then by \propref{tunnel inv}, we see that for
$
q_0 := p_{d-j},
$
we have
$
\cd(q_0) = \vd(q_0) = \varnothing.
$
Thus, $p_0$ is Pesin regular by \thmref{ref uni} c).
\end{proof}

\begin{prop}\label{unique crit}
If $F$ is regularly unicritical on $\Lambda$, then there exists a unique regular critical orbit $\{c_m\}_{m\in\bbZ} \subset \Lambda$.
\end{prop}

\begin{proof}
Let $q_1 \in \Lambda \setminus \{c_1\}$ be a generalized critical value. Since $q_1$ cannot be Pesin regular, $q_1 \in \Delta_{c_0}^{fc}$ by \propref{trap irreg}. For concreteness, assume $q_1 \in \Delta_{c_0}^c$. Then $q_1 \in \cT_{c_{-N}}$ for $N$ arbitrarily large. Let $-l<0$ be the slow tunnel escape moment for $q_{1+N}$. Then we have
$
l > \bdelta_1^{-1}N.
$ 
Since $c_{-N}$ is $(l-N, N)$-times $(\lambda^{-\bepsilon N}, \epsilon)_v$-regular along $\hE^v_{c_{-N}}$, we contradict \propref{no crit in reg nbh}.
\end{proof}


\section{Total Ordering Inside a Unicritical Cover}\label{sec:tot order}

Let $F$ be the mildly dissipative, infinitely renormalizable and regularity unicritical diffeomorphism considered in \secref{sec:unicrit}.

In \thmref{ref uni}, choose
$$
t = \bft := {L_1}^{-1}
\matsp{and}
L := 1/\underline{\bft}.
$$
Lastly, denote
$$
\cT_{c_m} := \cT_{c_m}(\bft)
\matsp{for}
m \in \bbZ,
$$
and
$$
\cd := \cd_\bft
\matsp{and}
\vd := \vd_\bft,
$$
where the $t$-critical depth $\cd_t$ and the $t$-valuable depth $\vd_t$ are defined as in \subsecref{subsec:crit tunnel}.


\subsection{Outer boundaries of tunnels and crescents}
Let
$$
\delta' := (1-\bdelta)\delta.
$$
For $n \geq 0$ and $i \in \{0, 1\}$, denote
$$
U_{c_i}^{\delta'n}(a) := \bbB(a\lambda^{\delta' n} l_{c_i}, l_{c_i}) := (-a\lambda^{\delta' n}l_{c_i}, a\lambda^{\delta' n}l_{c_i}) \times (-l_{c_i}, l_{c_i}) \subset \bbR^2,
$$
where $l_{c_i}>0$ is the regular radius at $c_i$. For $a > 0$, let
$$
T_{c_0}(a) := \{(x,y) \in U_{c_0} \; | \; |x| < a\bft \; \text{and} \; |y| < |x|^{1/\udelta}\}
$$
and
$$
\bpartial T_{c_0}(a) := \{(x,y) \in U_{c_0} \; | \; |x| < a\bft \; \text{and} \; |y| = |x|^{1/\udelta}\}.
$$
Let $j_n \geq 0$ be the largest integer such that $j_n \leq n$ and $-j_n$ is a Pliss moment in the backward orbit of $c_0$. Define
\begin{equation}\label{eq:tunnel a}
\cT_{c_{-n}}(a) := F^{-n}\circ \Phi_{c_0}^{-1}(T_{c_0}(a) \cap U_{c_0}^{\delta' j_n}(a))
\end{equation}
and
$$
\bpartial\cT_{c_{-n}}(a) := F^{-n}\circ \Phi_{c_0}^{-1}(\bpartial T_{c_0}(a) \cap U_{c_0}^{\delta' j_n}(a)).
$$
Note that $\cT_{c_{-n}} = \cT_{c_{-n}}(1)$. The {\it outer boundary of $\cT_{c_{-n}}$} is defined as
$$
\bpartial \cT_{c_{-n}} := \bpartial \cT_{c_{-n}}(2).
$$

A Jordan arc $\Gamma : (0, 1) \to \cT_{c_{-n}}(2)$ is said to be {\it vertically proper in $\cT_{c_{-n}}$} if $\Gamma$ extends continuously to $[0, 1]$, and $\Gamma(0)$ and $\Gamma(1)$ are in distinct components of $\bpartial \cT_{c_{-n}} \setminus \{c_{-n}\}$. More generally, a Jordan arc $\Gamma$ in $\Omega$ is said to be vertically proper in $\cT_{c_{-n}}$ if $\Gamma \cap \cT_{c_{-n}}(2)$ is.

For $a > 0$, let
$$
T_{c_1}(a) := F_{c_0}(T_{c_0}(a))
\matsp{and}
\bpartial T_{c_1}(a) := F_{c_0}(\bpartial T_{c_0}(a)),
$$
where
$
F_{c_0} := \Phi_{c_1}\circ F \circ \Phi_{c_0}^{-1}.
$
For $n \geq 1$, let $i_n \geq 1$ be the largest integer such that $i_n \leq n$ and $i_n$ is a Pliss moment in the forward orbit of $c_1$. Define
\begin{equation}\label{eq:crescent a}
\cT_{c_n}(a) := F^{n-1}\circ \Phi_{c_1}^{-1}(T_{c_1}(a) \cap U_{c_1}^{\delta' i_n}(a))
\end{equation}
and
$$
\bpartial\cT_{c_n}(a) := F^{n-1}\circ \Phi_{c_1}^{-1}(\bpartial T_{c_1}(a) \cap U_{c_1}^{\delta' i_n}(a)).
$$
Note that $\cT_{c_n} = \cT_{c_n}(1)$. The {\it outer boundary of $\cT_{c_n}$} is defined as
$$
\bpartial \cT_{c_n} := \bpartial \cT_{c_n}(2).
$$
Let
$$
\bbL^{h, +}_0 := \{(x, 0) \; | \; 0 < x < l_{c_1}\}
$$
and
$$
W^{h, +}_{c_n} := F^{n-1} \circ \Phi_{c_1}^{-1}(\bbL^{h, +}_0 \cap U_{c_1}^{\delta' i_n}(2)).
$$
Denote
$$
\tilde \cU_{c_1}^{i_n}(2) := \cU_{c_1}^{\delta' i_n}(2) \cap \cU_{c_1}^0(2\bft^2).
$$
The {\it crescent gap $\cT^g_{c_n}$ at $c_n$} is defined to be the connected component of $F^{n-1}(\tilde \cU_{c_1}^{\delta' i_n}(2)) \setminus \cT_{c_n}(2)$ that contain $W^{h, +}_{c_n}$. The {\it filled valuable crescent at $c_n$} is defined as
$$
\cT^f_{c_n} := \cT_{c_n}^2 \cup \cT^g_{c_n}.
$$

A Jordan arc $\Gamma : (0, 1) \to \cT^f_{c_n}$ is said to be {\it vertically proper in $\cT_{c_n}$} if $\Gamma$ is a concatenation of three curves $\Gamma_1 : (0, a_1) \to \cT_{c_n}(2)$; $\Gamma_2 : [a_1, a_2] \to \cT^g_{c_n}$ and $\Gamma_3 : (a_2, 1) \to \cT_{c_n}(2)$ such that $\Gamma_1$ and $\Gamma_3$ extends continuously to $[0, a_1]$ and $[a_2, 1]$ respectively; and for $i \in \{1, 2, 3\}$, the endpoints of $\Gamma_i$ are in distinct components of $\bpartial \cT_{c_n} \setminus \{c_n\}$. More generally, a Jordan arc $\Gamma$ in $\Omega$ is said to be vertically proper in $\cT_{c_n}$ if $\Gamma \cap \cT_{c_n}(2)$ is.

\begin{lem}\label{proper at value}
If a $C^1$-curve $\Gamma$ is sufficiently stable-aligned in $\cU_{c_1}$, and
$$
\Gamma \cap \cT_{c_1}^f \neq \varnothing,
$$
then $\Gamma$ is vertically proper in $\cT_{c_1}$.
\end{lem}

Let
$$
\hdelta := (1+\bdelta)\delta.
$$
For $a > 0$, let
$$
\hT_{c_0}(a) := \{(x,y) \in U_{c_0} \; | \; |x| < a\bft \; \text{and} \; |y| < |x|^{1/\hdelta}\} \supset T_{c_0}(a).
$$
For $n \geq 0$, define $\hcT_{c_{-n}}(a)$ by replacing $T_{c_0}(a)$ in \eqref{eq:tunnel a} with $\hT_{c_0}(a)$. Let $\hcU_{c_{-n}}$ be a regular neighborhood at $c_{-n}$ associated with a linearization of its $(\infty, n)$-time orbit with vertical direction $\hE^v_{c_{-n}}$. This means, in particular, that whenever $-n$ is a Pliss moment in the backward orbit of $c_0$, we have
$
\radius(\hcU_{c_{-n}}) \asymp 1.
$
The set
\begin{equation}\label{eq:thick tunnel}
\hcT_{c_{-n}} := \hcT_{c_{-n}}(1) \cap \hcU_{c_{-n}} \supset \cT_{c_{-n}}
\end{equation}
is referred to as a {\it thickened critical tunnel at $c_{-n}$}. A {\it thickened valuable crescent $\hcT_{c_{1+n}}$ at $c_{1+n}$} is defined by replacing $T_{c_1}(a)$ in \eqref{eq:crescent a} with
$$
\hT_{c_1}(a) := \Phi_{c_1}\circ F\circ \Phi_{c_0}^{-1}(\hT_{c_0}(a)),
$$
and letting
\begin{equation}\label{eq:thick crescent}
\hcT_{c_{1+n}} := \hcT_{c_{1+n}}(1) \supset \cT_{c_{1+n}}.
\end{equation}

\begin{lem}\label{fit in thick tunnel}
Let $0 \leq m < n$. If $c_{-n} \in \cT_{c_{-m}}$, then
$
\hcT_{c_{-n}} \subset \hcT_{c_{-m}}.
$
If $c_n \in \cT_{c_{-m}}$, then
$
\cT_{c_n} \subset \hcT_{c_{-m}}.
$
\end{lem}

\begin{proof}
By \propref{tunnel inv}, we have
$
n > \bepsilon^{-1}m.
$
Write
$
(x_0, y_0) := \Phi_{c_0}(c_{-n+m}).
$
Since
$$
\radius(\cU_{c_{-n+m}}) > \lambda^{\bepsilon(n-m)},
$$
we have
$$
n > \bepsilon^{-1}\logl |x_0|
$$
by \propref{no crit in crit nbh}. Hence, by \propref{tunnel size}
$$
\diam(\hcT_{c_{-n+m}}) < \lambda^{\delta n} < |x_0|^{1/\bepsilon}.
$$
Since
$$
\bbD_{c_{-n+m}}(k|x_0|^{1/\hdelta}) \subset \hcT_{c_0},
$$
for some uniform constant $k > 0$, the first claim follows.

The proof of the second claim is similar, and will be omitted.
\end{proof}


\subsection{Disjoint vertical manifolds}

Let $p \in \Lambda \setminus \{c_{-n}\}_{n = 0}^\infty$. If $p \notin \Delta^c_{c_0}$ (where $\Delta^c_{c_0}$ is the critical dust of $c_0$), then by \propref{trap irreg}, $p$ has a well-defined local strong-stable manifold $W^{ss}_{\loc}(p) \subset \cU_p$. If $\cd(p) = \varnothing$, define
$$
\hW^v(p) := W^{ss}_{\loc}(p).
$$
If $\cd(p) = d \in [0, \infty)$, define
$$
\hW^v(p) := W^{ss}_{\loc}(p) \cap \hcT_{c_{-d}}.
$$

For $n \geq 0$, we have $c_{-n} \notin \Delta^c_{c_0}$ by \propref{tunnel inv}. Let $\hcU_{c_{-n}}$ be a regular neighborhood at $c_{-n}$ associated with a linearization of its $(\infty, n)$-time orbit with vertical direction $\hE^v_{c_{-n}}$. For $q \in \cT_{c_{-n}}$, define $W^v_{c_{-n}}(q)$ to be the local vertical manifold at $q$ in $\hcU_{p_0}$. Moreover, let
$$
W^v_{c_{-n}}(\cT_{c_{-n}}) := \bigcup_{q \in \cT_{c_{-n}}} W^v_{c_{-n}}(q)
$$
(which is not a curve but an open set foliated by vertical curves). If $\cd(c_{-n}) = \varnothing$, define
$$
\hW^v(q) := W^v_{c_{-n}}(q)
\matsp{for}
q \in \cT_{c_{-n}}
$$
and
$$
\hW^v(\cT_{c_{-n}}) := W^v_{c_{-n}}(\cT_{c_{-n}}).
$$
If $\cd(c_{-n}) = d \in [0, \infty)$, define
$$
\hW^v(q) := W^v_{c_{-n}}(q) \cap \hcT_{c_{-d}}
\matsp{for}
q \in \cT_{c_{-n}}
$$
and
$$
\hW^v(\cT_{c_{-n}}) := W^v_{c_{-n}}(\cT_{c_{-n}}) \cap \hcT_{c_{-d}}.
$$

\begin{prop}\label{vertical disjoint}
Let $p, q \in \Lambda \setminus \Delta^c_{c_0}$. If $p$ and $q$ are not contained in the same connected component of $\Lambda$, then
$
\hW^v(p) \cap \hW^v(q) = \varnothing.
$
\end{prop}

\begin{proof}
The result is obvious unless either $p$ or $q$ is equal to $c_{-n}$ for some $n \geq 0$. For concreteness, suppose it is the former. Let $\hcU_{c_{-n}}$ be a regular neighborhood at $c_{-n}$ associated with a linearization of its $(\infty, n)$-time orbit with vertical direction $\hE^v_{c_{-n}}$. Then
$
F^n(\hcU_{c_{-n}}^{\bdelta n}) \subset \cU_{c_0}.
$

If
$$
\hW^v(q) \cap \hW^v(c_{-n}) \neq \varnothing,
$$
then we have
$$
F^n(\hW^v(q)) \subset \hW^v(F^n(q)) \subset \hcT_{c_0}.
$$
This contradicts the fact that
$
\hcT_{c_0} \cap \hW^v(c_0) = \varnothing.
$
\end{proof}


\subsection{Unicritical cover}\label{subsec:unicrit cover}

For $p \in \Lambda^P_{L, \delta}$, let $\bcB_p$ be an enclosed regular box at $p$ with outer boundary $\bpartial \bcB_p$ (see \secref{sec:mild diss}). The collection
$$
\cC := \{\bcB_p\}_{p \in \Lambda^P_{L, \delta}} \cup \{\cT_{c_m}\}_{m \in \bbZ}
$$
is a cover of $\Lambda$ by (pinched) neighborhoods. We refer to $\cC$ as a {\it unicritical cover of $\Lambda$}.

\begin{prop}\label{nice cover}
By slightly adjusting the sizes of the elements in $\cC$ if necessary, the following properties hold for all $Q \in \cC$.
\begin{enumerate}[i)]
\item If for some $p \in \Lambda \setminus \Delta^c_{c_0}$, we have $\hW^v(p) \cap Q \neq \varnothing$, then $p \in Q$.
\item If for some $n \geq 0$, we have $\hW^v(\cT_{c_{-n}}) \cap Q \neq \varnothing$, then $\cT_{c_{-n}} \subset Q$.
\end{enumerate}
\end{prop}

\begin{proof}
Let $\{-j_i\}_{i=0}^\infty$ be the decreasing sequence of Pliss moments in the backward orbit of $c_0$. For $i \geq 0$, let $a_i \in (\lambda^{\udelta}, 1]$. Consider a sequence $\{\chcT_{c_{-n}}\}_{n=0}^\infty$ of slightly smaller critical tunnels that satisfy
$$
\cT_{c_{-j_i}}(a_i) \subset \chcT_{c_{-j_i}} \subset \cT_{c_{-j_i}}
$$
and
$$
\chcT_{c_{-j_i -l}} = F^{-l}(\chcT_{c_{-j_i}})
\matsp{for}
0\leq l < j_{i+1} - j_i.
$$
Observe that all previous results about critical tunnels still hold if we replace $\{\cT_{c_{-n}}\}_{n = 0}^\infty$ by $\{\chcT_{c_{-n}}\}_{n = 0}^\infty$.

Let $p_1 \in \Lambda \cap \cT_{c_1}$ with $\cd(p_1) = \varnothing$. Write
$
(x_0, y_0) := \Phi_{c_0}(p_0).
$
By \thmref{ref uni} and \propref{stable center align}, $\hW^v(p_1)$ is sufficiently stable-aligned in $\cU_{c_1}$. If $p_0 = c_{-m}$ for some $m > 0$, then by \lemref{return dist lem}, we have
$
\diam(\cT_{c_{-m}}) < |x_0|^{1/\bepsilon}.
$

If $p_0 = c_{-m}$, let $q_0 \in \cT_{c_{-m}}$ and
$
\tiW^v(q_0) := W^v_{c_{-m}}(q_0).
$
Otherwise, let $q_0 = p_0$. Let $\tiE^v_{q_0} \in \bbP^2_{q_0}$ be the direction tangent to $\tiW^v(q_0)$. Then $\tiE^v_{q_1}$ and $\tiW^v(q_1)$ are sufficiently stable-aligned in $\cU_{c_1}$. Let $E^{v/h}_{q_0}$ be the vertical/horizontal direction at $q_0$ in $\cU_{c_0}$. Then
$$
\measuredangle(\tiE^v_{q_0}, E^h_{q_0}) > k|x_0|
$$
for some uniform constant $k > 0$. 

Suppose
$
q_{-j_i} \notin \cT_{c_{-j_i}}
$ and $
\tiW^v(q_{-j_i}) \cap \cT_{c_{-j_i}} \neq \varnothing.
$
By \propref{tunnel escape}, $q_{-j_i}$ is $j_i$-times forward $(1, \delta)_v$-regular along $\tiE^v_{q_{-j_i}}$. By \propref{vert angle shrink}, we see that $\tiW^v(q_{-j_i})$ is $\lambda^{(1-\bdelta)j_i}$-vertical in $\cU_{c_{-j_i}}$. On the other hand,
$$
|x_{-j_i}| > \lambda^{\bepsilon j_i}|x_0| > \lambda^{(1-\bdelta)\delta j_i}.
$$
It follows that
$$
\tiW^v_{q_{-j_i}} \cap \cT_{c_{-j_i}}(1-\lambda^{(1-\bdelta)j_i}) =\varnothing.
$$
Thus, we may choose $\chcT_{c_{-j_i}}$ such that
$$
\cT_{c_{-j_i}} \supset \chcT_{c_{-j_i}} \supset \cT_{c_{-j_i}}(1-\lambda^{(1-\bdelta)j_i}) \supset \cT_{c_{-j_i}}(\lambda^{\udelta}),
$$
and
$
\tiW^v_{q_{-j_i}} \cap \chcT_{c_{-j_i}} =\varnothing.
$

\end{proof}

For $Q \in \cC$, define its {\it critical depth} $\cd(Q) \in \bbN \cup \{\varnothing, 0\}$ as follows.
\begin{enumerate}[i)]
\item If $Q = \cT_{c_{-m}}$ for some $m \geq 0$, then $\cd(Q) := m$.
\item If $Q = \cT_{c_n}$ for some $n \geq 1$, then $\cd(Q) := \cd(c_n)$.
\item Otherwise, $\cd(Q) := \varnothing$.
\end{enumerate}
Let $X_Q$ be the set of all points $p \in \Lambda \cap Q$ such that $\cd(p) = \cd(Q)$.

\begin{prop}\label{vertical proper}
Let $Q \in \cC$. Then for all $p \in X_Q$, the curve $\hW^v(p)$ is vertically proper in $Q$. If $c_{-n} \in X_Q$ for some $n \geq 0$, then for all $q \in \cT_{c_{-n}} \subset Q$, the curve $W^v_{c_{-n}}(q)$ is vertically proper in $Q$.
\end{prop}

\begin{proof}
If $Q$ is a regular box, then the result is clearly true. If $Q = \cT_{c_1}$, then the result follows immediately from \thmref{ref uni}, \propref{stable center align} and \lemref{proper at value}. If $Q = \cT_{c_{-m}}$ or some $m \geq 0$ or $Q = \cT_{c_k}$ for some $k > 1$, then the result follows immediately from the previous case and \lemref{fit in thick tunnel}.
\end{proof}


\subsection{Total ordering with blowups}\label{subsec:blow up}

Let $(X, \leq_X)$ be a totally ordered set; $U$ a subset of $X$; and $\{(Y_u, \leq_{Y_u})\}_{u \in U}$ a collection of totally ordered sets. Define $\cY : U \to \{Y_u\}_{u \in U}$ by
$
\cY(u) := Y_u.
$
The {\it blow up of $X$ by $\cY$} is defined as the set 
$$
X[\cY] := (X \setminus U) \sqcup \bigsqcup_{u \in U} Y_u
$$
with the total order $\leq_{X[\cY]}$ given by the following properties. Let $x, y \in X[\cY]$. Then $x \leq_{X[\cY]} y$ if one of the following holds:
\begin{enumerate}[i)]
\item $x,y \in X \setminus U$ and $x \leq_X y$;
\item $x \in X \setminus U$ and $y \in Y_u$ with $x \leq_X u$;
\item $x \in Y_u$ and $y \in X \setminus U$ with $u \leq_X y$;
\item $x \in Y_u$ and $y \in Y_v$ with $u \leq_X v$; or
\item $x,y \in Y_u$ and $x \leq_{Y_u} y$.
\end{enumerate}

For $n \geq 0$, consider sets $Y^n$ and $U^n$, and a set-valued function $\cY^n$ obtained inductively as follows. First, let $U^0$ be a subset of $Y^0 := X$. Suppose that $Y^{n-1}$ and $U^{n-1}$ are given. Let $\{(Y^n_u, \leq_{Y^n_u})\}_{u \in U^{n-1}}$ be a collection of totally ordered sets, and define
$$
Y^n := \bigsqcup_{u \in U^{n-1}}Y^n_u.
$$
Next, let $\cY^n: U^{n-1} \to \{Y^n_u\}_{u \in U^{n-1}}$ be given by
$
\cY^n(u) := Y^n_u.
$
Lastly, let $U^n$ be a subset of $Y^n$.

Let $X^{\wedge 0} := X$, and for $n \geq 0$, inductively define an {\it $n$th blow-up of $X$} by
$$
X^{\wedge n} := X^{\wedge (n-1)}[\cY^n]
$$
with the total order $\leq_{X^{\wedge n}}$ defined as above. Next, we define an {\it infinite blow-up of $X$} by
$$
X^{\wedge \infty} := \{x \in X^{\wedge n} \setminus U^n \; \text{for some} \; n\geq 0\}
$$
with the total order $\leq_{X^{\wedge \infty}}$ given by the following property. Let $x,y \in X^{\wedge \infty}$. Then $x \leq_{X^{\wedge \infty}} y$ if $x \leq_{X^\wedge n} y$, where $n$ is any integer such that $x, y \in X^n \setminus U^n$.

A sequence
$
u_\infty := \{u_n\}_{n =0}^\infty
$
such that $u_0 \in U^0 \cap X$ and $u_n \in U^n\cap Y^n_{u_{n-1}}$ for $n > 0$ is called a {\it limit point} of $X^{\wedge \infty}$. Define the {\it completion $X^{\wedge \omega}$ of $X^{\wedge \infty}$} as the disjoint union of $X^{\wedge \infty}$ and the set of all its limits points, together with the total order $\leq_{X^{\wedge \omega}}$ given by the following property. Let $x, y \in X^{\wedge \omega}$. Then $x \leq_{X^{\wedge \omega}} y$ if
\begin{enumerate}[i)]
\item $x, y \in X^{\wedge \infty}$ and $x \leq_{X^{\wedge \infty}} y$;
\item $x \in X^{\wedge n} \setminus U^n$ and $y = \{u_i\}_{i =0}^\infty$ with $x \leq_{X^{\wedge n}} u_n$;
\item $x = \{u_i\}_{i =0}^\infty$ and $y \in X^{\wedge n} \setminus U^n$ with $u_n \leq_{X^{\wedge n}} y$; or
\item $x = \{u_i\}_{i =0}^\infty$ and $y = \{v_i\}_{i =0}^\infty$, and $u_n \leq_{X^{\wedge n}} v_n$, where $n$ is the smallest integer such that $u_n \neq v_n$.
\end{enumerate}


\subsection{Total order on $\Lambda_Q$}

Let $Q \in \cC$, and let $\cG$ be a collection of disjoint Jordan curves that are vertically proper in $Q$. Then we define a total order on $\cG$ as follows. If $Q = \cT_{c_n}$ for some $n \geq 1$, let $\hQ := \cT_{c_n}^f$. Otherwise, let $\hQ := Q$. Note that for any $\Gamma_0 \in \cG$, the complement $\hQ \setminus \Gamma_0$ is the disjoint union of exactly two connected components, say $\hQ^+_{\Gamma_0}$ and $\hQ^-_{\Gamma_0}$. Thus, there is a unique (strict) total order $<_\cG$ on $\cG$ such that for any $\Gamma \in \cG$ distinct from $\Gamma_0$, we have $\Gamma <_\cG \Gamma_0$ if $\Gamma\cap \hQ \subset \hQ^-_{\Gamma_0}$ and $\Gamma_0 <_\cG \Gamma$ if $\Gamma\cap \hQ \subset \hQ^+_{\Gamma_0}$.

For $Q \in \cC$, let $Y_Q$ be the set consisting of all elements $x$ that satisfy one of the following properties.
\begin{enumerate}[i)]
\item $x = p \in X_Q$ such that $p \neq c_{-n}$ for all $n \geq 0$. In this case, define $\Gamma^v_x := W^v_p(p)$.
\item $x = \cT_{c_{-n}}$ for some $n \geq 0$ such that $c_{-n} \in X_Q$. In this case, define $\Gamma^v_x := W^v_{c_{-n}}(c_{-n})$.
\end{enumerate}
Let $U_Q$ be the set of all elements in $Y_Q$ of type $ii)$. By Propositions \ref{vertical disjoint} and \ref{vertical proper}, the collection
$
\Gamma^v_Q := \{\Gamma^v_x\}_{x \in Y_Q}
$
consists of pairwise disjoint curves that are vertically proper in $Q$. As discussed above, this implies that $\Gamma^v_Q$ has a total order $\leq_{\Gamma^v_Q}$ (non-strict, since we may have $\Gamma^v_x = \Gamma^v_y$ with $x \neq y$), which induces a total order $\leq_{Y_Q}$ on $Y_Q$.

Let $Y^0 := Y_Q$ and $U^0 := U_Q$. Proceeding inductively, suppose that $Y^{n-1}$ and $U^{n-1}$ are given. For $x \in U^{n-1}$, define $Y^n_x := Y_x$; $U^n_x := U_x$; and
$$
Y^n:= \bigsqcup_{x \in U^{n-1}} Y^n_x
\matsp{and}
U^n := \bigsqcup_{x \in U^{n-1}} U^n_x.
$$
By the construction given in \subsecref{subsec:blow up}, we obtain the completion $Y_Q^\omega$ of an infinite blow-up of $Y_Q$ with a total ordering $<_{Y_Q^{\wedge \omega}}$.

\begin{thm}\label{thm:total order}
For $Q \in \cC$, let $\Lambda_Q :=\Lambda \cap Q$. Then
$$
\Lambda_Q = Y_Q^{\wedge \omega},
$$
and the set of limit points in $Y_Q^{\wedge \omega}$ is given by $\Lambda_Q \cap \Delta^c_{c_0}$. Thus, $\Lambda_Q$ has a total order $\leq_Q$ such that for $x,y\in \Lambda_Q$, we have $x \leq_Q y$ and $y \leq_Q x$ if and only if $y \in W^s(x)$.
\end{thm}

\begin{proof}
The result follows immediately from \propref{tunnel size}.
\end{proof}

Let $Q \in \cC$. A point $x \in X \subset \Lambda_Q$ is called an {\it extremal point of $X$} if either
\begin{itemize}
\item $x \leq_{Q} y$ for all $y \in X$; or
\item $x \geq_{Q} y$ for all $y \in X$.
\end{itemize}

The following three propositions record observations that follow immediately from the construction of the total order on $\Lambda_Q$.

\begin{prop}\label{order trans}
Let $Q, Q' \in \cC$, and let $X := \Lambda_Q \cap \Lambda_{Q'}$. Then $\leq_{Q}$ and $\leq_{{Q'}}$ induces the same total order on $X$ (up to orientation). That is, either
\begin{itemize}
\item $x \leq_{{Q'}} y$ for all $x, y \in X$ with $x \leq_{Q} y$; or
\item $x \geq_{{Q'}} y$ for all $x, y \in X$ with $x \leq_{Q} y$.
\end{itemize}
\end{prop}

\begin{prop}\label{order pres}
Let $Q, Q' \in \cC$ with $Q \neq \cT_{c_0}$, and let $X \subset \Lambda_Q$ be a set satisfying $F(X) \subset \Lambda_{Q'}$. Then $F$ preserves the total order on $X$ (up to orientation). That is, either
\begin{itemize}
\item $F(x) \leq_{{Q'}} F(y)$ for all $x, y \in X$ such that $x \leq_{Q} y$; or
\item $F(x) \geq_{{Q'}} F(y)$ for all $x, y \in X$ such that $x \leq_{Q} y$.
\end{itemize}
\end{prop}

\begin{prop}\label{order value}
Let $X \subset \Lambda_{\cT_{c_0}}$. Then the following holds.
\begin{enumerate}[i)]
\item If $c_0 \leq_{{\cT_{c_0}}} x_0$ for all $x \in X$, then for all $y, z \in X$ such that $y \leq_{{\cT_{c_0}}} z$, we have $F(y) \leq_{{\cT_{c_1}}} F(z)$.
\item If $x \leq_{{\cT_{c_0}}} c_0$ for all $x \in X$, then for all $y, z \in X$ such that $y \leq_{{\cT_{c_0}}} z$, we have $F(z) \leq_{{\cT_{c_1}}} F(y)$.
\end{enumerate}
\end{prop}

Let $Q \in \cC$. We define a larger covering element $\hQ \supset Q$ as follows. If $Q = \cT_{c_m}$ for some $m \in \bbZ$, then let $\hQ := Q$. Otherwise, $Q = \bcB_p$ for some $p \in \Lambda^P_{L, \delta}$. Recall that $\bcB_p$ is a quadrilateral bound by four nearly linear arcs for which the opposite sides are nearly parallel. Moreover, the horizontal width and the vertical height of $\bcB_p$ are approximately $1/L_1$ and $1/L_2$ respectively, where
$$
\bL < L_1
\matsp{and}
\overline{L_1} < L_2.
$$
Let $\hQ$ be a strictly larger encased regular box at $p$ whose horizontal width and the vertical height are approximately $1/L_1'$ and $1/L_2'$ respectively, where
$$
\bL < L_1' < \underline{L_1}
\matsp{and}
\overline{L_1'} < L_2' < \underline{L_2}.
$$

\begin{prop}\label{trap middle}
For $0 \leq i \leq n$, let $Q_i \in \cC \setminus \{\cT_{c_0}\}$. Suppose there exist $x_0, y_0, z_0 \in \Lambda_{Q_0}$ such that
$$
x_0 \leq_{{Q_0}} y_0 \leq_{{Q_0}} z_0,
$$
and $x_i, z_i \in \Lambda_{Q_i}$ for all $1 \leq i \leq n$. Then $y_n \in \Lambda_{\hQ_n}$ and
$$
x_n \leq_{{\hQ_n}} y_n \leq_{{\hQ_n}} z_n.
$$
\end{prop}

\begin{proof}
The result clearly holds if $Q_i = \cT_{c_m}$ for some $m \in \bbZ$. Suppose the $Q_i$'s are encased regular boxes. Let $\Gamma(x_n)$ and $\Gamma(z_n)$ be two proper sufficiently vertical arcs in $\hQ_n$ containing $x_n$ and $z_n$ respectively. By \propref{pres ver hor curv}, the preimages $\Gamma(x_0)$ and $\Gamma(z_0)$ of $\Gamma(x_n)$ and $\Gamma(z_n)$ under $F^{-n}$ are proper and sufficiently vertical in $Q_0$. Let $Q_0(x_0, z_0)$ be the component of $Q_0 \setminus (\Gamma(x_0) \cup \Gamma(z_0))$ containing $y_0$. Applying \propref{pres ver hor curv} again, we see that the image of $Q_0(x_0, z_0)$ under $F^i$ is still nearly a parallelogram whose height increased by no more than a factor of $\bL$. It follows that $F^i(Q_0(x_0, z_0)) \subset \hQ_i$. The claim about the order of $x_n, y_n$ and $z_n$ follows immediately from \propref{order pres}.
\end{proof}


\section{Connected Components of the Limit Set}

Let $F$ be the mildly dissipative, infinitely renormalizable and regularity unicritical diffeomorphism considered in \secref{sec:tot order}.

The {\it forward/backward/Pesin $\delta$-regularity of $p \in \Lambda$} is the smallest value $\cL^{+/-/P}_\delta(p) \in [1, \infty]$ such that $p$ is forward/backward/Pesin $(\cL^{+/-/P}_\delta(p), \delta)$-regular. We say that $p$ is {\it forward/backward/Pesin $\delta$-regular} if $\cL^{+/-/P}_\delta(p) < \infty$ (i.e. $p \in \Lambda^{+/-/P}_\delta$).

Let $Z$ be the connected component of $\Lambda$. Observe that either
$
Z \cap \{c_m\}_{m \in \bbZ} = \varnothing,
$
or there exists a unique integer $m \in \bbZ$ such that $c_m \in Z$. In the latter case, we say that $Z$ is a {\it critical component} if $m \leq 0$, or a {\it valuable component} if $m \geq 1$.

We say that $z \in Z$ is an {\it extremal point of $Z$} if for any unicritical covering element $Q \in \cC$ containing $z$, the point $z$ is an extremal point for $(Z_Q, \leq_{Q})$, where $Z_Q$ is a connected component of $Z \cap Q$ containing $z$. Denote
$$
\mathring{Z} := Z \setminus \{\text{extremal points of } Z\}.
$$

\begin{lem}\label{component regular}
Let $Z$ be a connected component of $\Lambda$. Then $\cL^{+/-/P}_\delta(p)$ varies lower semi-continuously on $p \in Z$. Moreover, $\cL^{+/-}_\delta(p) < \infty$ for $p \in \mathring Z$. Lastly, if $\cL^P_\delta(p) = \infty$ for some $p \in \mathring Z$, then $p = c_m$ for some $m \in \bbZ$.
\end{lem}

\begin{proof}
For $p \in Z$, let $t_\pm(p) \in [0, \bft]$ with $\bft := L_1^{-1}$ be the largest value such that
$$
p \notin \bigcup_{i=1}^\infty \cT_{\pm i}(t_{\pm}(p)).
$$
Clearly, $t_\pm(p)$ depends continuously on $p$. The first claim now follows from \thmref{ref uni}.

Let $n \geq 0$ such that $c_{-n} \notin Z$ and
$
Z \cap \cT_{-n} \neq \varnothing.
$
We claim that $\cT_{c_{-n}}$ must contain an extremal point of $Z$. Otherwise, the two extremal points of $Z \cap \Lambda_{\cT_{c_{-n}}}$ coincide with the two extremal points of $\cT_{c_{-n}}$. However, this is not possible unless $c_{-n} \in Z$. Consequently, if there exists
$
p \in Z \cap \Delta^c_{c_0},
$
then $p$ must be an extremal point in $Z$. By a similar argument, we conclude that if there exists
$
p \in Z \cap \Delta^v_{c_1},
$
then $p$ must be an extremal point in $Z$. The second and third claims now follow from \thmref{ref uni}.
\end{proof}

\begin{thm}[Non-trivial components are center curves]\label{component arc}
Every non-trivial component $Z_0$ of $\Lambda$ is a $C^r$-curve. Moreover, $Z_0'(t) \in E^c_{Z_0(t)}$ for all $t \in (0, |Z_0|)$.
\end{thm}

The proof of \thmref{component arc} is based on the following three lemmas.
 
 \begin{lem}\label{component arc 1}
 There exists a uniform constant $\kappa >0$ such that the following holds. Let $Z_0$ be a non-valuable connected component of $\Lambda$, and let $p_0 \in Z_0$ be a Pesin $\delta$-regular point. If there exists
$
y_0 \in (Z_0 \cap W^{ss}(p_0)) \setminus \{p_0\},
$
then there exists a Pliss moment $-N < 0$ in the backward orbit of $p_0$ such that 
 $$
 \dist(F^{-N}(y_0), F^{-N}(p_0)) > \kappa.
 $$
 \end{lem}
 
\begin{proof}
Since $\cd(p_0), \vd(p_0) < \infty$, by replacing $p_0 \in Z_0$ with $p_{-i} \in Z_{-i}$ for some $i \geq 0$, we may assume that
$
\cd(p_0) = \vd(p_0)=\varnothing.
$

If
$
p_{-i} \in \Lambda^P_{L, \delta}
$ for $
0 \leq i \leq I,
$
then we have
 $$
 \dist(y_{-I}, p_{-I}) > \bL^{-1}\lambda^{-(1-\bdelta)I}\dist(y_0, p_0).
 $$

Let $-m < 0$ be a moment such that
$
\cd(p_{-m}) = 0,
$
and let $-j < 0$ be the tunnel escape moment for $p_{-m}$. By \propref{tunnel escape}, \propref{tunnel escape} and \thmref{ref uni}, $p_{-m+1}$ and $p_{-m-j}$ are forward $(1, \delta)$-regular. If $y_{-m+1} \in W^{ss}_{\loc}(p_{-m+1})$, then
 $$
 \dist(y_{-m+1}, p_{-m+1}) > k\lambda^{-(1-\bdelta)m}\dist(y_0, p_0)
 $$
for some uniform constant $k > 0$. Observe that in this case, $y_{-m+1}$ is contained in the same connected component of $W^{ss}_{\loc}(p_{-m+1}) \cap \cT_{c_1}$ as $p_{-m+1}$. Otherwise, $Z_{-m+1}$ must contain $c_1$, which contradicts the assumption that $Z_0$ is not valuable. It follows that
$
y_{-m - i} \in W^{ss}_{\loc}(p_{-m-i})
$ for $
0 \leq i \leq j,
$
and 
$$
\dist(y_{-m-j}, p_{-m-j}) > k\lambda^{-(1-\bdelta)(m+j)}\dist(y_0, p_0).
$$

Proceeding inductively, we obtain the desired result.
\end{proof}

\begin{lem}\label{component arc 2}
Let $Z_0$ be a connected component of $\Lambda$, and let $p_0 \in Z_0$ be a Pesin $\delta$-regular point. Then there does not exist a sequence
$$
\{y^i_0\}_{i=1}^\infty \subset (W^{ss}(p_0) \cap Z_0)\setminus \{p_0\}
$$
such that $\dist(y^i_0, p_0)\searrow 0$ as $i \to \infty$.
\end{lem}
 
\begin{proof}
Suppose towards a contradiction that such a sequence exists. By \lemref{component arc 1}, there exists a subsequence $\{y^{i_k}_0\}_{k=1}^\infty$ and a decreasing sequence $\{-N_k\}_{k=1}^\infty$ of Pliss moments in the backward orbit of $p_0$ such that
$$
 \dist(y^i_{-N_k}, p_{-N_k}) > \kappa.
$$

Consider the unicritical cover $\cC$ of $\Lambda$ defined in \subsecref{subsec:unicrit cover}. We may assume that the vertical height of each element $Q \in \cC$ are bounded by $\kappa/2$. This means that if $p \in \Lambda_Q$ and $q \in W^{ss}_{\loc}(p)$, and
$
\dist(p, q) > \kappa/2,
$
then $q$ cannot be contained in $\Lambda_Q$.

By compactness of $\Lambda$, we may choose a finite subcover $\cC_0 \subset \cC$. Observe that if a connected component $Z$ of $\Lambda$ intersects multiple elements from $\cC_0$, then $Z$ must contain least one of the two extremal points for each of these elements. Hence, we see that there are at most finitely many components that can intersect multiple covering elements. However, by the previous observations, the component $Z_{-N_k}$ for all $k \geq 1$ cannot be contained in a single covering element. This is a contradiction.
\end{proof}
 
\begin{lem}\label{component arc 3}
Let $Z_0$ be a non-valuable connected component of $\Lambda$. Then there exists a constant $\kappa_{Z_0}> 0$ such that if $p_0 \in Z_0$ is Pesin $\delta$-regular, and
$
y_0 \in (W^{ss}(p_0) \cap Z_0)\setminus \{p_0\},
$
then
$
\dist(y_0, p_0) > \kappa_{Z_0}.
$
\end{lem}

 \begin{proof}
 Consider the finite cover $\cC_0$ of $\Lambda$ given in the proof of \lemref{component arc 2}. Let $N = N(Z_0) \in \bbN$ be the smallest integer such that for all $n \geq N$, the component $Z_{-n}$ does not intersect multiple covering elements in $\cC_0$. Then it follows from \lemref{component arc 1} that if $p_0 \in Z_0$ is Pesin $\delta$-regular, then
$$
W^{ss}_{\loc}(p_{-n}) \cap Z_{-n} = \{p_{-n}\}
\matsp{for}
n \geq N.
$$

Suppose there exists
$$
y_0 \in (W^{ss}_{\loc}(p_0) \cap Z_0)\setminus \{p_0\}.
$$
Then by \lemref{component arc 1}, there exists $0 \leq m < N$ such that
$
\dist(y_{-m}, p_{-m}) > \kappa.
$
Consequently,
$$
\dist(y_0, p_0) > K\lambda^{(1+\eta) N}\kappa,
$$
for some uniform constant $K \geq 1$.
 \end{proof}
 
\begin{proof}[Proof of \thmref{component arc}]
It suffices to prove the result for a non-valuable connected component $Z$.

Let $p_0 \in Z_0$ be Pesin $(K, \delta)$-regular, and let
$
\{\Phi_{p_m} : \cU_{p_m} \to U_{p_m}\}_{m\in\bbZ}
$
be a linearization of the full orbit of $p_0$. By \propref{component regular} and \lemref{component arc 3}, if $W_0 \subset \cU^0_{p_0}$ is a sufficiently small neighborhood of $p_0$, and $\Gamma_0$ is the the connected component of $Z_0 \cap W_0$ containing $p_0$, then 
$
\gamma_0 := \Phi_{p_0}(\Gamma_0)
$
is the horizontal graph of a $C^0$-map $g_0 : \bbI(l) \to \bbR$ for some sufficiently small $l > 0$.

We claim that $g_0$ is differentiable at $0$, and that $g_0'(0) = 0$. Let $\{-n_i\}_{i=1}^\infty$ be the decreasing sequence of Pliss moments in the backward orbit of $p_0$. As in the proof of \lemref{component arc 3}, there exists $N = N(Z_0) \geq 0$ such that for all $q_0 \in \Gamma_0$, we have
$$
W^{ss}_{\loc}(q_{-n}) \cap Z_{-n} = \{q_{-n}\}
\matsp{for}
n \geq N.
$$
For $n_i > N$, let
$$
\{\Phi_{p_{-n_i+m}}^i : \cU_{p_{-n_i+m}}^i \to U_{p_{-n_i+m}}^i\}_{m\in\bbZ}
$$
be a linearization of the full orbit of $p_{-n_i}$, and let $\Gamma_{-n_i}^i$ be the connected component of $\Gamma_{-n_i} \cap \cU_{p_{-n_i}}^{\bdelta n_i}$ containing $p_{-n_i}$. Observe that
$$
\gamma^i_0 := \Phi_{p_0} \circ F^{n_i}(\Gamma_{-n_i}^i)
$$
is a neighborhood of $0$ in $\gamma_0$. Thus, by \thmref{reg chart} and \propref{loc linear}, we see that
$$
|g_0(x)| < \bK \lambda^{(1-\bdelta)n_i}
\matsp{for}
|x| < \bK^{-1}\lambda^{\bdelta n_i}.
$$
By \propref{conditional pliss full} and \propref{freq pliss back}, we conclude that
$$
\gamma_0 \subset \{(x,y)\in U_{p_0} \; | \; |x| < l \; \text{and} \; |y| < |x|^{1/\bdelta}\}.
$$
The result follows.
\end{proof}


\section{Recurrence of the Critical Orbit}

Let $F$ be the mildly dissipative, infinitely renormalizable and regularity unicritical diffeomorphism considered in \secref{sec:tot order}.

For any $C^r$-curve $\gamma \subset \Omega$, let $\phi_\gamma : (0, |\gamma|) \to \gamma$ be its parameterization by arclength. Recall that $\|\gamma\|_{C^r} := \|\phi_\gamma\|_{C^r}.$
Let
$
F_\gamma := \phi_{f(\gamma)}^{-1} \circ F|_\gamma \circ \phi_\gamma.
$
The {\it distortion of $F$ along $\gamma$} is defined as
$$
\Dis(F, \gamma) := \sup_{t, s \in (0, |\gamma|)}\frac{|F_\gamma'(t)|}{|F_\gamma'(s)|}.
$$

The following statement is a straightforward generalization of 1D Denjoy Lemma.

\begin{prop}\label{2d denjoy}
Let $\gamma_0 \subset \Omega$ be a $C^2$-curve. Suppose for all $0 \leq i < n$, the curve
$
\gamma_i := F^i(\gamma_0).
$
is well-defined and $\|\gamma_i\|_{C^2} = O(1)$. Then
$$
\log(\cD(F^n|_\gamma)) = O\left(\sum_{i=1}^n |F^i(\gamma)|\right).
$$
\end{prop}

\begin{thm}\label{crit recur}
The critical value $c_1$ is recurrent.
\end{thm}

\begin{proof}
The proof is by contradiction. Suppose there exists $R > 0$ such that
\begin{equation}\label{crit no recur}
\bbD_{c_1}(R) \cap \{c_m\}_{m\in \bbZ} = \{c_1\}.
\end{equation}

\begin{lem}
Suppose \eqref{crit no recur} holds. Then the connected component $Z_{c_1}$ of $\Lambda$ containing the critical value $c_1$ is a non-trivial arc.
\end{lem}

\begin{proof}
Suppose not. Then $\cD^n_1 \subset \bbD_{c_1}(R)$ for $n$ sufficiently large. However, $c_{1+lR_n} \in \cD^n_1$ for all $l \in \bbZ$.
\end{proof}

By \thmref{component arc}, $Z_{c_1}$ is a $C^1$ Jordan arc. Consider the parameterization of $Z_{c_1}$ by its arclength. Denote
$
\mathring Z_{c_1} := Z_{c_1}(0, |Z_{c_1}|).
$
Choose ${t} \leq R$ in \thmref{ref uni}, and let $L := L(t) \geq 1$ be the regularity factor. Note that we have
$
\Lambda \cap \cT_{c_m} \subset \mathring Z_{c_m}.
$
Consequently, the points
$
z_m^0 := Z_{c_m}(0)
$ and $
z_m^1 := Z_{c_m}(|Z_{c_m}|)
$
are Pesin $(L, \delta)$-regular. Moreover, if $m \neq l$, then 
$
\cT_{c_m} \cap \cT_{c_l} = \varnothing.
$

\begin{lem}\label{tunnel surround}
Suppose \eqref{crit no recur} holds. If $\cT_{c_{-n}} \subset Q:= \bcB_p$ for some $n \geq 0$ and $p \in \Lambda^P_{L, \delta}$, then either
\begin{itemize}
\item $z_{-n}^0 <_{Q} p <_{Q} z_{-n}^1$ for all $p \in \cT_{c_{-n}}$; or
\item $z_{-n}^1 <_{Q} p <_{Q} z_{-n}^0$ for all $p \in \cT_{c_{-n}}$.
\end{itemize}
\end{lem}

For $p \in \Lambda^P_{L, \delta}$, recall the definition of $\bcB_p(a)$ with $a >0$ (see \secref{sec:mild diss}). Rather than setting $\bcB_p:= \bcB_p(1/2)$, we modify the construction of an encased regular box in the following way. By \lemref{tunnel surround}, $p$ is accumulated by other Pesin $(L, \delta)$-regular points. If
$$
q <_{{\bcB_p(1/2)}} p
\matsp{or}
p <_{{\bcB_p(1/2)}} q
\matsp{for all}
q \in \Lambda^P_{L, \delta} \cap \bcB_p(1/2),
$$
then define $\bcB_p$ as the quadrilateral bounded by $\bpartial \bcB_p$, $W^{ss}_{\loc}(p)$ and $W^{ss}_{\loc}(q)$. Otherwise, choose $q_-, q_+ \in \Lambda^P_{L, \delta} \cap \bcB_p(1/2)$ such that
$$
q_- <_{{\bcB_p(1/2)}} p  <_{{\bcB_p(1/2)}} q_+,
$$
and define $\bcB_p$ as the quadrilateral bounded by $\bpartial \bcB_p$, $W^{ss}_{\loc}(q_-)$ and $W^{ss}_{\loc}(q_+)$. Then
$$
\bcB_p \subset \bcB_p(1/2) \subset \bcB_p(1).
$$

Choose a finite unicritical cover of $\Lambda$
$$
\cC_0 = \{\bcB_p\}_{p \in \cP} \cup \{\cT_{c_m}\}_{|m| \leq M}
$$
where $\cP \subset \Lambda^P_{L, \delta}$ is a finite set, and $0 < M < \infty$.

\begin{lem}
Suppose \eqref{crit no recur} holds. There exists $I \geq 1$ such that the following holds. For all $i \geq I$, there exists $p_i \in \cP$ and $N_i \geq 1$ such that for all $n \geq N_i$, we have
$
Z_{c_i} \subset \cD^n_i \subset \bcB_{p_i}.
$
Moreover, $\cD^n_i$ does not contain an extremal point of $\bcB_{p_i}$.
\end{lem}

\begin{proof}
Clearly, $Z_{c_i}$ cannot intersect $\cT_{c_m}$ for $|m| \leq M$ if $i > M$. If $Z_{c_i}$ intersects multiple covering elements in $\cC_0$, then it must contain at least one of the extremal points in each element. Since this can happen only finitely many times, we obtain the desired integer $I \geq 1$. Since $\cD^n_i$ converges in Hausdorff topology to $Z_{c_i}$, the result follows.
\end{proof}

For $n \geq N_I$, let $I_n \geq I$ be the largest integer such that for $I \leq i\leq I_n$, we have
$$
\Lambda^n_i := \cD^n_i \cap \Lambda \subset \bcB_{p_i}(1).
$$
Denote the extremal points in $\Lambda^n_i$ with respect to $<_{{\bcB_{p_i}(1)}}$ by $a^n_i$ and $b^n_i$. It follows from \lemref{tunnel surround} that $a^n_i$ and $b^n_i$ are Pesin $(L, \delta)$-regular. Let $\hcD^n_i$ be the quadrilateral bounded by $\bpartial \bcB_{p_i}$, $W^{ss}_{\loc}(a^n_i)$ and $W^{ss}_{\loc}(b^n_i)$.

\begin{lem}\label{curv under crit comp}
Suppose \eqref{crit no recur} holds. Let $\Gamma_I$ be a sufficiently horizontal $C^r$ Jordan arc in $\bcB_{p_I}$ with uniformly bounded curvature. For $n \geq N_I$ and $I \leq i \leq I_n$, define
$$
\Gamma_I^n := \Gamma_I \cap \hcD^n_I
\matsp{and}
\Gamma_i^n := F^{i-I}(\Gamma_I^n).
$$
Then $\Gamma_i^n$ is sufficiently horizontal in $\bcB_{p_I}(1)$, and has uniformly bounded curvature. Moreover,
$\displaystyle
\sum_{i = I}^{I_n} |\Gamma_i^n| = O(1).
$
\end{lem}

\begin{proof}
The first claim is an immediate consequence of \propref{pres ver hor curv}.

Observe that $\{W^{ss}_{\loc}(a^n_i), W^{ss}_{\loc}(b^n_i)\}_{i=I}^{I_n}$ partitions $\{\bcB_p(1)\}_{p \in \cP}$ by nearly linear vertical arcs. Since $\hcD^n_i$ occupies each cell in this partition at most once, and $\{\bcB_p(1)\}_{p \in \cP}$ clearly has uniformly bounded total horizontal length, the second claim follows.
\end{proof}

For $n \geq N_I$ and $I \leq i \leq I_n$, let $\hZ_i$ be the quadrilateral bounded by $\bpartial \bcB_{p_i}(1)$, $W^{ss}_{\loc}(z_i^0)$ and $W^{ss}_{\loc}(z_i^1)$. Write
$$
\Gamma_i^n := A_i^n \cup \hGamma_i \cup B_i^n,
$$
where
$
\hGamma_i := F^{i-I}(\Gamma_I \cap \hZ_I),
$
and $A_i^n$ and $B_i^n$ are the two components of $\Gamma_i^n \setminus \hGamma_i$.

Since $\hcD^n_I$ converges to $\hZ_I$ as $n \to \infty$, we see that
$
|A_I^n|, |B_I^n| \to 0
$ as
$
n \to \infty.
$
On the other hand, either $|A_{I_n}^n|$ or $|B_{I_n}^n|$ must be uniformly bounded below, while $|\hGamma_{I_n}| \to 0$ as $n \to \infty$. Applying \lemref{curv under crit comp} and \propref{2d denjoy}, we obtain a contradiction. Hence, \eqref{crit no recur} cannot hold, and \thmref{crit recur} follows.
\end{proof}


\section{Total Order on Renormalization Pieces}\label{sec:order on piece}

Let $F$ be the mildly dissipative, infinitely renormalizable and regularity unicritical diffeomorphism considered in \secref{sec:tot order}.

Recall that the renormalization limit set of $F$ is given by
$$
\Lambda := \bigcap_{n=0}^\infty \bigcup_{i = 0}^{R_n-1} F^i(D^n_1),
$$
where
$$
r_n := R_{n+1}/R_n \geq 2,
$$
and $D^n_1$ is an $R_n$-periodic Jordan domain containing the critical value $c_1$. Denote
$$
D^n_i := F^{i-1}(D^n_1)
\matsp{and}
\Lambda^n_i := \Lambda \cap D^n_i
\matsp{for}
1\leq i \leq R_n.
$$
We refer to $\Lambda^n_i$ as a {\it renormalization piece of depth $n$}. By \thmref{crit recur}, there exists $N_1 \geq 0$ such that
$$
\cD^n_1 \subset \cU_{c_1}
\matsp{and}
\Lambda^n_1 \subset \cT_{c_1}
\matsp{for}
n \geq N_1.
$$

\begin{prop}\label{ren piece at crit}
Let $n \geq N_1$. If $\Lambda^n_i \subset \Lambda^{N_1}_0$ and $\Lambda^n_i \neq \Lambda^n_0$, then either
\begin{itemize}
\item $x <_{{\cT_{c_0}}} c_0$ for all $x \in \Lambda^n_i$; or
\item $c_0 <_{{\cT_{c_0}}} x$ for all $x \in \Lambda^n_i$.
\end{itemize}
\end{prop}

\begin{proof}
For $m \geq n \geq N_1$, let $J(m, n) \subset \{0, \ldots, R_m-1\}$ be the set of all indexes such that
$
\cD^m_j \subset \cD^n_0
$ for $
j \in J(m, n).
$
By \thmref{pinching}, for $n \geq N_1$, there exists $M_n >n$ such that for any $m \geq M_n$ and $j \in J(m, N_1) \setminus J(m, n)$, we have
$
\cD^m_j \subset \cT_{c_0} \setminus \cD^n_0.
$
Since $\Lambda^n_i$ is a union of $\cD^m_j$'s, the result follows.
\end{proof}

Consider a finite unicritical cover of $\Lambda$:
$$
\hat\cC = \{\bcB_p\}_{p\in \cP} \cup \{\cT_m\}_{|m| < \hM},
$$
where $\cP \subset \Lambda^P_{L, \delta}$ is a finite set, and $0< \hM < \infty$.

\begin{prop}\label{cover long comp}
There exist at most a finite set $\cZ$ of non-trivial connected components of $\Lambda$ that cannot be contained in a single covering element in $\cC_0$. Moreover, let $Z \in \cZ$. Then there exist $Q_1, \ldots, Q_l \in \cC_0$ and
$$
0 =:t_0 < t_1 < \ldots <  t_l := |Z|
$$
such that
$$
Z([t_{i-1}, t_i]) \subset Q_i
\matsp{for}
1\leq i \leq l.
$$
In this case, $Q_i$'s are all regular boxes, except possibly exactly one of $Q_0$ and $Q_k$, which may be a critical tunnel or a valuable crescent. Moreover,
$$
\cd(Z(t)) = \vd(Z(t)) = \varnothing
\matsp{for}
t \in [t_1, t_{l-1}].
$$
\end{prop}

\begin{prop}\label{deep depth}
There exists $\hN \geq N_1$ such that for all $n \geq \hN$ and $0\leq i < R_n$, one of the following cases hold.
\begin{enumerate}[i)]
\item There exists $Q \in \cC_0$ such that $\Lambda^n_i \subset Q$.
\item The piece $\Lambda^n_i$ contains exactly one component $Z$ in $\cZ$. In this case, let $Q_1, \ldots, Q_l \in \cC_0$ be the covering elements given in \propref{cover long comp}. Then
$$
\Lambda^n_i \setminus Z \subset Q_1 \cup Q_l.
$$
\end{enumerate}
\end{prop}

\begin{proof}
Let $\bcB \in \cC_0$ be a regular box. If the two extremal components of $\bcB$ are not contained $\cZ$, then $\Lambda_{\bcB}$ is a compact subset of $\bcB$. Hence, there exists $K_1 \geq 0$ such that any renormalization domain $\cD^n_i$ of depth $n \geq K_1$ that intersects $\Lambda_{\bcB}$ must be contained in $\bcB$.

Let $\cT_{c_m} \in \cC_0$ be a critical tunnel or crescent. By \thmref{crit recur}, there exists $K_2 \geq 0$ such that
$
c_m \in \Lambda^{K_2}_m \subset \Lambda_{\cT_{c_m}}.
$
Denote
$$
\tim := m \md{R_{K_2}}.
$$
If the two extremal components of $\cT_{c_m}$ are not contained $\cZ$, then $\Lambda_{\cT_{c_m}} \setminus \Lambda^{K_2}_{\tim}$ is a compact subset of $\cT_{c_m} \setminus \cD^{K_2}_{\tim}$. Hence, there exists $K_3 \geq K_2$ such that any renormalization domain $\cD^n_i$ of depth $n \geq K_3$ that intersects $\Lambda_{\cT_{c_m}} \setminus \Lambda^{K_2}_{\tim}$ must be contained in $\cT_{c_m}$.

The result follows by the finiteness of $\cC_0$.
\end{proof}

\begin{thm}\label{order on piece}
For $n \geq \hN$ and $1\leq i \leq R_n$, there exists a total order $<_{\Lambda^n_i}$ on $\Lambda^n_i$ such that following holds. For $Q \in \cC_0$, either
\begin{itemize}
\item $x <_{\Lambda^n_i} y$ for all $x,y \in \Lambda^n_i \cap Q$ with $x <_{Q} y$; or
\item $y <_{\Lambda^n_i} x$ for all $x,y \in \Lambda^n_i \cap Q$ with $x <_{Q} y$.
\end{itemize}
Moreover, if $1 \leq i < R_n$, then $F(x) <_{\Lambda^n_{i+1}} F(y)$ for all $x,y \in \Lambda^n_i$ with $x <_{\Lambda^n_i} y$.
\end{thm}

\begin{proof}
Suppose $\Lambda^n_i$ contains a component $Z \in \cZ$ (parameterized by its arclength $|Z|$). Let $Q_1, \ldots, Q_l \in \cC_0$ be the covering elements given in \propref{cover long comp}. Define a total order on $\Lambda^n_i$ by the following property. For $j \in \{1, l\}$, let $p_j, q_j \in Q_j \setminus Z$. Then
\begin{enumerate}[i)]
\item $p_j <_{\Lambda^n_i} q_j$ if $p_j <_{{Q_j}} q_j$;
\item $p_0, q_0 <_{\Lambda^n_i} Z(0)$;
\item $Z(t) <_{\Lambda^n_i} Z(s)$ if $0\leq t < s \leq |Z|$; and
\item $Z(|Z|) <_{\Lambda^n_i} p_l, q_l$.
\end{enumerate}
The result now follows by \propref{order trans}, \ref{order pres} and \ref{order value}.
\end{proof}

Let $\hN \leq m < n$ and $i \geq 0$. Observe that
$$
\Lambda^m_i = \bigsqcup_{k =0}^{R_n/R_m-1} \Lambda^n_{i+kR_m}.
$$
For $a, b \in \Lambda^m_i$ with $a \leq_{\Lambda^m_i} b$, denote
$$
[a, b]_{\Lambda^m_i} := \{p \in \Lambda^m_i \; | \; a\leq_{\Lambda^m_i}p \leq_{\Lambda^m_i} b\}.
$$
This notation generalizes to $(\cdot, \cdot)_{\Lambda^m_i}$ in the obvious way.

A {\it combinatorial component $\cK$ of $\Lambda^n_{i+kR_m}$ in $\Lambda^m_i$} is a maximal subset of $\Lambda^n_{i+kR_m}$ of the form
$
\cK = [a, b]_{\Lambda^m_i}
$
for some $a, b \in \Lambda^m_i$. We say that $\Lambda^n_{i+kR_m}$ is {\it combinatorially connected in $\Lambda^m_i$} if $\Lambda^n_{i+kR_m}$ has a unique combinatorial component $\cK$ in $\Lambda^m_i$ such that
$
\Lambda^n_{i+kR_m} = \cK.
$

\begin{thm}\label{combin connect}
Let $\hN \leq m < n$; $i \geq 0$ and $0 \leq k < R_n/R_m$. Then $\Lambda^n_{i+kR_m}$ is combinatorially connected in $\Lambda^m_i$.
\end{thm}

For the proof of \thmref{combin connect}, we need the following lemma.

\begin{lem}\label{finite comb comp}
The renormalization piece $\Lambda^n_{i+kR_m}$ has at most finitely many combinatorial components in $\Lambda^m_i$ which permute cyclically under $F^{R_n}$.
\end{lem}

\begin{proof}
Suppose $\Lambda^n_{i+kR_m}$ has infinitely many combinatorial components. Then one can find two sequences
$$
\{a_l\}_{l=1}^\infty \subset \Lambda^n_{i+kR_m}
\matsp{and}
\{b_l\}_{l=1}^\infty \subset \Lambda^n_{i+jR_m} \neq \Lambda^n_{i+kR_m}
$$
such that
$$
a_i <_{\Lambda^m_i} b_i <_{\Lambda^m_i} a_{i+1}
\matsp{for}
i \geq 1.
$$
By compactness, the unique limit $a_\infty = b_\infty$ of these sequences must belong to both $\Lambda^n_{i+kR_m}$ and $ \Lambda^n_{i+jR_m}$, which is impossible.

Since $\{c_{-l}\}_{l=0}^\infty$ accumulates to every connected component of $\Lambda$, we see that every combinatorial component of $\Lambda^n_{i+kR_m}$ must contain $c_{-l}$ for some $l \geq 0$. Consequently, every combinatorial component of $\Lambda^n_{i+kR_m}$ must eventually map to the unique one containing $c_{-R_n + i+kR_m}$. The result follows.
\end{proof}

\begin{proof}[Proof of \thmref{combin connect}]
It suffices to prove the result for $\Lambda^{n+1}_1$ in $\Lambda^n_1$. Suppose there exist combinatorial components
$
\cB_1 \subset \Lambda^{n+1}_1
$ and $
\cX_1 \subset \Lambda^{n+1}_j \neq \Lambda^{n+1}_1;
$
and points
$
a^0_1\in \Lambda^{n+1}_1
$ and $
y^0_1 \in \Lambda^{n+1}_j
$
such that
$$
a^0_1 <_{\Lambda^n_1} \cX_1 <_{\Lambda^n_1} \cB_1 <_{\Lambda^n_1} y^0_1.
$$

Proceeding by induction, suppose for $i \geq 0$, there exist points
$
a^i_1 \in \Lambda^{n+1}_1
$ and $
y^i_1 \in \Lambda^{n+1}_j
$
such that
$$
a^i_1 <_{\Lambda^n_1} \cX_{1+iR_n} <_{\Lambda^n_1} \cB_{1+iR_n} <_{\Lambda^n_1} y^i_1.
$$
By \thmref{order on piece}, \propref{order value} and \propref{ren piece at crit}, we see that after applying $F^{R_n}$, the following holds.
\begin{enumerate}[i)]
\item If $i = -1 \md{r_n}$, then either
$$
c_1 =: a^{i+1}_1 <_{\Lambda^n_1} \cX_{1+(i+1)R_n} <_{\Lambda^n_1} \cB_{1+(i+1)R_n} <_{\Lambda^n_1} y^{i+1}_1 := y^i_{1+R_n}
$$
or
$$
a^i_{1+R_n} =:a^{i+1}_1 <_{\Lambda^n_1} \cX_{1+(i+1)R_n}  <_{\Lambda^n_1} \cB_{1+(i+1)R_n} <_{\Lambda^n_1} y^{i+1}_1 := y^i_{1+R_n} <_{\Lambda^n_1} c_1.
$$
\item If $i = -k-1 \md{r_n}$, then either
$$
c_1 <_{\cT_{c_0}} a^i_{1+R_n} =: a^{i+1}_1 <_{\Lambda^n_1} \cX_{1+R_{n+1}} <_{\Lambda^n_1} \cB_{1+R_{n+1}} <_{\Lambda^n_1} y^{i+1}_1 := y^i_{1+R_n}
$$
or
$$
a^i_{1+R_n} =:a^{i+1}_1 <_{\Lambda^n_1} \cX_{1+R_{n+1}}  <_{\Lambda^n_1} \cB_{1+R_{n+1}} <_{\Lambda^n_1} y^{i+1}_1 := c_1.
$$
\item Otherwise, either
$$
c_1 <_{\cT_{c_0}} a^i_{1+R_n} =: a^{i+1}_1 <_{\Lambda^n_1} \cX_{1+R_{n+1}} <_{\Lambda^n_1} \cB_{1+R_{n+1}} <_{\Lambda^n_1} y^{i+1}_1 := y^i_{1+R_n}
$$
or
$$
a^i_{1+R_n} =:a^{i+1}_1 <_{\Lambda^n_1} \cX_{1+(i+1)R_n}  <_{\Lambda^n_1} \cB_{1+(i+1)R_n} <_{\Lambda^n_1} y^{i+1}_1 := y^i_{1+R_n} <_{\Lambda^n_1} c_1.
$$
\end{enumerate}
Thus, we see that $\cB_{1+iR_n}$ does not contain $c_1$ for all $i \geq 0$. This contradicts \lemref{finite comb comp}.
\end{proof}

\begin{prop}\label{end points}
Let $n \geq \hN$. Then the two extremal points of $\Lambda^n_1$ with respect to $<_{\Lambda^n_1}$ are $c_1$ and $c_{1+R_n}$. That is, for all $p \in \Lambda^n_1 \setminus \{c_1, c_{1+R_n}\}$, we have
$$
c_1 <_{\Lambda^n_1} p <_{\Lambda^n_1} c_{1+R_n}.
$$
\end{prop}

\begin{proof}
Suppose
$
c_{1+R_n} <_{\Lambda^n_1} c_{1-R_n}.
$
Then
$$
F^{R_n}([c_1, c_{1-R_n}]_{\Lambda^n_1}) = [c_1, c_{1+R_n}]_{\Lambda^n_1} \subset [c_1, c_{1-R_n}]_{\Lambda^n_1}.
$$
Consequently,
$$
\{c_{1+iR_n}\}_{i=1}^\infty \subset [c_{1+2R_n}, c_{1+R_n}]_{\Lambda^n_1}.
$$
Since
$$
c_1 \notin[c_{1+2R_n}, c_{1+R_n}]_{\Lambda^n_1},
$$
this contradicts \thmref{crit recur}.

Let $z_1 \neq c_1$ be an extremal point of $\Lambda^n_1$. Suppose
$
c_{1+R_n} <_{\Lambda^n_1} z_1.
$
Then
$$
F^{R_n}([c_1, c_{1-R_n}]_{\Lambda^n_1}) = [c_1, c_{1+R_n}]_{\Lambda^n_1} \neq [c_1, z_1]_{\Lambda^n_1}.
$$
This implies that
$$
F^{R_n}([c_{1-R_n}, z_1]_{\Lambda^n_1}) = [c_1, z_1]_{\Lambda^n_1}.
$$
Since $z_1$ cannot be a fixed point, this is a contradiction.
\end{proof}


\section{Regular H\'enon-like Returns}\label{sec:henon return}

Let $F$ be the mildly dissipative, infinitely renormalizable and regularity unicritical diffeomorphism considered in \secref{sec:tot order}.

Let $\cC$ be the unicritical cover of $\Lambda$ defined in \subsecref{subsec:unicrit cover}. By compactness of $\Lambda$, we may choose a finite subcover
$$
\cC_0 = \{\bcB_p\}_{p \in \cP} \cup \{\cT_{c_{-m}}\}_{m = 0}^{\hM_-} \cup \{\cT_{c_m}\}_{m = 1}^{\hM_+} \subset \cC,
$$
where $\cP \subset \Lambda^P_{L, \delta}$ is a finite set of Pesin $(L, \delta)$-regular points, and $\hM_\pm \in \bbN$. By increasing $\hM_\pm$ if necessary, we may assume that $\hM_\pm$ is a Pliss moment in the forward/backward orbit of $c_1$/$c_0$, and that we have $\cT_{c_{\hM_\pm}} \subset \bcB_{p^\pm}$ for some Pesin regular point $p^\pm \in \cP$. For $n \in \bbN$, recall that the $n$th renormalization domain containing $c_1$ is denoted $\cD^n_1$. The {\it valuable} and the {\it critical pieces of depth $n$} are defined as
$$
\Lambda^n_1 = \Lambda \cap \cD^n_1 \ni c_1
\matsp{and}
\Lambda^n_0 := \Lambda^n_{R_n} = \Lambda \cap \cD^n_{R_n} \ni c_0
$$
respectively.

Let
$$
\{\Phi_{c_m} : (\cU_{c_m}, c_m) \to (U_{c_m}, 0)\}_{m\in\bbZ}
$$
be a uniformization of the orbit of $c_0$ given in \thmref{unif reg crit}. For $p_0 \in \cU_{c_0}$, let $E^v_{p_0} \in \bbP^2_{p_0}$ and $E^h_{p_1} \in \bbP^2_{p_1}$ be the directions given by
$$
D\Phi_{c_0}(E^v_{p_0}) = E^{gv}_{\Phi_{c_0}(p_0)}
\matsp{and}
D\Phi_{c_1}(E^h_{p_1}) = E^{gh}_{\Phi_{c_1}(p_1)} = DF(E^v_{p_0}).
$$

Let $\hN \in \bbN$ be the integer given in \propref{deep depth}. In the rest of this section, we assume that $n \geq \hN$. The Main Theorem stated in \secref{sec:intro} is an immediate consequence of the following result.

\begin{thm}[Valuable chart]\label{val chart}
There exist $\tau_n > 0$ and a $C^r$-diffeomorphism $\Psi^n : (\Phi_{c_1}(\cV^n_1), 0) \to (V^n_1, 0)$, where
$$
I^n := (-\tau_n, \tau_n)
\matsp{and}
V^n_1 := I^n \times (-l_{c_1}, l_{c_1}),
$$
such that the following properties hold. 
\begin{enumerate}[i)]
\item We have
$
F^{R_n}(\cV^n_1) \Subset \cV^n_1 \subset \cU_{c_1},
$
and consequently
$
F^{R_n-1}(\cV^n_1) \subset \cU_{c_0}
$ and $
\Lambda^n_1 \subset \cV^n_1.
$
\item We have
$$
\|\Psi^n - \Id\|_{C^r} = O(\lambda^{(1-\bdelta)R_n})
\matsp{and}
\Psi^n|_{I^n \times \{0\}} \equiv \Id.
$$
\item For $p_1 \in \cV^n_1$, let
$$
E^{v,n}_{p_1}:= DF^{-R_n+1}(E^v_{p_{R_n}}) \in \bbP^2_{p_1}.
$$
Then
$$
D\bPsi^n(E^h_{p_1}) = E^{gh}_{\bPsi^n(p_1)}
\matsp{and}
D\bPsi^n(E^{v,n}_{p_1}) = E^{gv}_{\bPsi^n(p_1)},
$$
where $\bPsi^n :=\Psi^n \circ \Phi_{c_1}$.
\item Let $p_1 \in \cV^n_1$. Then $p_1$ is $(R_n-1)$-times forward $(1, \delta)_v$-regular along $E^{v,n}_{p_1}$, and $p_{R_n}$ is $(R_n-1)$-times backward $(1, \delta)_v$-regular along $E^v_{p_{R_n}}$.
\item Define
$$
H_n := \Phi_{c_0} \circ F^{R_n-1}\circ (\bPsi^n)^{-1}.
$$
Then we have
$$
H_n(x,y) = (h_n(x), e_n(x,y))
\matsp{for}
(x,y) \in V^n_1,
$$
where $h_n : I^n \to h_n(I^n)$ is a $C^r$-diffeomorphism, and $e_n : V^n_1 \to \bbR$ is a $C^r$-map such that
$$
e_n(\cdot, 0) \equiv 0
\matsp{and}
\|e_n\|_{C^r} = O(\lambda^{(1-\bdelta)R_n}).
$$
\end{enumerate}
\end{thm}

\begin{lem}\label{trap off piece}
For $p \in \Lambda^n_1$, write
$
\Phi_{c_1}(p) = (x_p, y_p).
$
Denote
$$
\chi_{n+1} := \inf_{p \in \Lambda^n_1 \setminus \Lambda^{n+1}_1}|x_p|.
$$
Then we have
$$
\sup_{p \in \Lambda^{n+1}_1}|x_p| < \chi_{n+1}+O(\chi_{n+1}^{1/\bepsilon}).
$$
\end{lem}

\begin{proof}
Let $q_1 \in \Lambda^n_1 \setminus \Lambda^{n+1}_1$ be a point satisfying $|x_{q_1}| = \chi_{n+1}$. If $\cd(q_1) = \varnothing$, then by \thmref{ref uni}, $q_1$ is forward $(1, \delta)$-regular. Let $N$ be an integer satisfying
$$
N \asymp \bepsilon^{-1}\logl |x_{q_1}|.
$$
Then by \lemref{reg nbh size}, $q_i \in \cU_{c_i}$ for $1\leq i \leq N$. Note that $\lambda^{(1-\bdelta)N} < |x_{q_1}|^{1/\bepsilon}$. By \propref{vert angle shrink}, we see that $W^{ss}_{\loc}(q_1)$ is a $O(|x_{q_1}|^{1/\bepsilon})$-vertical curve in $\cU_{c_1}$. By \thmref{combin connect}, $\Lambda^{n+1}_1$ is contained in the vertical strip bounded between $W^{ss}_{\loc}(q_0)$ and $W^{ss}_{\loc}(c_1)$. Thus, the result holds in this case.

If $\cd(q_1) \neq \varnothing$, let $q_1 \in \cT_{c_{-m}}$ for some $m > 0$ with $\cd(c_{-m})= \varnothing$. By \lemref{return dist lem}, we see that
$$
\dist(q_1, c_{-m}) < |x_{q_1}|^{1/\bepsilon}.
$$
Moreover, arguing as above, we see that the vertical manifold $\hW^v(c_{-m})$ at $c_{-m}$ is a $O(|x_{q_1}|^{1/\bepsilon})$-vertical curve in $\cU_{c_1}$. The set $(c_1, c_{-m})_{\cT_{c_1}}$ is contained in the vertical strip bounded between $\hW^v(c_{-m})$ and $W^{ss}_{\loc}(c_1)$. Moreover,
$
\Lambda^{n+1}_1 \subset (c_1, c_{-m})_{\cT_{c_1}} \cup \cT_{c_{-m}}.
$
The result follows.
\end{proof}

\begin{prop}\label{tunnel escape at entry}
We have
$
\cT_{c_{-R_n+1}} \cap \Lambda^{n+1}_1 = \varnothing
$ and $
\cT_{c_{R_n}} \cap \Lambda^{n+1}_0 = \varnothing.
$
\end{prop}

\begin{proof}
Suppose there exists $p_1 \in \cT_{c_{-R_n+1}} \cap \Lambda^{n+1}_1$. Let
$$
q_0 := p_{R_n} \in F^{R_n-1}(\cT_{c_{-R_n+1}}) \cap \Lambda^{n+1}_{R_n} \subset \cT_{c_0}.
$$
Observe that
$$
\dist(q_0, c_0) < \lambda^{-\eta R_n} \diam(\cT_{c_{-R_n+1}}) < K\lambda^{(1-\bdelta)\delta R_n}
$$
for some uniform constant $K \geq 1$. Since
$
\dist(q_1, c_1) \asymp \dist(q_0, c_0),
$
we conclude by \lemref{trap off piece} that
$
\chi_{n+1} < K\lambda^{(1-\bdelta)\delta R_n}.
$

Write
$
(x, y) := \Phi_{c_1}(c_{-R_n+1}).
$
We have
$$
\dist(c_{-R_n+1}, c_1) \asymp |x| < \diam(\cT_{c_{-R_n+1}}) + \chi_{n+1} < K\lambda^{(1-\bdelta)\delta R_n} \ll \diam(\cU_{c_{-R_n+1}}).
$$
This contradicts \propref{no crit in crit nbh}.

The second claim can be proved using the same proof {\it mutatis mutandis}.
\end{proof}

Let $p \in \Omega$. For $m \geq 0$, define
$$
\cd^m(p) := \left\{
\begin{array}{cl}
\cd(p) & : \text{ if } 0\leq \cd(p) < m\\
\varnothing & : \text{ otherwise}
\end{array}
\right..
$$
Similarly, for $k \geq 1$, define
$$
\vd^k(p) := \left\{
\begin{array}{cl}
\vd(p) & : \text{ if } 1\leq \vd(p) < k\\
\varnothing & : \text{ otherwise}
\end{array}
\right..
$$

\begin{lem}\label{entry for reg lem}
For $p_0 \in \Lambda^n_0$, let
$
z_0 := \Phi_{c_0}(p_0)
$ and $
E^v_{p_0} := D_{z_0}\Phi_{c_0}^{-1}(E^{gv}_{z_0}).
$
Then for $0 \leq i < R_n$ such that $\cd^i(p_{-i}) = \varnothing$, the point $p_{-i}$ is $i$-times forward $(O(L), \delta)_v$-regular along
$
E^{v,n}_{p_{-i}} = DF^{-i}(E^v_{p_0}).
$
If, additionally, we have $p_{-i} \in \cT_{c_1}$, then $p_{-i}$ is $i$-times forward $(1, \delta)_v$-regular along $E^{v,n}_{p_{-i}}$.
\end{lem}

\begin{proof}
First, suppose that
$
\cd(p_0) = 0.
$
Then for all $0 \leq i < R_n$, we have
$
\cd^i(p_{-i}) = \cd(p_{-i}).
$

Let $0< m<R_n$ be an integer such that
$
p_{-m} \in \cT_{c_0}
$ and $
\cd^m(p_{-m}) = 0.
$
Let $-j <0$ and $l > 1$ be the tunnel and crescent escape moments for $p_{-m}$ and $p_{-m+1}$ respectively. Then by \propref{tunnel inv}, we have
$$
\cd^{m-l}(p_{-m+l}) = \cd^{m+j}(p_{-m-j}) = \varnothing.
$$

Suppose the statement holds for $p_{-m+l}$, so that $p_{-m+l}$ is $(m-l)$-times forward $(O(L), \delta)_v$-regular along $E^{v,n}_{p_{-m+l}}$. Then by \lemref{1 tunnel recover}, \propref{stable center align} and \propref{tunnel escape}, we see that $p_{-m-j}$ is $(m+j)$-times forward $(O(1), \delta)_v$-regular along $E^{v,n}_{p_{-m-j}}$. Moreover, by \thmref{ref uni} and \propref{tunnel escape}, $p_{-m-j}$ is forward $(O(1), \delta)_v$-regular along $\hE^v_{p_{-m-j}}$. Thus, by \propref{vert angle shrink}, we see that
$$
\measuredangle(E^{v,n}_{p_{-m-j}}, \hE^v_{p_{-m-j}}) < O(\lambda^{(1-\bdelta) (m+j)}) \ll 1.
$$

Let $m_1 < R_n$ be the smallest integer such that $m_1 > m+j$ and $p_{-m_1} \in \cT_{c_0}$. This means that for $m+j < i < m_1$, we have $\cd^i(p_{-i}) = \varnothing$. Consequently, by \thmref{ref uni}, $p_{-i}$ is $i$-times forward $(L, \delta)_v$-regular along $\hE^v_{p_{-i}}$. Thus, by \propref{loc linear}, $p_{-i}$ is $i$-times forward $(O(L), \delta)_v$-regular along $E^{v,n}_{p_{-i}}$. Proceeding by induction, the claim follows.

If
$
\cd(p_0) \neq 0,
$
then let
$
0 < n_1 < \ldots < n_K \leq \infty
$
be the increasing sequence of all integers such that $p_0 \in \cT_{c_{-n_k}}$ for $k \geq 1$. Furthermore, let $-(l_k+n_k) < 0$ be the tunnel escape moment of $c_{-n_{k+1}+n_k}$. For $0 \leq i < R_n$, let $q_{-i} := c_{-n_{k_i}-i}$, where $k_i$ is the smallest integer such that $l_{k_i} > i$. This means that
$$
p_{-i} \in \cT_{c_{-n_{k_i}-i}}
\matsp{and}
\cd^i(c_{-n_{k_i}-i}) = \cd(c_{-n_{k_i}-i}).
$$
If $K < \infty$, let $q_{-i} := p_{-i}$ for $i \geq l_K$.

Observe that for $1 \leq k \leq K$, we have
$$
\cd(c_{-n_k-l_k}) = \cd(c_{-n_{k+1}-l_k}) =\varnothing.
$$
Hence, $c_{-n_k-l_k}$ is forward $(O(1), \bepsilon)_v$-regular along $\hE^v_{c_{-n_k-l_k}}$, and by \propref{tunnel escape}, $c_{-n_{k+1}-l_k}$ is forward $(O(1), \delta)_v$-regular along $\hE^v_{c_{-n_{k+1}-l_k}}$. Moreover, if $c_{-n_k-l_{k-1}}$ is $l_{k-1}$-times forward $(O(1), \delta)_v$-regular along $E^{v,n}_{c_{-n_k-l_{k-1}}}$, then by \propref{vert angle shrink} and \ref{loc linear}, it follows that $c_{-n_{k+1}-l_k}$ is $l_k$-times forward $(O(1), \delta)_v$-regular along $E^{v,n}_{c_{-n_{k+1}-l_k}}$.

The result now follows by replacing $p_{-i}$ in the argument for the previous case (where $\cd(p_0) = \varnothing$) with $q_{-i}$.
\end{proof}

\begin{lem}\label{entry back reg lem}
For $p_1 \in \Lambda^n_1$, let
$
z_1 := \Phi_{c_1}(p_1)
$ and $
E^h_{p_1} := D\Phi_{c_1}^{-1}(E^{gh}_{z_1}).
$
For $1 \leq i \leq R_n$ such that $\vd^i(p_i) = \varnothing$, the point $p_i$ is $i$-times backward $(O(L), \delta)_h$-regular along
$$
E^{h,n}_{p_i} := DF^{i-1}(E^h_{p_1}).
$$
If, additionally, we have $p_i \in \cT_{c_0}$, then $p_i$ is $i$-times backward $(1, \delta)_h$-regular along $E^{h,n}_{p_i}$.
\end{lem}

Let $p \in \Lambda^P_{L, \delta}$. Recall that an encased regular box $\bcB_p$ at $p$ is nearly a parallelogram, whose horizontal width and vertical height are commensurate with $1/L_1$ and $1/L_2$ respectively, where
$
\bL < L_1 < \underline{L_2}.
$
Let $\cB_p$ be a (non-encased) regular box at at $p$ whose horizontal width and vertical height are commensurate with $1/L_1'$ and $1/L_2'$ respectively, where
$
\bL < L_1' < \underline{L_1}
$ and $
\overline{L_1'} < L_2' < \underline{L_2}.
$
If $-j < -\hM_-$ is a Pliss moment in the backward orbit of $c_0$, let $\cB_{c_{- j}}$ be a regular box at $c_{-j}$ whose horizontal width and vertical height are uniformly bounded below.

Let $Q \in \cC_0$. We define a larger covering element $\hQ \supset Q$ as follows. If $Q = \bcB_p$ for some $p \in \cP$, let $\hQ := \cB_p$. Otherwise, we have $Q = \cT_{c_m}$ for some $-\hM_- \leq m \leq \hM_+$. In this case, let $\hQ := \hcT_{c_m}$, where $\hcT_{c_m}$ is a thickened tunnel/crescent at $c_m$ defined in \eqref{eq:thick tunnel}/\eqref{eq:thick crescent}.

For $p_0 \in \Lambda^n_0$, let
$
(x_0, y_0) := \Phi_{c_0}(p_0)
$ and
$$
W^v(p_0) := \Phi_{c_0}^{-1}(\bbL^v_{x_0} \cap U_{c_0}).
$$
For $0 \leq i < R_n$, define
$$
W^{v,n}(p_{-i}) := F^{-i}(W^v(p_0)).
$$

\begin{lem}\label{trans crit curv}
Let $p_0 \in \cT_{c_0}$. Let $-j <0$ and $i > 1$ be the tunnel and crescent escape time for $p_0$ and $p_1$ respectively. Let
$$
\{\hPhi_{c_{-j+m}} : \hcU_{c_{-j+m}} \to \hU_{c_{-j+m}}\}_{m = -\infty}^j
$$
be a linearization of the $(-\infty, j)$-orbit of $c_{-j}$ with vertical direction $\hE^v_{c_{-j}}$. Similarly, let
$$
\{\chPhi_{c_{i+m}} : \chcU_{c_{i+m}} \to \chU_{c_{i+m}}\}_{m = -i+1}^\infty
$$
be a linearization of the $(-i+1, \infty)$-orbit of $c_i$ with vertical direction $\hE^v_{c_i}$. Denote the connected components of $\hcU_{c_{-j}}^{\bepsilon j} \setminus \hcT_{c_{-j}}$ that are disjoint from $\hW^v(c_{-j})$ by $\hcU^-$ and $\hcU^+$. Then the following statements hold.
\begin{enumerate}[i)]
\item Let $\Gamma_0^v$ be $C^1$-curve so that $\Gamma_i^v$ is sufficiently vertical in $\chcU_{c_i}$ and the distances between $p_i$ and the endpoints of $\Gamma_i^v$ are at least $1/\overline{L_2}$. Then $\Gamma_{-j}^v$ is sufficiently vertical and proper in $\cB_{c_{-j}}$.
\item We have
$$
F^j(\hcU^- \cup \hcT_{c_{-j}}\cup\hcU^+) \subset \hcT_{c_0}.
$$
\item Let $\Gamma^h_{-j} \ni p_{-j}$ be a sufficiently horizontal curve in $\hcU_{c_{-j}}$. Then
$$
\Gamma_i^\pm := F^{j+i}(\Gamma_{-j}^h \cap \hcU^\pm)
$$
is a sufficiently horizontal curve in $\chcU_{c_i}$.
\end{enumerate}
\end{lem}

\begin{proof}
Write
$
(x_1, y_1) := \Phi_{c_1}(p_1).
$
Note that
$$
i > (1+\bdelta)\delta^{-1}\logl |x_1|.
$$
By \thmref{reg chart} and \propref{loc linear}, we see that $\Gamma_1^v$ is $|x_1|^{(1+\bdelta)/\delta}$-vertical in $\cU_{c_1}$. Furthermore,
$$
\diam(\Gamma_1^v) > |x_1|^{-(1+\bdelta)/\delta}/\overline{L_2} > 1
\matsp{and}
\dist(\Gamma_1^v) \asymp |x_1|.
$$
This implies that $\Gamma_1^v$ is sufficiently stable-aligned in $\cU_{c_1}$, and vertically proper in $\hcT_{c_1}$.

Let $q_0 \in \Gamma_0^v \cap \bpartial \hcT_{c_0}$, and write
$
(x_0, y_0) := \Phi_{c_0}(p_0)
$ and $
(u_0, v_0) := \Phi_{c_0}(q_0).
$
Note that
$$
j > (1+\bdelta)\delta^{-1}\logl |x_0|.
$$
By \thmref{reg chart} and \propref{loc linear}, we have
$$
|v_{-j}| > \lambda^{-(1-\bepsilon)j}|v_0| > \lambda^{-(1-\bepsilon)j}|x_0|^{1/\hdelta} > 1.
$$
Additionally, \propref{tunnel escape} implies that $\Gamma_{-j}$ is sufficiently vertical in $\hcU_{c_{-j}}$. The first claim follows.

The second claim follows from \thmref{reg chart} and \propref{loc linear}.

The third claim follows from \propref{stable center align} and \propref{crescent escape}.
\end{proof}

\begin{prop}\label{proper ver curv}
Let $p_0 \in \Lambda^n_0$. For $\hM_-\leq i < R_n$, if
$
p_{-i} \in Q \in \cC_0
$ and $
\cd^i(p_{-i}) = \varnothing,
$
then $W^{v,n}(p_{-i})$ is sufficiently vertical and vertically proper in $\hQ$. If, additionally, $Q = \cT_{c_1}$, then $W^{v,n}(p_{-i})$ is sufficiently stable-aligned in $\cU_{c_1}$.
\end{prop}

\begin{proof}
Denote
$
\Gamma_{-i}^v := W^{v,n}(p_{-i}).
$
Let
$
0 =: m_0 < m_1 < \ldots < m_K < R_n
$
be an increasing sequence of integers such that
$
p_{-m_k} \in \cT_{c_0}
$ and $
\cd^{m_k}(p_{-m_k}) = 0.
$
Let $(m_k-1-l_k) > 1$ and $-(j_k-m_k) < 0$ be the crescent and tunnel escape time for $p_{-m_k+1}$ and $p_{-m)l}$ respectively. By \propref{tunnel inv}, we see that
$$
\cd^l(p_{-l_k}) = \cd^j(p_{-j_k}) = \varnothing
\matsp{and}
j_{k-1} < l_k < m_k < j_k < l_{k-1}.
$$

If
$$
\diam(\Gamma_{-l_k}^v) > 1/\overline{L_2},
$$
then by \lemref{trans crit curv} i) and \propref{consist ver hor} i), $\Gamma_{-j_k}^v$ is sufficiently vertical and proper in any regular box $\hQ$ for which we have $p_{-j_k} \in Q$.

The result now follows from \propref{pres ver hor curv} ii).
\end{proof}

Let $p_1^-, p_1^+ \in \Lambda^n_1$. By \propref{proper ver curv}, $W^{v,n}(p_1^\pm)$ is vertical proper in $\cT_{c_1}$ and sufficiently stable-aligned in $\cU_{c_1}$. Let $\cV^n_1(p_1^-, p_1^+)$ be the connected component of $\cU_{c_1} \setminus (W^{v,n}(p_1^-) \cup W^{v,n}(p_1^+))$ that contain $p_1^-$ and $p_1^+$ in its boundary. For $1\leq i\leq R_n$, denote
$$
\cV^n_i(p_i^-, p_i^+) := F^{i-1}(\cV^n_1(p_1^-, p_1^+)).
$$

\begin{prop}\label{v fit}
There exist a uniform constant $\hK \in \bbN$, and a set
$
\{q^0_1, \ldots, q^{\hK}_1\} \subset \Lambda^n_1
$
such that
$$
c_1 =: q^0_1 <_{\cT_{c_1}} \ldots <_{\cT_{c_1}} q^{\hK}_1 := c_{1+R_n},
$$
and for all $0 \leq k < \hK$ and $\hM_+ \leq i \leq R_n$, we have
$
\cV^n_i(q^k_i, q^{k+1}_i) \subset \hQ
$
for some $Q \in\cC_0$.
\end{prop}

\begin{proof}
Let $Z$ be a non-trivial component in the finite set $\cZ$ given in \propref{cover long comp}, so that
$$
Z([t_{i-1}, t_i]) \subset Q_i \in \cC_0
\matsp{for}
1\leq i \leq l.
$$
Observe that if $\Gamma_1 \ni Z(t_1)$ is a sufficiently vertical proper Jordan arc in $Q_1$, and $\tiQ_1$ is the component of $Q_1 \setminus \Gamma_1$ containing $Z([t_0, t_1))$, then
$$
\Lambda \cap \tiQ_1 = \{p \in \Lambda_{Q_1} \; | \; p \leq_{Q_1} Z(t_1)\}.
$$
Moreover, the analogous property holds at $Z(t_{l-1})$ in $\Lambda_{Q_l}$.

Let
$$
\{z^1, \ldots, z^{\hK-1}\} \subset \bigcup_{Z \in \cZ} Z
$$
be the finite set of points such that each point $z^j$ with $1\leq j < \hK$ is a point analogous to $Z(t_i)$ with $1\leq i < l$ discussed above for some component $\tiZ \in \cZ$. By \thmref{order on piece}, we can define
$
\{q^1_1, \ldots, q^{\hK-1}_1\} \subset \Lambda^n_1
$
as the set of points ordered with respect to $<_{\cT_{c_1}}$ such that for $1\leq j < \hK$, there exist unique $1 \leq k < \hK$ and $1\leq m \leq R_n$ such that
$
z^j = q^k_m.
$
By \propref{tunnel escape at entry}, this means that
$$
\cd(q^k_1) = \vd(q^k_{R_n}) = \varnothing
\matsp{for}
1 \leq k < \hK.
$$

For $0\leq k < \hK$, denote
$$
q_1^- := q_1^k,
\hspace{5mm}
q_1^+ := q_1^{k+1}
\matsp{and}
\ticV^n_1 := \cV^n_1(q_1^-, q_1^+).
$$
By \thmref{order on piece}, we see that for all $1 \leq i < R_n$, there exists $Q^i \in \cC_0$ such that we have $q_i^\pm \in Q^i$. 

Let
$
0 =: m_0 < m_1 < \ldots < m_I := R_n
$
be the increasing sequence of all integers such that for $0 \leq i \leq I$, we have
$
q_{m_i}^-, q_{m_i}^+ \in \cT_{c_0}.
$
Further denote
$$
k_i := m_i - \hM_-
\matsp{and}
l_i := m_i + \hM_+.
$$
Then for $l_i\leq m \leq k_{i+1}$, there exists $p^m \in \cP$ such that
$
q_m^-, q_m^+ \in \bcB_{p^m}.
$

We claim that for all $0\leq i < l$ and $k_{i+1} \leq m \leq l_i$, the following holds. Let $q_1 \in \ticV^n_1$. Then $q_m \in \cB_{p^m}$. Moreover, if $q_m$ is not contained in a thickened valuable crescent of the form $\hcT_{m - m_j}$ for some $j \leq i$ (which contains either $q_m^-$ or $q_m^+$), then $E^{h,n}_{q_m}$ is sufficiently horizontal in $\cB_{p^m}$.

Proceeding by induction, suppose the claim holds for all $m < k_i$. Let $1<  j < i$ be an integer such that
$$
\vd^{m_j}(q_{m_j}^-) = \vd^{m_j}(q_{m_j}^+)= 0.
$$
Then by \propref{tunnel inv}, if $-k^\pm <0$ is the the tunnel escape moment for $q_{m_j}^\pm$, then
$$
\vd^m(q_m^\pm)= 0
\matsp{for}
m_j-k^\pm \leq m \leq m_j.
$$
By the induction hypothesis and \lemref{trans crit curv} ii), we see that $\ticV^n_{m_j} \subset \hcT_{c_0}$.

Suppose that for some $1 < j < i$, we have
$
\ticV^n_{m_j} \subset \hcT_{c_0}.
$
Let $-k^\pm < 0$ and $l^\pm >1$ be the tunnel and crescent escape moments for $q_{m_j}^\pm$ and $q_{m_j+1}^\pm$ respectively. Then by \propref{nice cover}, for $m_j-k^\pm\leq  m \leq m_j+l^\pm$, we have $\cT_{c_{m-m_j}} \subset Q^m$.

By these observations, we can conclude using the induction hypothesis and \lemref{trans crit curv} ii) that
$
\ticV^n_{m_j} \subset \hcT_{c_0}
\matsp{for all}
1 < j \leq i.
$
Let $q_{m_i} \in \ticV^n_{m_i} \subset \hcT_{c_0}$. Denote by $-k <0$ and $l >1$ the tunnel and crescent escape time for $q_{m_i}$ and $q_{m_i+1}$. Since $m_i-k < k_i$, by the induction hypothesis and \lemref{trans crit curv} iii), we see that $q_{m_i+l} \in \cB_{p^{m_i+l}}$, and $E^{v,n}_{q_{m_i+l}}$ is sufficiently horizontal in $\cB_{p^{m_i+l}}$.

Lastly, observe that if the induction hypothesis holds for all $m <  m'$ with $l_i\leq m' < k_{i+1}$, then by \propref{pres ver hor curv} i), and \eqref{eq:no stray} applied to $W^{v,n}(q)$ for all $q \in \ticV^n_m$ that is not contained in a thickened critical tunnel of the form $\hcT_{c_{-m_j+m}}$ with $i < j \leq I$ (which contains either $q_m^-$ or $q_m^+$), the induction hypothesis must also hold for $m$. The result follows.
\end{proof}

Let $p_1 \in \Lambda^n_1$. By \lemref{entry for reg lem}, $p_1$ is $R_n$-times forward $(1, \delta)_v$-regular along $E^{v,n}_{p_1}$. Let
$$
\{\tiPhi_{p_m} : \ticU_{p_m} \to \tiU_{p_m}\}_{m=1}^{R_n}
$$
be a linearization of the $(R_n-1)$-times forward orbit of $p_1$ with vertical direction $E^{v,n}_{p_1}$. Note that
$$
F^{R_n-1}(\ticU_{p_1}^{\bdelta R_n}) \subset \ticU_{p_{R_n}} \subset \cU_{c_0}.
$$
For $q_1 \in \ticU_{p_1}^{\bdelta R_n}$, write
$
z_{R_n} = (x_{R_n}, y_{R_n}) := \Phi_{c_0}(q_{R_n}),
$
and let
$$
W^v(q_{R_n}) := \Phi_{c_0}^{-1}(\bbL^v_{x_{R_n}} \cap U_{c_0}).
$$
Moreover, for $1\leq i \leq R_n$, define
$$
W^{v,n}(q_i) := F^{i-R_n}(W^v(q_{R_n})).
$$
By \propref{loc linear} and \ref{proper ver curv}, we see that $W^{v,n}(q_1)$ is vertically proper in $\cT_{c_1}$ and sufficiently stable-aligned in $\cU_{c_1}$.

Choose $\tic_1^n \in \hcU_{c_1}^{\bdelta R_n}$ and $\tic_{1+R_n}^n \in \hcU_{c_{1+R_n}}^{\bdelta R_n}$ such that
\begin{equation}\label{eq:defn vn1}
\cV^n_1 := \cV^n_1(\tic_1^n, \tic_{1+R_n}^n) \supset \{p \in \cU_{c_1} \; | \; \dist(p, \cV^n_1(c_1, c_{1+R_n})) < \lambda^{\bdelta R_n}\}.
\end{equation}
Denote
$
V^n_1 := \Phi_{c_1}(\cV^n_1),
$
and for $1 \leq i \leq R_n$,
$
\cV^n_i := F^{i-1}(\cV^n_1).
$

\begin{lem}\label{adoption 1}
Let $q_0 \in \hcT_{c_0}$. Let $-k <0$ and $l > 1$ be the tunnel and crescent escape moment for $q_0$ and $q_1$ respectively. Suppose that $q_l$ is $N$-times forward $(O(L), \delta)_v$-regular along  $\tiE^v_{q_l} \in \bbP^2_{q_l}$ for some
$
N> -\logl \bL.
$
Then $q_{-k}$ is $(N+k+l)$-times forward $(O(1), \delta)_v$-regular along $\tiE^v_{q_{-k}}$.
\end{lem}

\begin{proof}
The result follows immediately from \propref{vert angle shrink}, \ref{stable center align} and \ref{tunnel escape}.
\end{proof}

\begin{lem}\label{adoption 2}
Let $p_0 \in \hcT_{c_0}$ and $\tiE^v_{p_0} \in \bbP^2_{p_{-j}}$. Let $-k_0 <0$ be the tunnel escape moment for $p_0$, and let $-k_1 \leq -k_0$ be a Pliss moment in the backward orbit of $c_0$. Suppose $p_{-j}$ is $j$-times forward $(O(L), \delta)_v$-regular along  $\tiE^v_{p_{-j}}$ and $c_{-j} \in \cB_{p_{-j}}$ for all $k_0< j< k_1$. If $p_{-k_0}$ is $k_0$-times forward $(O(1), \delta)_v$-regular along $\tiE^v_{p_{-k_0}}$, then $p_{-k_1}$ is $k_1$-times forward $(O(1), \delta)_v$-regular along $\tiE^v_{p_{-k_1}}$.
\end{lem}

\begin{proof}
The result follows immediately from \propref{vert angle shrink}.
\end{proof}

\begin{prop}\label{entry reg}
The statement of \lemref{entry for reg lem} generalizes to $q_0 \in \cV^n_{R_n}$.

Similarly, the statement of \lemref{entry back reg lem} generalizes to $q_1 \in \cV^n_1$.
\end{prop}

\begin{proof}
By \propref{v fit}, there exist $q_0^-, q_0^+ \in \Lambda^n_0$ such that the following holds for all $0\leq i < R_n$. Let
$$
\ticV^n_{-i} := \cV^n_{R_n-i}(q_{-i}^-, q_{-i}^+) \ni q_{-i}.
$$
Then we have
$
\ticV^n_{-i} \subset \hQ^{-i}
$
for some $Q^{-i} \in \cC_0$.

Let
$
0 =: m_0 < m_1 < \ldots < m_I < R_n
$
be the increasing set of all moments such that for $0\leq i \leq I$, we have
$
q_{-m_i} \in \hcT_{c_0}
$ and $
\cd^{m_i}(q_{-m_i}) = 0.
$
Let $(-k_i +m_i) < 0$ and $(l_i-m_i) > 1$ be the tunnel and crescent escape moments for $q_{-m_i}$ and $q_{-m_i+1}$ respectively.

Proceeding by induction, suppose that the claim is true for $m < m_i$. By \thmref{order on piece} and \propref{v fit}, there exists $p_{-k_i} \in \Lambda_{Q^{-k_i}}$ such that $p_{-m} \in Q^{-m}$ for $l_i \leq m \leq k_i$, and
$
\cd^{k_i}(p_{-k_i}) = \varnothing.
$
For simplicity, assume that
$
p_{-k_i} = q_{-k_i}^+.
$
In general,
$
p_{-k_i} = c_{-d}
$
where $d = k_i - m_j$ for some $j < i$, and $\cT_{c_{-d}}$ contains either $q_{-k_i}^-$ or $q_{-k_i}^+$. The argument in this case requires some minor technical adjustments which we will omit.

By the induction hypothesis and \lemref{adoption 1}, $q_{-k_i}$ is $k_i$-times forward $(O(1), \delta)_v$-regular along $E^{v,n}_{q_{-k_i}}$. By \lemref{adoption 2}, $q_{-k_i}^+$ is also $k_i$-times forward $(O(1), \delta)_v$-regular along $E^{v,n}_{q_{-k_i}^+}$. Moreover, by \lemref{vert angle shrink}, we see that $E^{v,n}_{q_{-k_i}}$ is sufficiently vertical in a regular box $\cB_{q_{-k_i}^+}$ at $q_{-k_i}^+$.

By \lemref{entry for reg lem}, for $k_i < m < m_{i+1}$, the point $q_{-m}^+$ is $m$-times $(O(L), \delta)_v$-regular along $E^{v,n}_{q_{-m}^+}$. Since $q_{-m} \in \cB_{q_{-m}^+}$, we conclude by  \propref{loc linear} that $q_{-m}$ is $m$-times $(O(L), \delta)_v$-regular along $E^{v,n}_{q_{-m}}$. The first part of the result follows.

The proof of the second part is nearly identical to that of \thmref{ref uni} b), and will be omitted.
\end{proof}

\begin{proof}[Proof of \thmref{val chart}]
By \propref{entry reg}, the set $\cV^n_1$ can be covered by regular charts obtained by linearizing the $(R_n-1)$-time forward orbit of any point $p_1 \in \cV^n_1$ with vertical direction $E^{v,n}_{p_1}$. The desired diffeomorphism $\bPsi^n_{c_1} : \cV^n_1 \to V^n_1$ is produced by gluing these charts. It remains to prove that
$
F^{R_n}(\cV^n_1) \Subset \cV^n_1.
$

Write
$$
(x, y) := \Phi_{c_1}(c_{1+R_n}),
\hspace{5mm}
(a, b) := \Phi_{c_0}(c_{R_n})
\matsp{and}
(u, v) := \Phi_{c_0}(c_{2R_n}).
$$
Observe that
$$
\Phi_{c_1}(\cV^n_1) \supset (-\lambda^{\bdelta R_n}, x + \lambda^{\bdelta R_n}) \times \bbI(l_{c_1}).
$$
By \propref{entry reg}, we also see that
$$
\Phi_{c_0}(\cV^n_{R_n}) \subset (-\lambda^{\bdelta R_n} + a, u+\lambda^{\bdelta R_n}) \times \bbI(\lambda^{(1-\bdelta)R_n}).
$$
Since for
$$
(u', v') := \Phi_{c_1} \circ F \circ \Phi_{c_0}^{-1}(u, v),
$$
we have $0< u' < x$, the result now follows from \thmref{unif reg crit}.
\end{proof}


\appendix

\section{Regular Unicriticality in 1D}\label{sec:reg unicrit 1d}

Let $f : I \to I$ be a $C^2$-unimodal map on an interval $I \subset \bbR$ with the critical point at $0 \in I$. Suppose that $f$ is infinitely renormalizable. That is, there exist a nested sequence $I =: I^0_1 \Supset I^1_1 \Supset \ldots$ of intervals and a sequence of natural numbers $1 =: R_0 < R_1 < \ldots$ such that $I^n_1$ is $R_n$-periodic and
$$
r_{n-1} := R_n/R_{n-1} \geq 2
\matsp{for}
n \in \bbN.
$$
We assume that $f$ is {\it of bounded type}: there exists a constant $K >1$ such that $r_n < K$ for all $n \in \bbN$.

Let
$$
I^n_0 := f^{-1}(I^n_1) \ni 0
\matsp{and}
I^n_i := f^i(I^n_0)
\matsp{for}
0 \leq i < R_n.
$$
The renormalization limit set of $f$ is given by
$$
\Lambda := \bigcap_{n=1}^\infty \bigcup_{i=0}^{R_n-1} f^i(I^n_1).
$$
Denote the orbit of $c_0 := 0$ by
$$
c_{-n} := (f^n|_\Lambda)^{-1}(0)
\matsp{and}
c_n := f^n(0)
\matsp{for}
n \in \bbN.
$$

For $x \in \Lambda \setminus \{c_0\}$, define $\depth_c(x)$ as the largest integer $d \geq 0$ such that
$
x \in I^d_0 \setminus I^{d+1}_0.
$
Also define $\entry_d(x)$ as the integer $0 < s < r_d$ such that
$
f^{sR_d}(x) \in I^{d+1}_0.
$

\begin{thm}[Real {\it a priori} bounds]\label{1d a priori}
There exist uniform constants $\sigma \in (0,1)$ and $\nu >0$ such that the following statements hold for all $n \in \bbN$.
\begin{enumerate}[i)]
\item For
$
x \in I^n_0 \setminus I^{n+1}_0,
$
let
$
s := \entry_n(x).
$
Then
$$
|(f^{iR_n})'(x)| > \nu
\matsp{for}
0 \leq i \leq s.
$$
\item We have
$$
|(f^{iR_n}|_{I^n_1})'| > \nu
\matsp{for}
0 \leq i < r_n.
$$
\end{enumerate}
\end{thm}

Fix $\epsilon_0 >0$ sufficiently small. For $\epsilon \in (0, \epsilon_0)$, let
$
\bepsilon := \epsilon^\alpha
$ and $
\uepsilon := \epsilon^{1/\alpha}
$
for some universal constant $\alpha \in (0,1)$. For $t > 0$ and $n \geq 0$, denote
$$
D^t_{-n} :=[c_{-n}- te^{-\uepsilon n}, c_{-n} + te^{-\uepsilon n}].
$$

\begin{defn}
We say that $f$ is {\it $(\epsilon, \uepsilon)$-regularly unicritical on $\Lambda$} if for any $t >0$, there exists $L = L(t)\geq 1$ such that the following condition holds. If
$$
x \in \Lambda \setminus \bigcup_{i=0}^{N-1} D^t_{-i}
$$
for some $N \in \bbN$, then
$$
|(f^n)'(x)| > L e^{-\epsilon n}
\matsp{for}
1 \leq n \leq N.
$$
\end{defn}

\begin{thm}\label{reg unicrit 1d}
For any $\epsilon \in (0, \epsilon_0)$, the infinitely renormalizable unimodal map $f$ is $(\epsilon, \uepsilon)$-regularly unicritical on its renormalization limit set $\Lambda$.
\end{thm}

To prove \thmref{reg unicrit 1d}, we first need the following lemma.

\begin{lem}\label{approach}
For $t > 0$ and $n \in \bbN$, let
$$
C^t_n(0) := [-te^{-\uepsilon R_n}, te^{-\uepsilon R_n}] \subset I^n_0
\matsp{and}
C^t_n := \bigcup_{i=0}^{R_n-1} (f^i|_{I^n_{R_n-i}})^{-1}(C^t_n(0)).
$$
Then
$$
C^t_n \subset \bigcup_{i=0}^{R_n-1}D^t_{-i}.
$$
\end{lem}

\begin{proof}
Write
$$
i = s_{n-1}R_{n-1} + \ldots + s_0 R_0,
$$
where $0 \leq s_j < r_j$ for $0\leq j <n$. By \thmref{1d a priori} i), for any $J \subset I^{j+1}_{R_{j+1}}$, we have
$$
\left|f^{-s_jR_j}|_{I^{j+1}_{R_{j+1}}}(J)\right| < \nu^{-1} |J|.
$$
It follows that for
$$
C^t_n(i) :=  (f^i|_{I^n_{R_n-i}})^{-1}(C^t_n(0)),
$$
we have
$$
|C^t_n(i)| < \nu^{-n}|C^t_n(0)|.
$$
\end{proof}

For $x \in \Lambda \setminus \{c_1\}$, define $\depth_v(x)$ as the largest integer $d \geq 0$ such that
$
x \in I^d_1 \setminus I^{d+1}_1.
$

\begin{proof}[Proof of \thmref{reg unicrit 1d}]
Let
$$
x_0 \in \Lambda \setminus \bigcup_{i=0}^{N-1} D^t_{-i}
$$
for some $N \in \bbN$. For $0 \leq n \leq N$, let $n = m + k+1$, where
$$
d_c := \depth_c(x_m) = \max_{0 \leq i \leq n} \depth_c(x_i).
$$
Let
$$
d_v := \depth_v(x_{m+1}) = \max_{0\leq i \leq n} \depth_v(x_i).
$$
Decompose $m$ as
$$
m = s_0 R_0 + \ldots + s_{d_c-1}R_{d_c-1},
$$
with $0\leq s_l < r_l$ for $0\leq l <d_c$. Similarly, we have
$$
k = t_0 R_0 + \ldots + t_{d_v-1}R_{d_v-1},
$$
with $0\leq t_l < r_l$ for $0\leq l <d_v$. 

By \thmref{1d a priori}, we have
$$
|(f^m)'(x_0)| > \nu^{d_c}
\matsp{and}
|(f^k)'(x_{m+1})| > \nu^{d_v}.
$$
Moreover,
$$
|f'(x_m)| \asymp |x_m| > te^{-\bepsilon m}
$$
by \lemref{approach}. Since
$
d_c, d_v < \log n,
$
the result follows.
\end{proof}


\section{Pliss Lemma}

Let $\alpha_1 < \alpha_2 < \alpha_3$. For some $N \in \bbN\cup\{ \infty\}$, consider a sequence $\{a_n\}_{n=1}^N \subset (\alpha_1, \infty)$. Suppose either:
\begin{equation}\label{pliss1}
\frac{1}{i}(a_1 + a_2 + \ldots + a_i) \leq \alpha_2
\matsp{for}
1\leq i \leq N,
\end{equation}
or more weakly (assuming $N = \infty$):
\begin{equation}\label{pliss2}
\limsup_{i \to \infty}\frac{1}{i}(a_1 + a_2 + \ldots + a_i) \leq \alpha_2.
\end{equation}
An integer $1 \leq k \leq N$ is a {\it direction-preserving Pliss moment} if
$$
\frac{1}{i+1}(a_k +a_{k+1} \ldots + a_{k + i}) \leq \alpha_3
\matsp{for all}
0 \leq i \leq N - k.
$$
Similarly, $1 \leq m \leq N$ is a {\it direction-reversing Pliss moment} if either $m=1$, or $m > 1$ and
$$
\frac{1}{i}(a_{m -i} + \ldots +a_{m-2}+ a_{m-1}) \leq \alpha_3
\matsp{for all}
1 \leq i < m.
$$
Lastly, $1\leq n \leq N$ is an {\it absolute Pliss moment} if it is both a direction-preserving and direction-reversing Pliss moment.

Let $\{k_i\}, \{m_i\}, \{n_i\} \subset \{1, \ldots, N\}$ be the increasing sequences of all direction-preserving, direction-reversing and absolute Pliss moments respectively.

\begin{lem}\label{plisslem}
Let $\{j_i\} = \{k_i\}$ or $\{m_i\}$. If \eqref{pliss1} holds, then we have
$$
\frac{i}{j_i} \geq \frac{\alpha_2 - \alpha_3}{\alpha_1 - \alpha_3}
\matsp{for all}
i \geq 1.
$$
If \eqref{pliss2} holds, then we have
$$
\liminf_{i\to \infty}\frac{i}{j_i} \geq \frac{\alpha_2 - \alpha_3}{\alpha_1 - \alpha_3}.
$$
\end{lem}

\begin{lem}\label{abspliss}
Let $\{n_i\}\subset \{1, \ldots, N\}$ be the increasing sequence of all absolute Pliss moments. If \eqref{pliss1} holds, then
$$
\frac{i+2}{n_i} \geq \frac{2\alpha_2-\alpha_1 - \alpha_3}{\alpha_1 - \alpha_3}.
$$
If \eqref{pliss2} holds, then
$$
\liminf_{i\to \infty}\frac{i}{n_i} \geq \frac{2\alpha_2-\alpha_1 - \alpha_3}{\alpha_1 - \alpha_3}.
$$
\end{lem}

\begin{rem}
By replacing $\alpha_1, \alpha_2, \alpha_3$ and $a_n$'s by their negative values, we can obtain the same lower density bounds as in \lemref{plisslem} and \lemref{abspliss} for sequences that have lower bounds on their average sums rather than upper bounds.
\end{rem}


\bigskip

\begin{tabular}{l l l}
\emph{Sylvain Crovisier} &&
\emph{Mikhail Lyubich}\\

Laboratoire de Math\'ematiques d'Orsay &&
Institute for Mathematical Sciences\\

CNRS - Univ. Paris-Saclay &&
Stony Brook University\\

Orsay, France &&
Stony Brook, NY, USA\\

&&\\

\emph{Enrique Pujals} &&
\emph{Jonguk Yang}\\

Graduate Center - CUNY &&
Institut für Mathematik\\

New York, USA &&
Universität Zürich\\

&& Zürich, Switzerland

\end{tabular}

\end{document}